\documentclass[12pt]{amsart}
\usepackage{epsfig}
\usepackage{amssymb, amsmath}
\usepackage{amsfonts,graphics,epsfig}
\usepackage{amsmath, amsthm, amsfonts}
\usepackage{mathrsfs}
\usepackage{pb-diagram, pb-xy}
\usepackage[all]{xy}
\usepackage{comment}
\usepackage{tikz}
\usepackage{setspace}
\usepackage[pagewise]{lineno}
\usepackage[left=1.5in,right=1.5in,top=1.5in,bottom=1.5in]{geometry}
\usepackage{color}   %May be necessary if you want to color links
\usepackage{hyperref}
\hypersetup{
    colorlinks=true, %set true if you want colored links
    linktoc=all,     %set to all if you want both sections and subsections linked
    linkcolor=red,%choose some color if you want links to stand out
        citecolor=blue,  
}
\usepackage{tcolorbox}
\usepackage{enumitem}
\usepackage{soul}

\newcommand{\mbR}{\mathbb{R}}
\newcommand{\mbC}{\mathbb{C}}
\newcommand{\mbZ}{\mathbb{Z}}
\newcommand{\mbQ}{\mathbb{Q}}

\def\mbP{\mathbb{P}}

\newcommand{\<}{\le}
\def\>{\ge}
\def\eps{\epsilon}

\def\vphi{\varphi}

\def\subset{\subseteq}

\newcommand{\lrd}{\lfloor}
\newcommand{\rrd}{\rfloor}

\newcommand{\num}{\equiv}

\newcommand{\bir}{\dashrightarrow}
\newcommand{\D}{\Delta}

\def\mcO{\mathcal{O}}

\def\mcE{\mathcal{E}}
\def\mcF{\mathcal{F}}
\def\mcG{\mathcal{G}}

\def\mcL{\mathcal{L}}
\def\mcM{\mathcal{M}}
\def\mcN{\mathcal{N}}

\def\injective{\hookrightarrow}

\newtheorem{theorem}{Theorem}[section]
\newtheorem{lemma}[theorem]{Lemma}
\newtheorem{proposition}[theorem]{Proposition}
\newtheorem{corollary}[theorem]{Corollary}

\theoremstyle{remark}
\newtheorem{remark}[theorem]{Remark}
\theoremstyle{definition}
\newtheorem{definition}[theorem]{Definition}
\theoremstyle{definition}

\numberwithin{equation}{section}
\theoremstyle{definition}

\theoremstyle{definition}
\newtheorem{setup}[theorem]{Setup}

\def\deg{\operatorname{deg}}

\def\Supp{\operatorname{Supp}}
\def\dim{\operatorname{dim}}
\def\codim{\operatorname{codim}}

\def\Ex{\operatorname{Ex}}

\def\max{\operatorname{max}}

\def\hor{\operatorname{\hor}}
\def\ver{\operatorname{\ver}}

\def\Center{\operatorname{Center}}

\def\sm{\operatorname{\textsubscript{\rm sm}}}
\def\sing{\operatorname{\textsubscript{\rm sing}}}
\def\orb{\mathrm{orb}}
\def\ver{\operatorname{\textsubscript{\rm ver}}}
\def\hor{\operatorname{\textsubscript{\rm hor}}}

\def\red{\operatorname{\textsubscript{\rm red}}}

\usepackage{lipsum}

\author{Omprokash Das}
\address{School of Mathematics\\
Tata Institute of Fundamental Research\\
Homi Bhabha Road, Navy Nagar\\
Colaba, Mumbai 400005, India}
\email{omprokash@gmail.com}
\email{omdas@math.tifr.res.in}
\author{Wenhao Ou}
\address{Institute of Mathematics\\
 Academy of Mathematics and Systems Science\\
  Chinese Academy of Sciences\\
   Beijing 100190, CHINA}
\email{wenhaoou@amss.ac.cn}

\usepackage[french,english]{babel}

\date{}

\begin{document}
\title{On the Log Abundance for Compact {K\"ahler} threefolds II}
\maketitle

\begin{abstract}
    In this article we show that if $(X, \Delta)$ is a log canonical compact K\"ahler threefold pair such that $K_X+\Delta$ is nef of numerical dimension $\nu(X, K_X+\Delta)=2$, then $K_X+\Delta$ is semi-ample. This result combined with our previous work   shows that the log abundance holds for log canonical compact K\"ahler threefold pairs.
\end{abstract}

%\begin{abstract}
 %       In this article we show that if $(X, \Delta)$ is a log canonical compact K\"ahler threefold pair such that $K_X+\Delta$ is nef and the numerical dimension $\nu(X, K_X+\Delta)=2$, then $K_X+\Delta$ is semi-ample. This result combined with our previous work   shows that the log abundnace holds for log canonical compact K\"ahler threefold pairs.
        % $(X, \Delta)$.
%\end{abstract}

%Dans cet article, nous d\'emontrons que si $(X,D)$ est une paire log canonique telle que $X$ est une vari\'et\'e k\"ahlerienne  de dimension trois  et que $K_X+\Delta$ est nef  de dimension num\'erique $\nu(X, K_X+\Delta)=2$, alors $K_X+\Delta$ est semi-ample. Combinant avec notre r\'esultat pr\'ec\'edent, nous d\'emontrons le th\'eor\`eme d'abondance log canonique pour des vari\'et\'e k\"ahlerienne de dimension trois. 

\tableofcontents

\section{Introduction}

The Minimal Model Program (MMP) is an important tool in birational classifications of projective varieties $Z$. 
When $\dim Z = 3$, the MMP was fully established in the 90's, and one has the following classification. 
For a smooth projective threefold, there is a mildly singular birational model $Z\dashrightarrow X$, such that 
either there is a Fano fibration $X\to Y$, 
or the canonical class $K_{X}$ is nef. 
In the second case, the canonical class  $K_{X}$  induces a fibration $X\to Y$, thanks to the following theorem, known as the abundance theorem  (see \cite{Miy87, Miy88, Miy88b} and \cite{Kaw92}). 

\begin{theorem}
\label{thm:abundance-proj}
    Let $X$ be a projective threefold with terminal singularities. 
If $K_X$ is nef, then it is semi-ample. 
\end{theorem}

%If $X$ is not algebraic, then  the numerical dimension $\nu(X, K_X)$ is at most $2$. When $\nu(X, K_X)=0$,  it follows from the non-vanishing theorem of \cite[Theorem 0.3]{DP03} that $mK_X \cong \mcO_X$ for some $m>0$.  If $\nu(X, K_X)=1$, then one can adapt the method for projective varieties to conclude, see \cite[Section 8.A]{CHP16}.  If  $\nu(X, K_X)=2$, it is claimed in \cite[Theorem 8.2]{CHP16} that $K_X$ is semi-ample.  However, the authors have acknowledged a gap in the proof of \cite{CHP16}.  The issue arises in Step 1 of the proof of \cite[Theorem 8.2]{CHP16}, where a result of Kawamata, namely \cite[Theorem 9.6]{Kaw88} is incorrectly applied. A possible approach to the solution is suggested in \cite{CHP23}. \\

We note that a threefold may be singular throughout this paper. 
The MMP was extended to the category of compact K\"ahler threefolds in \cite{HP16} and \cite{HP15}. 
Particularly, we can recover the similar bimeromorphic classification. 
When $Z$ is a smooth compact K\"ahler threefold, there is a mildly singular bimeromorphic model $Z\dashrightarrow X$  such that 
either there is a Fano fibration $X\to Y$, or the canonical class $K_{X}$ is nef. 

%Particularly, when $Z$ is a non-uniruled smooth compact K\"ahler threefold, there is a mildly singular bimeromorphic model $Z\dashrightarrow X$  such that  $K_{X}$ is nef.

It is natural to expect Theorem \ref{thm:abundance-proj} holds for K\"ahler threefolds.  
Before proceeding further, we would  like to explain the outline of the complete proof of the abundance theorem for projective threefolds.  Assume that $X$ is a   projective threefolds with mild singularities, and that the canonical divisor $K_X$ is nef. If the Kodaira dimension of $X$ satisfies $\kappa(K_X)=3$, then the semi-ampleness of $K_X$ follows from the Kawamata-Shokurov Base-point free theorem. 
If $1\leq \kappa(K_X)\leq 2$, then we let $f:X\bir Z$ be the Iitaka fibration of $K_X$. 
In this case,  by an  argument of Kawamata,  
the semi-ampleness of $K_X$ follows from the MMP and the abundance theorem on general fibers of $f$, which in our case are either a curves or surfaces. 
We refer to the proofs of \cite[Theorem 7.3]{Kaw85} and \cite[Theorem 3.1]{DW19} for these lines of arguments.   
It remains to study the following two problems. 
The first one is the non-vanishing theorem, i.e. $H^0(X, mK_X)\neq 0$ for some $m\in\mbZ^+$. This  was settled by Miyaoka in \cite{Miy88b}.  
The last one is to deal with the situation when $\kappa(K_X)=0$.  
We discuss according to the numerical dimension $\nu(K_X)$. 
If $\nu(K_X)=3$, then it follows that $\kappa(K_X)=3$, 
which contradicts the assumption. 
If  $1\leq \nu(K_X)\leq 2$,   then
Miyaoka and Kawamata  showed that   $\kappa(K_X)\geq 1$ (see \cite{Miy88, Kaw92}, or  \cite[Chapter 13 and 14]{Kol92}).  
Thus we can only have  $\nu(K_X)=0$, and the abundance theorem was proved by Nakayama in this case, see   \cite[Corollary V.4.6]{Nak04}.

Apart from the non-vanishing theorem, the most difficult part of the  abundance theorem is to show  that the condition $\nu(K_X)=2$ implies that   $\kappa(K_X)\geq 1$. We will see below that this is even a bigger problem when $X$ is just assumed to be compact K\"ahler.

Now we assume that $X$ is a compact K\"ahler variety of dimension $3$ with mild singularities and the canonical bundle $K_X$ if nef. To prove the abundance conjecture on $X$, we will follow the same general strategy as  above. In this case the non-vanishing theorem is proved by Demailly and Peternell in \cite{DP03}. If $\kappa(K_X)=3$, then $X$ is projective by Moishezon, and we are done by the projective case. When $1\leq \kappa(K_X)\leq 2$, the semi-ampleness of $K_X$ is proved in \cite{CHP16}. The case when $0\leq \nu(K_X)\leq 1$ is  also proved in \cite{CHP16}.  The remaining case when $\nu(K_X)=2$ is what we will deal with in this article.   
We prove the following theorem.

\begin{theorem}
\label{thm:lc-log-abundance}
Let $(X,\Delta)$ be a lc pair such that $X$ is a compact K\"ahler threefold. 
Assume that $K_X+\Delta$ is nef of numerical dimension $2$. 
Then it is semi-ample.
\end{theorem}

We follow the same strategy as in \cite[Chapter 14]{Kol92}; we note that this is also the same strategy followed by \cite{CHP16, CHP23}. This method involves two technical tools:
\begin{enumerate}
    \item define an orbifold type second Chern class on compact K\"ahler threefolds with klt singularities, and prove a Bogomolov-Gieseker (BG) type inequality on it.
    \item compare orbifold type second Chern classes with  usual  second Chern classes (on a desingularization). 
\end{enumerate}
When $X$ is   projective, an orbifold type second Chern class is defined in Chapter 10 of \cite{Kol92} via $\mathbb{Q}$-sheaves, 
and the related BG type inequality on surfaces with quotient singularities is also proved in the same chapter. Then in dimension $3$, the BG type inequality is proved by taking a hyperplane section and reducing it to the case of  surfaces. 
In the  setting of general K\"ahler varieties,  these  techniques are not available and the difficulty of proving the BG type inequality is well known.

Once a BG type inequality is established, the key arguments of proving the abundance theorem from this inequality (in the projective case) is explained in \cite[Lemmas 14.3.1-14.3.3]{Kol92}. One crucial observation   is that, in  all  lemmas of \cite{Kol92}, the orbifold second Chern class is always represented by an 1-cycle in the Chow group, and this is possible because on a projective orbifold every orbifold coherent sheaf has a \emph{finite} locally free resolution, see Definition 10.4, Lemma 10.5 and Definition 10.6 of \cite{Kol92}. 
Note that this is no longer true on compact analytic orbifolds.  
In fact, there are compact K\"ahler manifolds on which some coherent sheaf doesn't have a finite locally free resolution, see \cite{Voi02}. 
Without the cycle representation of the orbifold second Chern class, 
we need new tools to  compare it to the usual second Chern class on a desingularization of the orbifold, which greatly hinders our ability to adopt the arguments of the proof of \cite[Lemmas 14.3.1-14.3.3]{Kol92}. 
We dedicate Section \ref{section:compare-Chern} of our article dealing with this problem.  
Additionally, we explain the arguments of \cite[Chapter 14]{Kol92} in details with our new tools in  Section \ref{section:Euler-char}.

Here we would like to emphasize that, in \cite{CHP23}, the authors claim that a proof of the BG type inequality on  compact K\"ahler threefolds with klt singularities immediately implies the abundance theorem for the case of $\nu(K_X)=2$. However, we could not follow their argument. 
We believe that  the tools to compare the two types of Chern classes as in \cite[Lemmas 14.3.1-14.3.3]{Kol92}  are indispensable for the complete proof of abundance theorem.

%\begin{remark}
We also note that the conjectural  solution suggested in \cite{CHP23} is to first prove a BG type inequality on a compact K\"ahler threefold with klt singularities. 
However, we take a different approach here.   
For the specific compact  K\"ahler threefold $X$ with klt singularities we work with,   we first construct a K\"ahler orbifold $Y$ and a bimeromorphic morphism $\pi:Y\to X$ with nice properties (see Lemma \ref{lemma:construction-of-Y-1}), and then we use Faulk's orbifold BG type inequality  (see \cite{Faulk2022}) on $Y$.

Inspired by the techniques developed in this article, seven months after,  
the second author proved the BG type inequality on klt compact K\"ahler varieties of arbitrary dimensions in \cite{Ou2024}.  
Another four months after, Guenancia and P{\v a}un proved the BG type inequality on klt compact K\"ahler threefolds by a different method in \cite{GP24}. 
Nevertheless, as mentioned above, we believe our article is still relevant, for its new tools on comparison of Chern classes.

%Afterward, we mirror the arguments of \cite[Chapter 14]{Kol92}. 
%The crux of our techniques lies in Section \ref{section:Euler-char}, and the necessary cohomological tools are developed in Section \ref{section:compare-Chern}. 
%\end{remark}

Combining the main theorem of \cite{DasOu2022} with the theorem above, we complete the log abundance for log canonical K\"ahler threefolds. 

\begin{corollary}
\label{cor:log-abundance}
Let $(X,\Delta)$ be a lc pair such that $X$ is a compact K\"ahler threefold. 
If $K_X+\Delta$ is nef, 
then it is semi-ample.
\end{corollary}

As an application of the abundance theorem, 
we have the following corollary, as shown in \cite[Theorem 1.2]{CHP16}.
It is one of the steps towards the three dimensional  Kodaira problem in \cite{Lin2017}. 

\begin{corollary}
\label{cor:simple-3fold}
Let $X$ be a simple compact K\"ahler threefold with klt singularities. 
Then $X$ is bimeromorphic to a quotient $Y/G$, where $Y$ is a complex torus and $G$ is a finite group acting on $Y$.
\end{corollary}

% For the proof of Theorem \ref{thm:lc-log-abundance}, we follow the method for projective threefolds in \cite{Kol92}. 
% A main difficulty in the K\"ahler case is that Bogomolov-Gieseker type inequalities are not known for compact K\"ahler varieties with klt singulariites. 
% Our solution is to construct a bimeromorphic model which has quotient singularities only, and then apply Bogomolov-Gieseker type inequalities for compact K\"ahler orbifolds due to \cite{Faulk2022}.  

This article is organized in the following manner. 
After reviewing some preliminaries in Section \ref{section:pre}, we will recall the notion of complex orbifolds and prove an orbifold Miyaoka inequality in   Section \ref{section:chern}.  
In Section \ref{section:compare-Chern}, we introduce relative Chern classes, which
compare orbifold Chern classes and usual Chern classes on a desingularization. 
We also show that relative second Chern classes can be computed by reducing to surfaces. 
In Section  \ref{section:log-cotan}, we prove some positivity results on second Chern classes of cotangent sheaves. 
With these in hand, we complete the proof of the main theorem in the last three sections. \\

{\bf Acknowledgments} We would like thank 
Junyan Cao, 
Mitchell Faulk, 
Patrick Graf, 
Christopher Hacon,
Andreas H\"oring, 
Chiu-Chu Melissa Liu, 
Mihai P\u{a}un,  
Burt Totaro 
and Zheng Xu 
for important and valuable communications.  
We are grateful to the referee for useful suggestions, in particular, for the elegant proof of Lemma \ref{lemma:A-bounded-above-v2}. 
W. Ou was supported by the National Key R\&D Program of China (No. 2021YFA1002300). 
O. Das was partially supported by the Start--Up Research Grant(SRG), Grant No. \# SRG/2020/000348 of the Science and Engineering Research Board (SERB), Govt. Of India, during the preparation of this article.

\section{Preliminaries} 
\label{section:pre}

We will fix some notation and recall some elementary results in this section.  
% Throughout the present paper, we will denote by $\mbZ^+$ the set of positive integers. 

\subsection{Complex analytic varieties}

A \emph{complex analytic variety} 
%or an \emph{analytic variety} or simply a \emph{variety}  
is a reduced and irreducible  complex space. 
For a complex analytic variety $X$, we will denote by $X_{\sm}$ its smooth locus and by $X_{\sing}$ its singular locus. 
We also call a  smooth complex analytic variety a complex manifold.  
A pair $(X,\Delta)$ consists of a complex analytic variety $X$ and a divisor $\Delta$ with rational coefficients.

When $X$ is a %compact 
normal complex analytic variety and $B$ a reduced Weil divisor on $X$, we define  $\Omega^{[1]}_X(\log\, B)$, the reflexive logarithmic cotangent sheaf as follows. 
Let $r\colon Y\to X$ be a log resolution of $(X, B)$ and let $\Gamma=r_*^{-1}B$.  
Then the logarithmic cotangent bundle $\Omega_Y^{1}(\log\, \Gamma)$ is well defined. 
We define  $\Omega^{[1]}_X(\log B) := (r_* \Omega_Y^{1}(\log\, \Gamma) )^{**}$. 
We note that such a definition is independent of the resolution $Y$.

A normal complex analytic variety $X$ is called $\mathbb{Q}$-factorial if every Weil divisor is $\mathbb{Q}$-Cartier, and  $(\omega_X^{\otimes m})^{**}$ is invertible for some integer $m>0$, 
where $\omega_X:= (\bigwedge^{\dim X}\Omega_X ^1)^{**}$ is the canonical sheaf.   
To mimic the situation of algebraic varieties, 
we regard $\omega_X$ as  a $\mathbb{Q}$-Cartier divisor and denote it by $K_X$. 
Then we can consider the sum  $a K_X+\Delta$ for any  $\mathbb{Q}$-Weil divisor  $\Delta$ and any rational number $a$.    
For an integer $l$,  
we say that $l(aK_X+\Delta)$ is Cartier if $la$ is an integer,  $l\Delta$ is integral and the reflexive sheaf 
\[ \mathcal{O}_X(laK_X+l\Delta) := (\omega_X^{\otimes la} \otimes \mcO_X(l\Delta))^{**}\] 
is invertible. 
If $S$ is a subvariety in $X$, we use  the following convention for the restriction $(aK_X+\Delta)|_S$ of $aK_X+\Delta$ on $S$.  
For any $\mathbb{Q}$-Cartier $\mathbb{Q}$-Weil divisor $F$ on $S$, 
we write $(aK_X+\Delta)|_S \equiv  F$ 
(resp. $(aK_X+\Delta)|_S \sim_{\mathbb{Q}}  F$) 
if for some integer $l>0$,  both 
$l(aK_X+\Delta)$ and $lF$ are Cartier,   and $c_1(\mathcal{O}_X(laK_X+l\Delta)|_S) = c_1(\mathcal{O}_S(lF))$ (resp. $ \mathcal{O}_X(laK_X+l\Delta)|_S \cong \mathcal{O}_S(lF)$).    

Following   \cite{Grauert1962}, a complex analytic variety $X$ is  called \emph{K\"ahler}, 
if there exists a closed positive $(1, 1)$-form on $X$ such that the following holds: for every point $x\in X$,  there exists an open neighborhood $x\in U$,  a closed embedding $ U\to V$ into an open subset $V\subset\mbC^N$, and a strictly plurisubharmonic $\mathcal{C}^\infty$ function $f:V\to\mbR$ such that $\omega|_{U\cap X_{\sm}}=(\sqrt{-1}\partial\bar{\partial}f)|_{U\cap X_{\sm}}$.  
We also refer to \cite[Section 1]{Demailly1985} for general $(p,q)$-forms on analytic varieties.

For a normal compact K\"ahler variety $(X,\omega)$, a class in $H^2(X, \mathbb{R})$ is called \emph{nef} if it can be represented by a smooth form $\alpha$ with local potentials such that 
for every $\epsilon >0$, there exists a $\mathcal C ^\infty$ function $f_\epsilon$ such that $\alpha + \sqrt{-1}\partial \bar \partial f_\epsilon \ge -\epsilon \omega$.  
For more explanations on the notation, we refer to \cite[Definition 1.3]{DasOu2022} and the references therein.

Let $f:X\to Y$ be a proper surjective morphism of normal complex analytic varieties. 
A divisor $D$ is called  $f$-\textit{exceptional} if $\codim_Y f(\Supp(D))\>2$.
The following lemma is a consequence of Hironaka's flattening theorem.  
 
\begin{lemma}\label{lemma:p-q-exceptional}
Let $f:X\to Y$ be a proper surjective morphism between  compact complex normal analytic varieties. 
Then there exist projective bimeromorphic morphisms $g\colon Y'\to Y$ and $g'\colon X'\to X$ with the following commutative diagram, such that
\begin{equation*}
        \xymatrixcolsep{3pc}\xymatrixrowsep{3pc}\xymatrix{X'\ar[r]^{g'}\ar[d]_{f'} & X\ar[d]_f\\
        Y'\ar[r]^g & Y}
\end{equation*}
\begin{enumerate} 
	\item $Y'$ is smooth, and 
	\item for any proper bimeromorphic morphism $\varphi\colon  W\to X'$, every $(f'\circ\varphi)$-exceptional divisor  is also $(g'\circ\varphi)$-exceptional.
\end{enumerate}
\end{lemma}

\begin{proof}
 By   \cite[Corollary 1]{Hir75}, there is a  projective bimeromorphic morphism $Y_1\to Y$ such that the natural morphism $f_1\colon X_1 \to Y_1$ is flat, where $X_1$ is the main component of $X\times_Y Y_1$. 
In particular, $f_1$ is equidimensional. 
Let $Y' \to Y_1$ be a desingularization, and let $X'$ be the normalization of the main component of $X_1\times_{Y_1} Y'$. 
Then the induced morphism $f'\colon X'\to Y'$ is also equidimensional. Let $D\subset X'$ be a prime Weil divisor. Since $f'$ is equidimensional, either $f'(D)=Y'$ or $f'(D)$ is a prime Weil divisor on $Y'$. In particular, for any proper bimeromorphic morphism $\varphi\colon W\to X'$, every $(f'\circ\varphi)$-exceptional divisor is in fact $\varphi$-exceptional, and hence it is  $(g'\circ\varphi)$-exceptional. This completes our proof.
\end{proof}

\subsection{Index-one cover} 

%\begin{definition}\label{def:reflexive-Q-cartier}
Let $X$ be a normal variety and $\mcL$ a reflexive sheaf of rank $1$. Then $\mcL$ is called \textit{$\mbQ$-Cartier}, if there is a positive integer $m>0$ such that $(\mcL^{\otimes m})^{**}$ is an invertible sheaf. 
The smallest such $m$ is   called the \textit{Cartier index} of $\mcL$. 
 
%\end{definition}

Let $(o\in X)$ be a germ of normal complex analytic variety, and $\mcL$ a $\mbQ$-Cartier reflexive sheaf of rank 1 on $X$.  
Let $m>0$ be an integer such that $(\mcL^{\otimes m})^{**}$ is locally free. 
Shrinking $X$,    we  assume that there is an isomorphism $\varphi\colon (\mcL^{\otimes m})^{**} \cong \mcO_X$. 
Such an isomorphism induces a cyclic cover $f\colon Y \to X$,  
which is a  finite surjective morphism  of degree $m$ of normal varieties, 
such that $(f^*\mcL)^{**}\cong \mcO_Y$. 
Furthermore, $f$ is only branched over the locus where $\mathcal{L}$ is not locally free, and hence is \'etale in codimension one.  
When $m$ is the Cartier index of $\mcL$, the  morphism $f$ is called the \emph{index-one cover} of $\mcL$. 
For more explanations, we refer to \cite[Definition 2.52]{KM98}.   
In the next two lemmas, we prove some universal properties on cyclic covers.

\begin{lemma}
\label{lemma:unique-cyclic-cover} 
With the notation above, up to isomorphism,   $f$ is locally unique in the following sense. 
If $\varphi' \colon  (\mcL^{\otimes m})^{**} \cong \mcO_X $ is another isomorphism, and if $f'\colon Y' \to X$ the the cyclic cover induced by $\varphi'$, 
then up to shrinking $X$ around $o$, there is an isomorphism $Y\cong Y'$ over $X$.
\end{lemma}

\begin{proof} 
There is some nowhere vanishing holomorphic function $u$ on $X$ such that $\varphi' = u\cdot \varphi$. 
Shrinking $X$ if necessary, we may assume that $u$ has a $m$-th root $v$.  
We set $\mcL^{[i]} =  (\mcL^{\otimes i})^{**} $ for all integer $i$. 
Then from the construction of \cite[Definition 2.52]{KM98}, 
we see that $Y  = \mathrm{Spec}_X\, \bigoplus_{i \> 0 } \mcL^{[-i]}$ such that the $\mcO_X$-algebra structure on  $\bigoplus_{0\< i <m } \mcL^{[-i]}$ is given by 
\[
\mcL^{[-i]} \times \mcL^{[-j]} \overset{\varphi^{-1}}{\longrightarrow} \mcL^{[-i-j+m]}. 
\]
If we denote this algebra by $\mathcal{A}$, and the one corresponding to $Y'$ by $\mathcal{A}'$, then the multiplication by $v^{-1}$ on $\mcL$ induces an isomorphism of $\mcO_X$-algebras from $\mathcal{A}$ to $\mathcal{A}'$. 
This proves that $Y\cong Y'$. 
\end{proof}

\begin{lemma}
\label{lemma:cyclic-factor}
Let $\mcL$ be a reflexive sheaf of rank 1 on  a  pointed normal complex analytic variety $(o\in X)$. 
Let $m,n,d$ be positive  integers such that $n=dm$ and that $(\mcL^{\otimes m})^{**} \cong \mcO_X$.  
If $f\colon Y \to X$ and $g \colon Z\to X$ are the  cyclic covers defined as above, with of degree $m$ and $n$ respectively, 
then, by shrinking $X$ around $o$ if necessary, $Z$ is isomorphic to a disjoint union of copies of $Y$.  

In particular, every such kind of cyclic cover is locally isomorphic to a disjoint union of copies of the index-one cover of $\mcL$.
\end{lemma}

\begin{proof} 
We denote by $\varphi \colon (\mcL^{\otimes m})^{**} \to  \mcO_X$ and $\gamma \colon (\mcL^{\otimes n})^{**} \to \mcO_X$ 
the isomorphisms. 
Then there is a nowhere vanishing holomorphic function $u$ on $X$ such that 
$\gamma = u \cdot \varphi^d$. 
%, where $\varphi^d\colon (\mcL^{\otimes n})^{**} \to \mcO_X$ is the isomorphism induced by $\varphi$.   
If $u=1$, then we see that $Z$ is just the disjoint union of $d$ copies of $Y$.

In general, shrinking $X$ if necessary, we may assume that $u$ has a $d$-th root $v$.  
Then we consider the isomorphism $\varphi' = v\cdot  \varphi$. 
It induces a cyclic cover $Y' \to X$ which is isomorphic to $Y$ by Lemma \ref{lemma:unique-cyclic-cover}. 
We notice that $(\varphi')^d = \gamma$. 
Hence $Z$ is just a disjoint union of copies of $Y'$. 
This completes the proof of the lemma.  
\end{proof}

An integral Weil divisor $\D$ on $(o\in X)$ is   $\mathbb{Q}$-Cartier if and only if the reflexive sheaf $\mcO_X(\D)$ is $\mathbb{Q}$-Cartier. 
If it is the case, then the Cartier index of $\D$ is equal to the one of  $\mcO_X(\D)$.  
We define  the index-one cover of   $\D$ over $(o\in X)$   as the index-one cover $f\colon Y\to X$ of $\mcO_X(\D)$. 
Then  the integral divisor  $f^*\D$ is Cartier.

\subsection{MMP for complex analytic varieties}

Since the log canonical abundance is known for projective threefolds (see \cite{Kol92} and \cite{KMM94}), 
in this paper, we mainly focus on compact K\"ahler threefolds which are not algebraic. 
Particularly,  if $X$ is a uniruled non-algebraic compact K\"ahler threefold with $\mathbb{Q}$-factorial klt singularities, 
then the base of the MRC fibration has dimension $2$ by \cite[Lemma 1.19]{DasOu2022}. 
The following theorems follow  from \cite[Theorem 2.1 and Theorem 2.3]{DasOu2022}.

\begin{theorem}
    \label{thm-non-vanishing}
Let $(X, \Delta)$ be a $\mbQ$-factorial compact K\"ahler threefold dlt pair. 
Assume that $X$ is non-algebraic and that the base of the MRC fibration for $X$ has dimension $2$.  
Then the following assertions are equivalent:
 \begin{enumerate}
 \item $K_X+\D \sim_{\mathbb{Q}} D \> 0$ for some $\mathbb{Q}$-divisor $D$;
 \item $K_X+\D$ is pseudo-effective;
 \item $(K_X+\D)\cdot F\>0$ for a general fiber  $F$ of the MRC fibration of $X$.
 \end{enumerate} 
\end{theorem}

 \begin{theorem}
 \label{thm-non-vanishing-general-setting}
 Let $(X, \Delta)$ be a $\mbQ$-factorial compact K\"ahler threefold dlt pair. 
 If $K_X+\D$ is pseudo-effective, then it is $\mathbb{Q}$-effective.
 \end{theorem}

The   MMP for a three-dimensional compact K\"ahler  pair $(X, \D)$  was studied in \cite{HP16}, \cite{HP15}, \cite{CHP16}, \cite{DasOu2022},  \cite{DasHacon2020} and \cite{DHP22}. 
We summarize the following statement.  

\begin{theorem}
\label{thm:existence-dlt-MMP}
    Let $(X,\D)$ be a dlt pair such that $X$ is a $\mathbb{Q}$-factorial  compact K\"ahler threefold.  
    Then the cone theorem, the contraction theorem and the existence of flips  hold for $(X,\D)$. 
    Furthermore,  any $(K_X+\D)$-MMP   terminates. 
\end{theorem}

\begin{proof}
If $K_X+\D$ is pseudoeffective, then 
the cone theorem follows from \cite[Theorem 2.26]{DasOu2022}; 
if $K_X+\D$ is not pseudoeffective, then 
the cone theorem is established in \cite[Theorem 5.2]{DHP22}.   
The contraction theorem follows from \cite[Corollary 5.3 and Corollary 5.6]{DHP22}.  
For the existence of flips, we refer to \cite[Theorem 4.3]{CHP16}. 
Finally, for the termination of   flips, see \cite[Theorem 1.12]{DasOu2022}. 
\end{proof}

We can now introduce the following special MMP (see also \cite[Lemma 5.1]{KMM94} for projective threefolds).

\begin{lemma}
\label{lem:special-mmp}  
Let $(X,\D)$ be a dlt pair such that $X$ is a $\mathbb{Q}$-factorial  compact K\"ahler threefold.  
    Suppose that $K_X+\Delta$ is   not nef, 
    and that there is an effective  divisor $H$  such that $K_X+\Delta+H$ is nef. 
    Then  a $(K_X+\Delta)$-MMP, trivial with respect to $K_X+\Delta+H$, is a sequence of bimeromorphic maps
 \[
        (X, \Delta)=(X_0, \Delta_0)\bir  \cdots \bir (X_i, \Delta_i)\bir\cdots \bir (X_n, \Delta_n), 
 \]
 such that each $(X_i,\Delta_i)$ is dlt, and that  the following properties hold:
\begin{enumerate}
\item For each $0\<i\<n-1$, there is a $(K_{X_i}+\Delta_i)$-negative extremal ray $R_i$ such that $(K_{X_i}+\Delta_i+H_i)\cdot R_i=0$, and $X_i\bir X_{i+1}$ is the divisorial contraction or flip corresponding  to $R_i$.
\item If $K_X+\Delta+\lambda H$ is pseudoeffective for some $ \lambda<1$, then 
there is a  smallest integer $n\> 0$  such that $K_{X_n}+\Delta_n+(1-\epsilon)H_n$ is nef for all $\epsilon >0$ sufficiently small. 
\item If $K_X+\Delta+\lambda H$ is not pseudoeffective for any $\lambda<1$, then  
there is a  $(K_{X_n}+\Delta_n)$-Mori fibration $f\colon X_n\to Y$ such that $K_{X_n}+\Delta_n+H_n$ is $f$-relatively numerically trivial. 
\end{enumerate} 
\end{lemma}

\begin{proof}
    With the help of Theorem \ref{thm:existence-dlt-MMP}, the same argument of \cite[Theorem 1.26]{DasOu2022} shows that such a MMP always exists. 
\end{proof}

\subsection{Holomorphic foliations} \label{subsec:holomorphic-foliation}
Let $X$ be a normal complex analytic variety. 
A \emph{foliation} on $X$ is a saturated subsheaf $\mcF$ of the reflexive tangent sheaf $T_{X}$, which is closed under the Lie bracket on the smooth locus of $X$.  
For a  dominant  morphism  $f\colon Y\to X$  of normal complex analytic varieties, 
the pullback foliation   $f^{-1}\mcF$ of $\mcF$  is the foliation  on $Y$ induced by the image of $f^*\mcF$ in $T_Y$, 
 via  the differential map of $f$.

%Thanks to this result, we may define the canonical class of a foliation. 
Assume that $X$ is a complex manifold and that $\mcF$ a holomorphic foliation on $X$.    
We define 
\[K_\mcF := (\det \mcF)^* \] 
and  use the following notation. 
For  a $\mathbb{Q}$-divisor $\D$ on $X$, for an integer $m$ such that $m\D$ is integral, we set 
\[
\mcO_X(m(K_\mcF + \D)) := (\det \mcF)^{\otimes (-m)} \otimes \mcO_X(m\D). 
\]

\section{Miayoka's inequality on complex orbifolds} 
\label{section:chern}

The goal of this section is to prove a Miyaoka type inequality on compact  K\"ahler threefolds with quotient singularities, see Proposition \ref{prop:psef-c2-general}.    
Complex analytic varieties with quotient singularities are closely related to complex orbifolds. 
Hence, in this section, we will first recall the notion of complex orbifolds and orbifold coherent sheaves.  
Then we will introduce  blowups of  complex orbifolds, which enable us to reduce the study of the second Chern class  of orbifold torsion-free sheaves to the one of orbifold vector bundles.  
With these tools in hand, we can  then prove the main result  of the section.

\subsection{Complex orbifolds}   
 
Smooth ($\mathcal{C}^{\infty}$)  orbifolds or V-manifolds  were first introduced in \cite{Satake1956} and later on were studied in various contexts.
%in \cite{Bla96} and other contexts.  
%A detailed explanation   can be found in \cite[Section 4]{ChenRuan2022} for example.  
A complex orbifold is a smooth orbifold that every local chart is a complex manifold and every local group action is holomorphic. 
We briefly recall its definition by adapting the notion of  \cite[Section 4]{ChenRuan2002}.

\begin{definition}\label{def:orbifold}
We take the following definitions.
\begin{enumerate}
    \item Let $U$ be a connected Hausdorff space. An \textit{orbifold chart} (or a \textit{uniformization system}) of $U$ is a triple $(\widetilde{U}, G, \pi)$ such that $\widetilde{U}$ is an  open domain  in some $\mathbb{C}^n$, that $G$ is a finite group acting holomorphically on $\widetilde{U}$, and that $\pi\colon \widetilde{U}\to U$ is a continuous map inducing a homeomorphism from $\widetilde{U}/G$ to $U$.  
    We  denote by $\mathrm{ker}\, G \subseteq G$ the maximal subgroup acting trivially on $U$.  The chart is called \textit{effective} if $\mathrm{ker}\, G $ is the trivial  subgroup.  
    It is called \textit{standard} if  the action of $G$ is free in codimension one.

 %   \item Two orbifold chart $(\widetilde{U}_1,G_1,\pi_1)$ and $(\widetilde{U}_2,G_2,\pi_2)$  of $U$ are isomorphic if there are isomorphisms $\varphi \colon \widetilde{U}_1 \to \widetilde{U}_2$ and $\lambda \colon G_1 \to G_2$ such that $\pi_1 = \pi_2 \circ \varphi$ and that $\varphi$ is equivariant with respect to $\lambda$.

    \item Let $i\colon U' \hookrightarrow  U$ be an open subset, and $(\widetilde{U}',G',\pi')$ an orbifold chart of $U'$. 
          An \textit{injection} (or embedding) from $(\widetilde{U}',G',\pi')$ to  $(\widetilde{U}, G, \pi)$ consists of an open embedding $\varphi \colon \widetilde{U}'\to \widetilde{U}$ and a group monomorphism $\lambda \colon G'\to G$ such that $\varphi$ is equivariant with respect to $\lambda$,  that $\lambda$ induces an isomorphism from  $\mathrm{ker}\, G' $ to $\mathrm{ker}\, G$, and that $i\circ \pi' = \pi\circ \varphi$.  

    \item Let $X$ be a connected second countable Hausdorff space. An \textit{ orbifold atlas} on $X$ is a collection $\mathcal{U} = \{(\widetilde{U},G,\pi)\}$ of orbifold charts of open subsets of $X$ which covers $X$ and is compatible in the following sense. 
    For any two  orbifold charts $(\widetilde{U},G,\pi)$ and $(\widetilde{U}',G',\pi')$ in $\mathcal{U}$,  
    of open subsets $U$ and $U'$ respectively, for any point $x\in U\cap U'$, there is an open neighborhood $V$ of $x$ with an orbifold chart  $(\widetilde{V},H,\rho)$ such that there are injections from $(\widetilde{V},H,\rho)$ to  $(\widetilde{U},G,\pi)$ and  $(\widetilde{U}',G',\pi')$.

    \item An  orbifold  atlas  $\mathcal{U}$ on $X$ is a \textit{refinement} of another one $\mathcal{V}$ if every orbifold chart of $\mathcal{U}$ admits an injection into an orbifold chart of $\mathcal{V}$.

    \item  Two orbifold  atlas   on   $X$ are \textit{equivalent} if they admit a common refinement. 
           An orbifold atlas is \textit{maximal} if it contains any atlas equivalent to it. 
           From Zorn's lemma, we see that every equivalent class of orbifold atlas contains a unique maximal element.

    \item A complex orbifold is  a connected second countable Hausdorff space $X$ equipped with an equivalent class of orbifold atlas. 
    We write $X_{\orb} =  \{ (V_i,G_i,\pi_i) \}$ or simply $X_{\orb} = \{ (V_i,G_i) \}$ for such an orbifold structure with an orbifold atlas $\{ (V_i,G_i,\pi_i)\}$. 
    %, where we omit the projection map $\pi_i$ for simplicity. 
    The underlying spcae $X$ is called the \textit{quotient space} of $X_{\orb}$. 
    The dimension of $X_{\orb}$ is the complex dimension of any orbifold chart. 
    %We denote by $pr_X \colon \sqcup V_i \to X$   the natural projection. 

    \item  A complex orbifold is called \textit{effective} or \textit{reduced},    if every chart is effective. 
    It is called \textit{standard} if every chart is standard.   
           
     \item Let $X_{\orb}=\{(V_i,G_i)\}$ be a complex orbifold. 
     An \textit{orbifold subvariety} $Z_{\orb}$ is defined by a collection $\{W_i\subseteq V_i\}$ of $G_i$-invariant subvarieties  which is compatible on the overlaps, in the following sense. 
     If there are injections of orbifold charts $\varphi_i\colon V \to V_i$ and $\varphi_j\colon V\to V_j$, then $\varphi_i^{-1}(W_i) =\varphi_j^{-1}(W_j).$  
     The subvariety $Z_{\orb}$ is called smooth if every $W_i$ is a submanifold of $V_i$. 
     When $Z_{orb}$ is smooth with irreducible image $Z$ in $X$, we obtain an induced orbifold structure $Z_{\orb}=\{(W_i,G_i)\}$. 

     \item An orbifold differential form $\sigma $ on an orbifold $X_{\orb}=\{(V_i,G_i)\}$ is given by a collection $\{ \sigma_i \}$ of $G_i$-invariant differential forms on the collection of charts $\{V_i\}$  which is compatible on the overlaps. 

    %\item A complex orbifold $X_{\orb}=\{(V_i, G_i)\}$ of  dimension $n$ is a second countable connected Hausdorff space $X$ equipped with a  collection of charts $\{(V_i,G_i)\}$  such that $V_i$  is a  connected  open subset of $\mathbb{C}^n$, $G_i$ is a finite group acting faithfully and holomorphically  on $V_i$, $\pi_i \colon V_i \to X$ is a continuous $G_i$-equivariant map which induces an homoemorphism from $V_i/G_i$ to $\pi_i(V_i) \subseteq X$. Furthermore, the collection $\{\pi(V_i)\}$ forms an open covering of $X$. 
    %We call $X$ the quotient space and  denote the natural morphism $\sqcup V_i \to X$ by $pr_X$.   

    %\item  A system of chart  $X_{\orb}=\{(V_i, G_i)\}$ is called standard if each action of $G_i$ is free in codimension one. 

   % \item For any open subset $U\subset X$, a uniformization of $U$ is a pair $(V,G)$ where $V$ is a connected complex manifold and $G$ is finite group acting  on $V$, faithfully, holomorphically  and freely in codimension one,  such that $U\cong V/G$.  

\end{enumerate}
\end{definition}

%Throughout this paper, every orbifold is assumed to be effective unless specified otherwise. 
In \cite[Theorem 1]{Satake1956}, Satake proved the de Rham isomorphism theorem for effective  oriented 
$\mathcal{C}^\infty$ orbifolds, in particular for effective  complex orbifolds: for every integer $p\> 0$,  there is a natural isomorphism
\begin{equation}\label{eqn:de-rham-thm}
H^p_{\rm dR}(X_{\orb},\mathbb{R}) \cong H^p(X,\mathbb{R}).    
\end{equation}

%Satake also proved  the Poincar\'e duality for compact oriented orbifolds, see \cite[Theorem 3]{Satake1956}.

 \begin{remark}\label{remark-orbifold-def} 
 If $X_{\orb} = \{(V_i,G_i)\}$ is a non effective orbifold, then there is a canonically induced effective orbifold structure $X_{\orb}' = \{(V_i, G'_i)\}$, where $G_i' = G_i/\mathrm{ker}\, G_i$. 
 We note that an orbifold differential form on $X_{\orb}$ is an orbifold differential form on $X_{\orb}'$, and vice versa. 
 Therefore, the  isomorphism  \eqref{eqn:de-rham-thm} still holds for non effective orbifold. 

 %The reason why we introduce non effective orbifold is to study smooth orbifold subvarieties $Z_{\orb} \subseteq Y_{\orb}$.  
 %Even $Y_{\orb}$ is effective, the induced orbifold structure on $Z_{\orb}$ may   be non effective.  
\end{remark}

%Satake   constructed partitions of the unity for orbifolds. 
%We can then define integration on a compact orbifold. 
%We note that,  local   uniformizations always exist, and is unique up to isomorphism, see \cite[Theorem 1]{Prill1967}. 

%We can extend this notion and define a current  on an orbifold. It is a differential form with distribution coefficient, see [Demailly, 1.C]. We can obtain   de Rham cohomology by using currents, which is naturally isomorphic to the previous one, see [Demailly, page 12].

\begin{definition}\label{def:orb-vector-bundle}
We also use the following notions of orbifold vector bundles and Hermitian metrics, etc. 
Let  $X_{\orb}=\{(V_i, G_i)\}$ be a complex  orbifold, possibly non effective.
    \begin{enumerate}
        \item An orbifold vector bundle $\mcE_{\orb} = \{\mcE_i\}$ on $X_{\orb}=\{(V_i, G_i)\}$  is a collection of holomorphic $G_i$-linearized vector bundles $ \mcE_i $ on $V_i$  which is compatible along the overlaps, in the following sense.   
        For any chart $(V,G,\pi)$ in the maximal atlas of $X_{\orb}$,  there is a $G$-linearized vector bundle $\mcE$ on $V$ satisfying the following properties.  
        If $(V,G)$ is equal to some $(V_i,G_i)$, then $\mcE = \mcE_i$. 
        If there is an injection $\varphi  \colon (V,G,\pi) \to (V',G',\pi')$, then there is an isomorphism $\Psi_{\varphi} \colon \mcE \cong  \varphi^*\mcE' $. 
        Furthermore, if $\iota \colon (V',G',\pi') \to (V'',G'',\pi'')$ is another injection, then we have $\Psi_{\iota \circ \varphi} = (\varphi^*\Psi_{\iota}) \circ \Psi_{\varphi}$.

    %    \item Assume that $X_{\orb}$ is an effective orbifold.   Then an orbifold vector bundle  $\mathcal{E}_{\orb}=\{\mathcal{E}_i\}$ is equivalent to the following data.           There is an open dense subset $X^\circ \subseteq X$ such that each projection map $\pi_i\colon V_i \to X$ is \'etale over $X^\circ$. Hence $X^\circ$ has a natural structure of complex manifold. There is a holomorphic vector bundle $\mcE^\circ$ on $X^\circ$ such that  $\pi_i^*\mcE^\circ$ extends to the vector bundle $\mcE_i$ on $V_i$.

        \item An orbifold coherent sheaf $\mcE_{\orb}=\{\mcE_i\}$ on $X_{\orb}$ is defined in the same manner.
        %, see for example \cite[Section 2.1]{Faulk2022} for more details. 

        \item An orbifold coherent sheaf $\mcE_{\orb}$ is called torsion-free (resp. reflexive) if every data $\mcE_i$ on $V_i$ is so.

        \item For an orbifold vector bundle $\mcE_{\orb}$,  a Hermitian metric $h_{\orb}$ on $\mcE_{\orb}$ is a collection $\{h_i\}$ of $G_i$-invariant Hermitian metrics on $\mcE_i$ which are compatible on the overlap.

        \item For a Hermitian orbifold vector bundle $(\mcE_{\orb}, h_{\orb})$, 
        we can define the orbifold Chern classes $\hat{c}_i (\mcE_{\orb})$  using the curvature tensors $c_i(h_{\orb})$, which are automatically orbifold differential forms.  
        They are uniquely determined in the orbifold de Rham cohomology groups of $X_{\orb}$, hence by \eqref{eqn:de-rham-thm} in the singular cohomology groups of  $X$.  We note that this does not depend on the choice of metrics. 
% For more details, we refer to \cite[Section 2]{Bla96}.

        \item The  complex orbifold $X_{\orb}$ is called \textit{ K\"ahler} if there is a closed  positive orbifold $(1,1)$-form $\omega$. 
 A cohomology class $\alpha \in H^2(X, \mathbb{R})$ is called an orbifold K\"ahler class on $X_{\orb}$ if there is an orbifold K\"ahler form $\omega$ whose cohomology class is $\alpha$. 
    \end{enumerate}
\end{definition}

%We note that orbifold vector bundles and their Chern classes are also well-defined for non effective orbifolds, see for example \cite[Section 4.3]{ChenRuan2002} 

\begin{remark}\label{rmk:quotient-singularities}
Complex orbifolds are closely related to complex analytic varieties with quotient singularities. 
We summarize  some properties as follows.

\begin{enumerate}
\item From \cite[Th\'eor\`eme 1]{Cartan1954}, we see that the 
quotient space $X$ of an effective complex orbifold $X_{\orb}$ has a natural structure of complex analytic variety, with  quotient singularities. 
Furthermore, if $(V_i,G_i)$ is an orbifold chart of $X_{\orb}$, then the quotient map $\pi_i\colon V_i \to \pi_i(V_i)$ is holomorphic. 
The   holomorphic functions on any open subset $U\subseteq \pi_i(V_i)$ are exactly the $G_i$-invariant holomorphic functions on $\pi_i^{-1}(U)\subseteq V_i$.

\item Conversely, given  a complex analytic variety $X$ of dimension $n$ with quotient singularities,  it admits a unique standard orbifold structure   after \cite[Theorem 1 and Theorem 2]{Prill1967}.  
We call it  the \textit{standard orbifold structure} on $X$.
% By \cite[Lemma 9.9]{Kaw88}, we may assume that $V_i\subseteq \mathbb{C}^n$ is an open neighborhood of the origin, and that the action of $G_i$ comes from a linear action on $\mathbb{C}^n$. 
% Then, thanks to the  Chevalley–Shephard–Todd theorem, we may assume that each action of $G_i$ on $V_i$ is free in codimension $1$. 
% We will make such an assumption throughout the whole paper, whenever we consider an orbifold structure on a complex analytic variety with quotient singularities. 

\item A same quotient space $X$ may have non isomorphic orbifold structures.  For example, we take $X = \mathbb{C}$. 
Let $( \mathbb{C}, \{\mathrm{Id}\}, \pi_1 )$ and $( \mathbb{C}, \{\mathrm{Id}, -\mathrm{Id}\}, \pi_2)$ be two orbifold charts of $X$, where $\pi_1$ is the identity map and $\pi_2$ is the   map sending a complex number $z$ to $z^2$. 
Then these two charts induces different orbifold structures on $X$.

\item  Let $X_{\orb}= \{(V_i,G_i)\}$ be an effective complex orbifold and $X$ its quotient space. 
Assume that $Y$ is a complex analytic space and there are $G_i$-equivariant morphisms $q_i\colon V_i\to Y$, 
where the actions of the $G_i$'s on $Y$ are trivial. 
Assume that   the $q_i$'s are compatible on the overlap. 
Then they induces a holomorphic map $f\colon X\to Y$. 
Indeed, by the definition of topological quotient spaces, the map is   well-defined and continuous. 
To prove that it is holomorphic, it is enough to show that for any local holomorphic function $\varphi$ on $Y$, the function  $\varphi\circ f$ is holomorphic.  
By definition, each $\varphi\circ q_i$ is  $G_i$-invariant and holomorphic. 
Hence we can conclude by (1) of  Remark \ref{rmk:quotient-singularities}. 
  
\item  Let $X_{\orb}= \{(V_i,G_i, \pi_i)\}$ be a standard complex orbifold,  and let  $X$ be its quotient space. 
Then  there is a correspondence between  reflexive  coherent sheaves $\mcE$ on $X$  and  orbifold reflexive sheaves $\mcE_{\orb}=\{\mcE_i\}$ on $X_{\orb}$.
Indeed, if $\mcE$ is a reflexive coherent sheaf on $X$, then the reflexive pullbacks $(\pi_i^*\mcE)^{**}$ define  an orbifold reflexive sheaf on $X_{\orb}$. 
Conversely, if $\mcE_i$ is a $G_i$-equivariant reflexive coherent sheaf on $V_i$, then the $G_i$-invariant pushforward  $((\pi_i)_* \mcE_i)^{G_i}$ is a reflexive sheaf on $\pi_i(V_i) \subseteq X$,  see for example \cite[Lemma A.4]{GKKP11}.

\item Let $X_{\orb}= \{(V_i,G_i, \pi_i)\}$ be a  complex orbifold  and $X$ its quotient space. 
Assume that $\mcL$ is a torsion-free sheaf on $X$ of rank one such that $\mcN=(\mcL^{\otimes b})^{**}$ is a line bundle of some $b>0$. 
We can define the ($\mathbb{Q}$-)Chern class $c_1(\mcL)\in H^2(X, \mathbb{R})$ as $b^{-1}c_1(\mathcal{N})$.  
In addition, we can defined the orbifold first Chern class $\hat{c}_1(\mathcal{L})$ of $\mcL$ as follows. 
We let $\mathcal{L}_{\orb}$ be the orbifold reflexive coherent sheaf on $X_{\orb}$ defined by $(\pi_i^*\mcL)^{**}$, which must be an orbifold line bundle, see for example \cite[Chapter 2, Lemma 1.1.15]{OSS11}.  
Then we define $\hat{c}_1(\mathcal{L}) = \hat{c}_1(\mathcal{L}_{\orb}).$  
We claim that
\[b\cdot c_1(\mcL) = c_1(\mcN)  = \hat{c}_1(\mcN)   = b\cdot \hat{c}_1(\mathcal{L}) \in H^2(X, \mathbb{R}), 
\]
where we identify $\mathcal{N}$ with the  orbifold line bundle on $X_{\orb}$ induced by  it.   
The first equality is by definition, and the last one follows from the fact that $((\pi_i^*\mcL)^{**})^{\otimes b} \cong \pi_i^*\mathcal{N}$.  
To see the equality in the middle,  we first note that,   locally on $X$, the line bundle $\mathcal{N}$ is trivial and we can embed $X$ into some affine space $\mathbb{C}^N$. 
Thanks to a partition of the unity,  
we  can hence construct a Hermitian metric on $\mathcal{N}|_{X_{\sm}}$, which extends to an orbifold Hermitian metric on the orbifold line bundle   induced by  $\mathcal{N}$.   
It follows that  $\hat{c}_1(\mcN)|_{X_{\sm}} = c_1(\mcN)|_{X_{\sm}}$.    
Since $X_{\sing}$ has real codimension at least 4 in $X$, we obtain that 
$\hat{c}_1(\mcN)  = c_1(\mcN)  \in  H^2(X, \mathbb{R})$. 
This explains that the orbifold first Chern class coincides with the first Chern class for a $\mathbb{Q}$-Cartier reflexive sheaf or   a  $\mathbb{Q}$-Cartier divisor.  
For this reason, we may use  $\hat{c}_1$ and $c_1$ interchangeably.

\item Let $(V,G, \pi)$ be an  orbifold chart of $U$ and let $b=|G|$.  
If $\mcL$ is a torsion-free sheaf on $U$ of rank one, 
then $ (\mcL^{\otimes b})^{**}$ is a line bundle.  
Indeed, $(\pi^*\mcL)^{**}$  is a $G$-linearized line bundle on $V$. 
For any point $v\in V$, the action of its  stabilizer $G_v$  on the fiber $((\pi^*\mcL)^{**})^{\otimes b}_v$ is trivial  (this is known as Kempf's descending condition, see \cite[Th\'eor\`eme 2.3]{DN1989} for the algebraic version).   
It follows that there is some $G$-invariant local section of $((\pi^*\mcL)^{**})^{\otimes b}$ which does not vanish at $v$. 
This implies that there is a line bundle $\mathcal{M}$ on $U$ such that $\pi^*\mathcal{M} \cong ((\pi^*\mcL)^{**})^{\otimes b}$  and  $(\mcL^{\otimes b})^{**} \cong \mathcal{M}$. 
As a consequence, if  $X_{\orb}= \{(V_i,G_i, \pi_i)\}$ is  a  complex orbifold with a finite family of chart, if  $\mcL$ is a torsion-free sheaf on the quotient space $X$ of rank one, 
then   $(\mcL^{\otimes q})^{**}$ is a line bundle for the lcm $q$ of the cardinalities $|G_i|$. 
\end{enumerate}
\end{remark}

%Thus by \cite[Th\'eor\`eme 2.3]{DN1989}, there is a line bundle $\mathcal{M}$ on $U$ such that $\pi^*\mathcal{M} \cong ((\pi^*\mcL)^{**})^{\otimes b}$. 

After (5) of Remark \ref{rmk:quotient-singularities}, we will use the following notion.  
Assume that  $X$ is a complex analytic variety with quotient singularities, and that $\mcE$ a reflexive coherent sheaf on $X$. 
If $\mcE$ induces an orbifold vector bundle $\mcE_{\orb}$ on the standard orbifold structure, we set 
\[
\hat{c}_i(\mcE) := \hat{c}_i(\mcE_{\orb}).
\]

The following lemma from \cite[Lemma 1 and Proposition 4]{Wu23}  
(see  also  \cite[Construction 6.1]{CGNPPW}) 
relates K\"ahler orbifolds and K\"ahler varieties with quotient singularities.  
For the reader's convenience, we include a proof.

\begin{lemma}
\label{lemma:construction-orb-kahler}
Let $X$ be a compact complex analytic variety with quotient singularities, and 
let $X_{\orb}$ be an effective orbifold structure on $X$. 

If there is a K\"ahler form $\omega$ on $X$, 
then there is an orbifold K\"ahler form on $X_{\orb}$ of the shape 
\[
\omega_{\orb} = pr_X^* \omega + \sqrt{-1}\partial\overline{\partial} \varphi, 
\]
for some $\mathcal{C}^{\infty}$ orbifold real function $\varphi$, where $pr_X^*\omega$ is the orbifold form on $X_{\orb}$ defined by pulling back  $\omega$ on every orbifold chart.  
In particular, the class of $\omega$ in $H^{2}(X,\mathbb{R})$ is an orbifold K\"ahler class on $X_{\orb}$. 

Conversely, if there is an orbifold K\"ahler form $\alpha_{\orb}$ on $X_{\orb}$, 
then there is a K\"ahler form $\alpha$ on $X$ which has the same class as $\alpha_{\orb}$ in $H^{2}(X,\mathbb{R})$. 
\end{lemma}

\begin{proof}
For the first part, we will adapt the method of   \cite[Lemma 3.5]{Dem92}. 
Let  $n=\dim X$. 
There is a finite open covering $\{U_i\}$ of $X$, such that 
for each $i$, there is an integer $N_i$ and a relatively compact open subset $W_i\subseteq \mathbb{C}^{N_i}$ with the following properties. 
There is a closed embedding  $\iota_i$ from $U_i$ into $W_i$, there is a K\"ahler form $\Omega_i$ on a neighborhood the closure $\overline{W_i}$,  such that 
$\omega|_{U_i} = \iota^*\Omega_i.$ 

By considering bump functions supported in the $W_i$'s, we can  obtain a partition of the unity $\{\eta_i\}$ subordinate to $U_i$, so that there is a $\mathcal{C}^{\infty}$ function $\rho_i$, compactly supported in $W_i$, such that $\eta_i = \rho_i\circ \iota_i$. 
Since $\Omega_i$ is positive, by compactness, there are real numbers $A_i>0$ such that, on every $W_i$,   
\[
\sqrt{-1}( \rho_i \partial\overline{\partial} \rho_i - \partial \rho_i \wedge \overline{\partial} \rho_i  )
\> -A_i \Omega_i
\]

Up to refining the covering $\{U_i\}$, we can assume that there is an orbifold chart $(V_i,G_i)$ for each $U_i$, such that $V_i\subseteq \mathbb{C}^n$ is a relatively compact open  neighborhood of the origin. 
We denote by $\pi_i \colon V_i \to W_i$ the composite map of $V_i\to U_i\to W_i$, which is a holomorphic map. 
Let $\theta_i = \rho_i \circ \pi_i$. 
Then $\{\theta_i\}$ is a $\mathcal{C}^{\infty}$ orbifold  partition of the unity, subordinate to $\{V_i\}$.  
The inequality of the previous paragraph implies that, on each $V_i$,  
\begin{equation}\label{eqn:K-orb-1}
    \sqrt{-1}( \theta_i \partial\overline{\partial} \theta_i - \partial \theta_i \wedge \overline{\partial} \theta_i  )
\> -A_i \cdot \pi_i^*\Omega_i = -A_i \cdot  pr_X^*\omega.
\end{equation}

For each $i$, we define  
$w_i ({x}) =  \sum_{g\in G_i} ||g( {x})||^2 \>0$ for any $ {x}\in V_i\subseteq \mathbb{C}^n$,
where $||\cdot||$ is the standard Euclidean norm on $\mathbb{C}^n$. 
Then $w_i$ is $G_i$-invariant and $\sqrt{-1} \partial\overline{\partial} w_i$ is a  $G_i$-invariant K\"ahler form on $V_i$.
We set
\[
\gamma:= \epsilon \cdot \sqrt{-1} \cdot  \sum_i \theta_i \partial\overline{\partial} w_i
\] 
for some $\epsilon>0$ small enough, 
so that 
\begin{equation}\label{eqn:K-orb-2}
\sqrt{-1} \partial\overline{\partial} w_i \> \gamma
\end{equation}
on $V_i$ for all $i$. 
By construction,  
$\gamma$ is an orbifold positive $(1,1)$-form on $X_{\orb}$.  
We note that each $w_i$ is bounded on $V_i$. 
Hence  there are constants $C_i$ such that, for any  $x\in V_i$, we have \begin{equation}\label{eqn:K-orb-3}
w_i(x) \< C_i + \sup_{\ k\neq i,\ V_k\ni x} w_k(x), 
\end{equation}
in the following sense. 
Let $y$ be the image of $x$ in $X$. 
Then we say that $x\in V_k$ if and only if $y\in U_k$. 
If this is the case, then we set $w_k(x) := w_k(x')$, where $x'\in V_k$ is any point lying over $y$.

Let 
\[
K= 2\sum_i A_i \exp(C_i)
\]
and 
\[
w= \log\, \left(\sum_i \theta_i^2 \exp(w_i)\right).
\]
Then $w$ is a $\mathcal{C}^\infty$ orbifold function on $X_{\orb}$. 
By same calculation as in \cite[Lemma 3.5]{Dem92}, the properties of \eqref{eqn:K-orb-1}, \eqref{eqn:K-orb-2} and \eqref{eqn:K-orb-3} imply that 
\[
\sqrt{-1} \partial\overline{\partial} w \> \gamma - K\cdot pr_X^* \omega. 
\]
We set 
$\varphi = \frac{1}{K} w$. 
It follows that 
\[
pr_X^* \omega + \sqrt{-1}\partial\overline{\partial} \varphi \> \frac{1}{K} \gamma,
\]
which is positive. 
This proves the first assertion.  \\

For the converse, we still let $\{U_i\}$ be the finite open covering of $X$ 
and  $(V_i,G_i,\pi_i)$ be the orbifold chart of $U_i$, as in the first part. 
In particular, each $U_i$ is relative compact in $X$ and each $V_i$ is relatively compact in $\mathbb{C}^n$.  
Then  $\alpha_{\orb}$ is defined by a K\"ahler form $\alpha_i$  on   $V_i$, 
such that $\alpha_i = \sqrt{-1}\partial\overline{\partial} \psi_i$ for some $G_i$-invariant $\mathcal{C}^\infty$ strongly plurisubharmonic (psh for short) function, defined in a neighborhood of $\overline{V_i}$ in $\mathbb{C}^n$.  
Then $\psi_i$ descends to a continuous function $\psi'_i$ on $U_i$ with $\psi_i=\psi_i' \circ \pi_i$. 

We note that $\psi_i'$ is strongly psh as well by \cite[Lemma II.3.1.2]{Varouchas1989}.  
By \cite[Statement II.2.1]{Varouchas1989}, we deduce that there is a K\"ahler form $\alpha$ on  $X$, whose class in $H^2(X,\mathbb{R})$ is the same as $\alpha_\orb$.  
This completes the proof of the lemma. 
\end{proof}

In the remainder of the paper, for an effective compact complex orbifold $X_{\orb}$ with quotient space $X$, we will identify the notion of K\"ahler classes on $X$ and the notion of orbifold K\"ahler class on $X_{\orb}$.  

\subsection{Stable orbifold sheaves} 
 
We can extend the notion of slope stability for torsion-free sheaves to the setting of  orbifolds.

\begin{definition}\label{def:orbifold-stability}
Let $X_{\orb}$ be a  compact complex orbifold of dimension $n$ and $\mcE_{\orb}$ a torsion-free orbifold coherent sheaf of rank $r>0$.  
Let $\alpha\in H^2(X,\mathbb{R})$ be any class.  
We can define the slope of $\mcE_{\orb}$ with respect to $\alpha^{n-1}$ as
\[
\mu_{\alpha^{n-1}}(\mcE_{\orb}):= \frac{\hat{c}_1(\mcE_{\orb}) \cdot \alpha^{n-1}}{r},
\]
where $\hat{c}_1(\mcE_{\orb})$ is defined as the orbifold first Chern class of the determinant orbifold  line bundle  $(\bigwedge^r \mcE_{\orb})^{**}$.  
The sheaf $\mcE_{\orb}$ is said to be $\alpha^{n-1}$-semistable if for any non zero  subsheaf $\mcF_{\orb} \subseteq \mcE_{\orb}$, we have $\mu_{\alpha^{n-1}}(\mcE_{\orb}) \> \mu_{\alpha^{n-1}}(\mcF_{\orb})$.
It is said to be $\alpha^{n-1}$-stable if the inequality  is  strict for every non zero saturated proper subsheaf $\mcF_{\orb}$. 
\end{definition}

Harder-Narasimhan filtrations and Jordan-H\"older filtrations  exist  for orbifold coherent sheaves, with respect to orbifold K\"ahler classes.  
To see this, we can adapt the classic argument (see for example \cite[Section V.7]{Kob14}) to the orbifold setting. 
 We first prove two lemmas.

\begin{lemma}
\label{lemma:inclusion-inequality}
Let $   X_{\orb} $ be a  compact complex K\"ahler orbifold  of dimension $n\ge 2$. 
Fix some nef classes $\alpha_1,...,\alpha_{n-1} \in H^2(X,\mathbb{R})$. 
Let $\Omega = \alpha_1\cdots\alpha_{n-1}$ and $\mathcal{F}_{\orb} \subseteq \mathcal{E}_{\orb}$  be orbifold coherent torsion-free sheaves on $X_{\orb}$ of the same rank.  
Then $\hat{c}_1(\mathcal{F}_{\orb}) \cdot \Omega \le \hat{c}_1(\mathcal{E}_{\orb}) \cdot \Omega$. 
\end{lemma}

\begin{proof}
By continuity and by perturbing $\alpha_1,...,\alpha_{n-1}$, we may assume that they are K\"ahler classes, and hence orbifold K\"ahler classes by Lemma \ref{lemma:construction-orb-kahler}. 
By taking the determinants, we may assume that $\mathcal{F}_{\orb}$ and $\mathcal{E}_{\orb}$ are orbifold line bundles. 
Up to twisting with $\mathcal{E}_{\orb}^{*}$, we may assume further that $\mathcal{E}_{\orb}$ is  trivial. 
Then $\mathcal{F}_{\orb}$ is an orbifold coherent ideal sheaf  
and $\hat{c}_1(\mathcal{F}_{\orb}) \cdot \Omega \le 0$. 
This completes the proof of the lemma. 
\end{proof}

\begin{lemma}
\label{lemma:A-bounded-above-v2}
Let $  X_{\orb} $ be a  compact complex K\"ahler orbifold  of dimension $n\ge 2$, 
and  let $X$ be its  quotient space. 
Assume that $\alpha_1,...,\alpha_{n-1} \in H^2(X,\mathbb{R})$ are nef classes. 
Let $\Omega = \alpha_1\cdots\alpha_{n-1}$ and let $\mcE_{\orb}$  be  an orbifold coherent torsion-free sheaf on $X_{\orb}$. 
Then the following set of real numbers is bounded from above, 
\[
\{   \hat{c}_1(\mathcal{G}_{\orb}) \cdot \Omega \  : \  \mathcal{G}_{\orb} \mbox{ is an  orbifold coherent subsheaf of } \mathcal{E}_{\orb}\}. 
\] 
\end{lemma}

\begin{proof} 
For any orbifold coherent sheaf $\mathcal{G}_{\orb}$, we set $m(\mathcal{G}_{\orb}):= \hat{c}_1(\mathcal{G}_{\orb}) \cdot \Omega$. 
Let $\Sigma$ be the set of non zero orbifold subsheaves of $\mathcal{E}_{\orb}$. 
We define a partial order $\le$ on $\Sigma$, such that $\mathcal{G}_\orb \le \mathcal{G}_\orb'$ if and only if 
\begin{enumerate}
    \item $\mathcal{G}_\orb \subseteq \mathcal{G}_\orb'$, 
    \item  and $m(\mathcal{G}_\orb) \le m(\mathcal{G}_\orb')$.  
\end{enumerate}
Let $\mathcal{F}_\orb$ be a maximal element of $(\Sigma, \le)$ which has minimal rank.  
Such an element always exists, since   $\Sigma$  satisfies the ascending chain condition with respect to the inclusion.
We will show that $m(\mathcal{F}_\orb) \ge  \sup_{\mathcal{G}_\orb \in \Sigma}m(\mathcal{G}_\orb)$, which implies the lemma.  

For simplicity, we omit the lower script $_\orb$ in the following argument. 
Assume by contradiction that there is some $\mathcal{G}\in \Sigma$ such that $m(\mathcal{G}) > m(\mathcal{F})$. 
If $\mathcal{G} \not\subseteq \mathcal{F}$, then $\mathcal{F}\subsetneq \mathcal{F} + \mathcal{G}$. 
It follows that $m(\mathcal{F} + \mathcal{G}) < m(\mathcal{F})$ for  $\mathcal{F}$ is maximal.  
There is an exact sequence of orbifold coherent sheaves 
\[
0\to \mathcal{F} \cap \mathcal{G} \to \mathcal{F} \oplus \mathcal{G} \to \mathcal{F}+ \mathcal{G} \to 0, 
\]
and hence  $m( \mathcal{F} \cap \mathcal{G})+ m(\mathcal{F} + \mathcal{G}) = m( \mathcal{F}  )+ m( \mathcal{G})$. 
We deduce that 
\[m( \mathcal{F} \cap \mathcal{G}) >  m(\mathcal{G}) > m(\mathcal{F}). \]
Therefore, replacing $\mathcal{G}$ by $\mathcal{F} \cap \mathcal{G}$, 
we can assume that $\mathcal{G} \subseteq \mathcal{F}$.  
Furthermore, we may assume that $\mathcal{G}$ is maximal among all elements $\mathcal{H}$ of $\Sigma$, such that $\mathcal{H} \subseteq  \mathcal{F}$ and that  $m(\mathcal{H}) > m(\mathcal{F})$.  

By Lemma \ref{lemma:inclusion-inequality}, we must have $\mathrm{rank}\, \mathcal{G} < \mathrm{rank}\, \mathcal{F}$. 
Let $\mathcal{G}'$ be a maximal element, among the elements of $\Sigma$ containing $\mathcal{G}$. 
Then  $\mathcal{G}'$ is also a maximal element of $\Sigma$, and  $m(\mathcal{G}')   \ge m(\mathcal{G}) > m(\mathcal{F})$. 
By Lemma \ref{lemma:inclusion-inequality}, either 
   $\mathcal{G}'  \not\subseteq \mathcal{F}$ 
   or $\mathrm{rank}\, \mathcal{G}' < \mathrm{rank}\, \mathcal{F}$.  
Since $\mathcal{F}$ is a  maximal element of minimal rank, we deduce that $\mathcal{G}' \not\subseteq \mathcal{F}$. 
Thus $m(\mathcal{F}+\mathcal{G}') < m(\mathcal{F})$.    
By the same argument of the previous paragraph, we obtain that 
\[m(\mathcal{F}\cap \mathcal{G}') > m(\mathcal{G}') \ge m(\mathcal{G}). \]
This is a contradiction, 
for   $\mathcal{G} \subseteq \mathcal{F}\cap \mathcal{G}' \subseteq \mathcal{F}$, and $\mathcal{G}$ is chosen to be maximal at the end of  the previous paragraph. 
\end{proof}

Now we can conclude the existence of Harder-Narasimhan filtrations and Jordan-H\"older filtrations.

\begin{lemma}
\label{lemma:filtration}
Let $X_{\orb}$ be a compact complex orbifold of dimension $n$ and $\omega$ an orbifold K\"ahler class on $X_{\orb}$. 
Let $\mcE_{\orb}$ be a torsion-free orbifold  coherent sheaf on  $X_{\orb}$. 
Then the following assertions hold. 
\begin{enumerate}
\item There is a unique filtration of saturated orbifold subsheaves 
\[ 0 = \mcF^0_{\orb} \subseteq  \cdots \subseteq \mcF^k_{\orb} = \mcE_{\orb} \]
such that the  subquotients $\mcF^j_{\orb}/\mcF^{j-1}_{\orb}$ are  $\omega^{n-1}$-semistable with strictly decreasing slopes.  
\item Assume that $\mcE_{\orb}$ is  $\omega^{n-1}$-semistable. 
Then there is a filtration  of saturated orbifold subsheaves 
\[ 0 = \mcG^0_{\orb} \subseteq  \cdots \subseteq \mcG^k_{\orb} = \mcE_{\orb} \]
such that the subquotients $\mcG^j_{\orb}/\mcG^{j-1}_{\orb}$ are   $\omega^{n-1}$-stable with   equal slopes.
\end{enumerate}
\end{lemma}

\begin{proof}
Following the strategy of \cite[Theorem V.7.15 and Theorem V.7.18]{Kob14}, the lemma follows from  Lemma \ref{lemma:A-bounded-above-v2}. 
\end{proof}

\subsection{Bimeromorphic morphisms of complex orbifolds} 

We will define the blowup of an orbifold.  
The objective of this operation is to reduce the study of orbifold coherent torsion-free sheaf to that of orbifold vector bundle.

Let $Y_{\orb}=\{(V_i,G_i)\}$ be a compact complex   orbifold, and let 
  $Z_{\orb}$ be an orbifold subvariety of $Y_{\orb}$ which consists of a family $\{W_i\subseteq V_i\}$ of $G_i$-invariant closed subvarieties. 
  % We note that the action of $G_i$ on $W_i$ might not be faithfull in this notion of orbifold subvariety. 
We assume that  the image   $Z$ of $Z_{\orb}$ in $Y$ is irreducible.  

%When each $W_i$ is a submanifold of $V_i$, we say that $Z_{\orb}$ is smooth.  In this case, 
When each $W_i$ is smooth, we  can define the blowup 
\[ f_{\orb} \colon X_{\orb} \to Y_{\orb} \]
of $Y_{\orb}$ with center $Z_{\orb}$ as follows. 
Let $\varphi_i \colon \widehat{V}_i \to V_i$ be the blowup with center $W_i$. 
There is an induced action of $G_i$ on $\widehat{V}_i$ such that  $\varphi_i$ is $G_i$-equivariant. 
Then $X_{\orb} = \{(\widehat{V}_i,G_i)\}$  is a complex compact orbifold and the family of morphisms $\{\varphi_i\}$ induces the morphism $f_{\orb}$ of orbifolds. 
We call it the \textit{blowup} of $Y_\orb$ at the  smooth center $Z_{\orb}$.  
It induces a holomorphic morphism  $f \colon X\to Y$  of the quotient spaces. 
%We remark that $f$ is good in the sense of \cite[Definition 4.4.1]{ChenRuan2002}. 
%Therefore, the same argument as \cite[Lemma 4.4.3]{ChenRuan2022} implies that, we can pullback orbifold vector bundles,  more generally orbifold coherent sheaves,  from $X_{\orb}$ to $Y_{\orb}$. 

Let $F_i\subseteq \widehat{V}_i$ be the exceptional locus of $\varphi_i$.  
They define a  smooth  orbifold subvariety $E_{\orb}$ of $X_{\orb}$. 
Each morphsim $F_i \to W_i$ is a $\mathbb{P}^d$-bundle, where $d+1$ is the codimension of $Z$ in $Y$. 
For any integer $n$, 
we  can define the orbifold line bundle  $\mcO_{X_{\orb}}(nE_{\orb}) := \{ \mcO_{\widehat{V}_{i}}(nF_{i} )\}$. 
We note that,  for  $\mcO_{\widehat{V}_{i}}(-F_{i})$, its restriction on every fiber of $\varphi_i$ is isomorphic to $\mcO_{\mathbb{P}^d}(1)$, and its restriction on $ \widehat{V}_i\setminus{F_{i}}$ is trivial.  
%Therefore, the 
A similar argument as in \cite[Proposition 3.24]{Voisin2002} implies that, if $\alpha$ is an orbifold K\"ahler class on $Y_{\orb}$, then there is a rational number  $b>0$, such that 
\[ 
f^*\alpha - b \hat{c}_1(\mcO_{X_{\orb}}(E_{\orb}))
\]
is an orbifold  K\"ahler class on $X_{\orb}$.  

Alternatively, let $E\subseteq X$ be the  exceptional locus of $f\colon X\to Y$, which is equal to the image of $E_{\orb}$ in $X$. 
Then there is a positive integer $k$ such that $\hat{c}_1(\mcO_X (E) ) = k \hat{c}_1(\mcO_{X_{\orb}}(E_{\orb})$. 
We then  deduce that 
\[ 
f^*\alpha - a \hat{c}_1(\mcO_{X}(E))
\]
is an orbifold K\"ahler class for some rational number $a>0$.  \\

We remark that the relative ampleness of $\mathcal{O}_{\widehat{V}_i}(-F_i)$ over $V_i$ above still holds even if the blowup center $W_i$ is not smooth.  
More generally, we can consider  a morphism   $f_{\orb} \colon X_{\orb}   \to Y_{\orb}$  
such that on each orbifold chart, $f_\orb$ is defined as a proper bimeromorphic map $\varphi_i\colon \widehat{V}_i\to V_i$.  
We call it a \textit{proper bimeromorphic morphism of orbifolds}. 
If each $\varphi_i$ is a composition of   blowups at $G_i$-invariant centers,   
we say that  $f_{\orb}$ is a \textit{composition of  blowups}. 
In this case, if there are $k$ irreducible components $E_1,...,E_k$ of the exceptional locus of the induced morphism $f\colon X\to Y$,    
and if $\omega$ is an orbifold K\"ahler class on $Y_{\orb}$, then there are rational numbers $a_1,...,a_k>0$ such that 
\begin{equation}\label{eqn:orbifold-kahler-class} 
f^*\alpha - \sum_{i=1}^k a_j \hat{c}_1(\mcO_{X}(E_{j}))
\end{equation}
is an orbifold K\"ahler class on $X_\orb$.

We recall the following statement on resolution of singularities, 
which allows us to make a coherent sheaf locally free by taking functorial bimeromorphic transforms.  
A construction is called functorial if it commutes with local analytic isomorphisms. 
In particular,  the functoriality implies  the construction is equivariant with respect to any actions of groups.

\begin{theorem}
\label{thm:func-reso-loc-free}
Let $X$ be a complex analytic variety, and let $\mathcal{E}$ be a torsion-free coherent sheaf on $X$. 
Then there is a  bimeromorphic morphism $r\colon \widetilde{X} \to X$ 
satisfying the following properties. 
\begin{enumerate} 
\item $r^*\mathcal{E}/\mathrm{torsion}$ is locally free.  
\item $\widetilde{X}$ is a smooth complex analytic variety. 
\item $r$ is obtained by a sequence of blowups at  centers contained in the union of $X_{\sing}$ and the non-locally-free locus of $\mathcal{E}$.
\item The construction of  $r$ is functorial. 
\end{enumerate} 
\end{theorem}

\begin{proof} 
Since we will show the functoriality of the construction, 
we only need to prove the theorem locally.  
We can assume that $(o\in X)$ is a germ of complex analytic variety.  
Then  there is  an exact sequence of coherent sheaves 
\begin{equation}\label{eqn:res-sheaf-01}
    \mathcal{O}_X^{\oplus p} \overset{\varphi}{\longrightarrow} \mathcal{O}_X^{\oplus q}  \overset{\psi}\longrightarrow \mathcal{E} \to 0.
\end{equation}
The morphism $\varphi$ is represented by a  $q\times p$ matrix $\Theta(x)$ with entries as holomorphic functions on $X$. 
Since $\mathcal{E}$ is torsion-free, if $M$ is the maximum of the ranks of $\Theta(x)$ for $x\in X$, then $q-M$ is the rank of $\mathcal{E}$. 
Let $m$ be the minimum of the ranks of $\Theta(x)$ for $x\in X$. 
Then $\mathcal{E}$ is locally free if and only if $M=m$.     
Assume that $m<M$. Then the rank of $\Theta(o)$ is equal to $m$ up to shrinking $X$ around $o$.  
We also note that $q-m$ is equal to the dimension of the residue of $\mathcal{E}$ at $o$, 
that is, $q-m=\dim\left (\mcE\otimes_{\mathcal{O}_X} \mbC(o)\right)$, where $\mathbb{C}(o)\cong \mathbb{C}$ is the residue field of $X$ at   $o$. 
We let $\mathcal{I}\subseteq \mathcal{O}_X(X)$ be ideal generated by the determinants of all $(m+1)$-minors of the $\Theta(x)$.
We will prove the theorem in three steps. 
\\

\textit{Step 1.} In this step, we will show that the ideal sheaf $\mathcal{I}$ on $(o\in X)$ depends only on $\mathcal{E}$, and is independent of the choice of the exact sequence \eqref{eqn:res-sheaf-01}.   
We note that the sheaf $\mathcal{I}$ is unchanged if we compose $\varphi$ in \eqref{eqn:res-sheaf-01} by some automorphisms of  $\mathcal{O}_X^{\oplus p}$ or $\mathcal{O}_X^{\oplus q}$ over $X$.  
Hence, we may assume that $\Theta(x)$ is of the form 
\[
\begin{pmatrix}
\mathbf{1}_m & \mathbf{0}\\
\mathbf{0} & \Lambda(x) &
\end{pmatrix},	
\]
where $\mathbf{1}_m$ is the constant identity square matrix of rank $m$. 
Hence $\mathcal{I}$ is the matrix generated by the entries of $\Lambda(x)$. 
It follows that we can always assume that $m=0$ and $q$ is equal to the dimension $d$ of the residue  of $\mathcal{E}$ at $o$.

We will next show that $\mathcal{I}$ may depend on  $\psi$ in \eqref{eqn:res-sheaf-01}, but is independent of $\varphi$.  
Indeed, we assume that  $\theta\colon \mathcal{O}_X^{p'} \to \mathcal{O}_X^{q}$ is another resolution of the kernel of $\psi$. 
Then there is a morphism $\alpha \colon \mathcal{O}_X^{p'} \to \mathcal{O}_X^{p}$  such that $\theta =\varphi\circ \alpha$. 
This shows that the ideal defined by $\theta$ is contained in the one defined by $\varphi$. 
Since the rolls of $\varphi$ and $\theta$ are symmetric, we deduce that these ideals defined by $\vphi$ and $\theta$ are equal.

Now we will show that $\mathcal{I}$ is also independent of the choice of $\psi$.   
As shown in the first paragraph of Step 1, we may assume that $q=d$ in  the exact sequence \eqref{eqn:res-sheaf-01}, where $d$ is the dimension of the residue of $\mathcal{E}$ at $o$.
Assume that 
\begin{equation*} 
    \mathcal{O}_X^{\oplus s} \overset{\varphi'}{\longrightarrow} \mathcal{O}_X^{\oplus d}  \overset{\psi'}{\longrightarrow} \mathcal{E} \to 0 
\end{equation*}
is another exact sequence. 
Since we have assumed that $q=d$ in \eqref{eqn:res-sheaf-01}, 
there is an endomorphism $\beta $ of $\mathcal{O}_X^{\oplus d}$ such that $\psi' = \psi\circ \beta$.  
We note that the residue morphisms  of  $\psi$ and $\psi'$ at $o$ are surjective morphisms of complex vector spaces of the same dimension $d$, 
hence they are isomorphisms. 
This implies that  $\beta$ is an isomorphism around $o$. 
Thus 
\begin{equation*} 
  \xymatrixcolsep{3pc}\xymatrixrowsep{3pc}\xymatrix{
   \mathcal{O}_X^{\oplus p} \ar[r]^{\beta^{-1}\circ \varphi}  &  \mathcal{O}_X^{\oplus q}  \ar[r]^{\psi\circ \beta  } & \mathcal{E} \ar[r] & 0 
        } 
\end{equation*}
is also exact. 
The previous paragraphs then imply that the ideal $\mathcal{I}$ defined by  $\psi$ is the same as the one defined by $\psi'$. 
This completes the proof of the first Step.   \\

\textit{Step 2}. 
In this step, we will construct a functorial modification $X_1$ of $(o\in X)$,  
which decreases the difference $M-m$ on $X_1$.   
From the first step, we may assume that $m=0$, \textit{i.e.} $\Theta(o)$ is the null matrix.  
Let $f_1\colon X_1\to X$ be the blowup at the ideal $\mathcal{I}$. 
We can pullback the sequence \eqref{eqn:res-sheaf-01}, 
and obtain an exact sequence of coherent sheaves 
\[
\mathcal{O}_{X_1}^{\oplus p} \overset{f_1^*\varphi}{\longrightarrow} \mathcal{O}_{X_1}^{\oplus q} \longrightarrow f_1^*\mathcal{E} \to 0.
\]  
We note that $f_1^*\varphi$ is represented by the matrix $\Gamma=\Theta\circ f_1$.
Let $o_1\in X_1$ be a point.  
Then there is an open neighborhood $U$ of $o_1$, such that  $f_1^*\mathcal{I}$ is generated by some holomorphic function $t$.  
As we have assumed that $m=0$, the ideal $\mathcal{I}$ is generated by the entries of $\Theta$. 
It follows that every entry of the matrix $\Gamma|_U$ is divisible by $t$, and $t^{-1}\Gamma$ is non zero at $o_1$. 
The matrix $t^{-1}\Gamma$ then defines a morphism $\gamma\colon \mathcal{O}_{X_1}^{\oplus p} \to \mathcal{O}_{X_1}^{\oplus q}$, and we denotes its cokernel by $\mathcal{F}_1$. 
Then  the dimension of the residue of $\mathcal{F}_1$ at $o_1$ is smaller than the one of $\mathcal{E}$ at $o$. 
By construction, 
there is a natural surjective morphism from $f_1^*\mathcal{E}$ to $\mathcal{F}_1$, which is generically isomorphic. 
It induces a surjective and generically isomorphic  
morphism from  $\mathcal{E}_1:= f_1^*\mathcal{E}/\mathrm{torsion}$ 
to $\mathcal{F}_1/\mathrm{torsion}$, 
which must be an isomorphism. 
Thus we obtain a surjective morphism from $\mathcal{F}_1$ to  $\mathcal{E}_1$.  
Hence the dimension of the residue of $\mathcal{E}_1$ at $o_1$ is smaller than the one of $\mathcal{E}$ at $o$. 
This completes the proof of the second step.  
We note that this construction is functorial since the ideal $\mathcal{I}$  depends only on $\mcE$ by the first step.
\\

\textit{Step 3.}
By repeating the procedure  of the second step  for finitely many times, and then we take a functorial resolution of singularities in the end, 
we obtain the required morphism $r\colon  \widetilde{X} \to X$. 
This completes the proof of the theorem. 
\end{proof}

As a corollary, we have the following lemma.   

\begin{lemma}
\label{lemma:orbifold-res-torfree}
Let $Y_{\orb}={(V_i,G_i)}$ be a compact complex orbifold, and $\mcE_{orb} = \{\mcE_i\}$ an orbifold coherent torsion-free sheaf on $Y_{\orb}$. 
Then by successively blowing up, we can obtain a morphism $f_{\orb} \colon X_{\orb} \to Y_{\orb}$ so that the following properties hold. 
\begin{enumerate}
    \item There is an orbifold atlas  $X_{\orb} = {(\widehat{V}_i,G_i)}$ so that $f_{\orb}$ is induced by the morphisms  $\varphi_i\colon \widehat{V}_i \to V_i$, which are  successive blowups of $G_i$-invariant  centers, over the non-locally-free loci of the $\mathcal{E}_i$.
    \item The orbifold coherent sheaf $(f_{\orb}^*\mcE_{orb})/\mathrm{torsion}$, which is defined by the coherent sheaves 
    $(\varphi_i^*\mcE_i)/\mathrm{torsion}$ on $\widehat{V}_i$, is an orbifold vector bundle. 
%    \item For each irreducible components $E$ of the exceptional locus of $f \colon X \to Y$, it corresponds to a smooth orbifold subvariety of $X_{\orb}$. 
\end{enumerate}
\end{lemma}

\begin{proof}
It suffices to  apply   Theorem \ref{thm:func-reso-loc-free} to each coherent sheaf $\mathcal{E}_i$ on $V_i$.   
\end{proof}

\subsection{Miyaoka's inequality }
%\label{section:Miyaoka}

We will prove  the   following proposition, 
which is a variant of \cite[Theorem 7.11]{CHP16}, and its argument goes back to  \cite[Theorem 6.1]{Miy87}. 
It will play an important role in the  positivity results on cotangent sheaves in Section \ref{section:log-cotan}. 

\begin{proposition}
\label{prop:psef-c2-general}
Let $(X,\omega)$ be a compact K\"ahler threefold with quotient singularities only. 
Let $\mcE$ be a reflexive sheaf on $X$ which induces an orbifold vector bundle $\mcE_{\rm orb}$ on the standard orbifold structure $X_{\rm orb}$ of $X$.  
Assume that the following conditions hold:
\begin{enumerate}
\item $\alpha$ is a nef class on $X$ such that $\alpha^{2} \not\num 0$ and $\alpha^3=0$. 
\item $\mcE$ is generically nef with respect to $\alpha_\epsilon $ for all $0<\epsilon\ll 1$, where $\alpha_\epsilon = \alpha+\epsilon \omega$. 
\item $\hat{c}_1(\mcE) = \alpha - \beta$ such that $\alpha \cdot \beta\num 0$, 
\end{enumerate} 
Then $\hat{c}_2(\mcE) \cdot \alpha \> 0$.
\end{proposition}

%Here,  we set $\hat{c}_2(\mcE) = \hat{c}_2(\mcE_{\orb})$. 
We recall that $\mcE$ is generically nef with respect to a class $\gamma$ if  
for any  torsion-free quotient $\mcE \to \mathcal{Q}$ of coherent sheaves, we have $\hat{c}_1(\mathcal{Q})\cdot \gamma^2   \> 0$.
In the proof of \cite[Theorem 6.1]{Miy87}, one needs to consider second Chern classes of torsion-free sheaves. 
To deal with this problem in the setting of orbifolds, we will use  the following assertion, which reduce the situation to Chern classes of orbifold vector bundles.

\begin{lemma}
\label{lemma:A-BG-inequality-orbifold}
Let $(Y_{\orb}, \alpha)$ be a compact K\"ahler orbifold of dimension $n\ge 2$.   
Let $\mcE_{\orb}$  be a  torsion-free orbifold coherent sheaf of rank $r$  on $Y_{\orb}$ which is   $\alpha^{n-1}$-stable. 
Assume that $f_{\orb}\colon X_{\orb}\to Y_{\orb}$ is a  composition of blowups  as in Lemma \ref{lemma:orbifold-res-torfree}. 
In particular,  
$\mcG_{\orb} = (f_{\orb}^*\mcE_{\orb})/{\mathrm{torsion}}$ is an orbifold vector bundle. 
Then we have 
\[
\left(\hat{c}_2(\mcG_{\orb}) - \frac{r-1}{2r}\hat{c}_1(\mcG_{\orb})\right)\cdot (f^*\alpha)^{n-2} \> 0, 
\] 
where $f\colon X\to Y$ is the induced morphism.
\end{lemma}

The idea of the proof is to approximate $f^*\alpha$ by a family of orbifold K\"ahler classes $\{\alpha_\epsilon\}$  on $X_{\orb}$. 
In order to carry out this method, we will show in the next  lemmas,  
that $\mcG_{\orb}$ is $(\alpha_\epsilon)^{n-1}$-stable for all $\epsilon>0$  small enough. 
  Our method adapts those of \cite{GrebKebekusPeternell2016IMRN}, \cite{Toma2016} and \cite{Toma2019}.   
%Before giving the proof of this lemma, we will first establish a few preparatory results.

\begin{lemma}
\label{lemma:modification-stable}   
Let  $f_\orb\colon X'_\orb\to X_{\orb}$ be a proper bimeromorphic morphism between compact complex orbifolds, and let $f\colon X'\to X$ be the induced morphism between quotient spaces. 
Let $\mathcal{E}_{\orb}$ and $\mathcal{E}_{\orb}'$ be torsion-free orbifold coherent sheaves on $X_{\orb}$ and $X_{\orb}'$, respectively. 
Assume that $f_{\orb}^*\mathcal{E}_{\orb}$ is isomorphic to $\mathcal{E}'_{\orb}$ on the isomorphic locus of $f_{\orb}$.  
Then the following properties hold. 
\begin{enumerate}
    \item  There is a correspondence between   saturated subsheaves $\mathcal{F}'_\orb$ of $\mathcal{E}'_\orb$ and  saturated subsheaves $\mathcal{F}_\orb$ of $\mathcal{E}_\orb$,  such that  
    \[
    \hat{c}_1(\mathcal{F}_\orb)\cdot \alpha^{n-1} 
    =
    \hat{c}_1(\mathcal{F}'_\orb)\cdot (f^*\alpha)^{n-1},  
    \]
    where $n=\dim X$ and  $\alpha\in H^2(X,\mathbb{R})$ is any class.   
    \item  $\mathcal{E}_\orb$ is $\alpha^{n-1}$-(semi)stable if and only if $\mathcal{E'}_\orb$ is $(f^*\alpha)^{n-1}$-(semi)stable. 
\end{enumerate}
\end{lemma}

\begin{proof} 
Let $U'\subseteq X'$ be the largest open subset over which $f_{\orb}$ is an isomorphism and let $U=f(U')\subseteq X$. 
Then the complement of $U$ in $X$ has codimension at least $2$. 
By  assumption, the orbifold coherent sheaves $(f_{\orb})_*\mathcal{E}'_{\orb}$ and $\mathcal{E}_{\orb}$ are isomorphic over $U$. 
Thus they have the same reflexive hull which is equal to $\mathcal{E}_{\orb}^{**}$.

We first fix a saturated subsheaf $\mathcal{F}'_\orb$ of $\mathcal{E}'_\orb$. 
There is a natural morphism  $(f_\orb)_*\mathcal{F}'_{\orb} \to ((f_{\orb})_*\mathcal{E}'_{\orb})^{**}= \mathcal{E}_{\orb}^{**}$.  
We then define  $\mathcal{F}_{\orb}$ as the intersection of $\mathcal{E}_{\orb}$ and the image  of $(f_{\orb})_*\mathcal{F}'_\orb$ inside $\mathcal{E}_{\orb}^{**}$.

Conversely, let $\mathcal{F}_\orb$ be a saturated subsheaf of $\mathcal{E}_\orb$.  
Then we can consider the intersection $\mathcal{H}_\orb$ of $\mathcal{F}_{\orb}$ and $(f_{\orb})_*\mathcal{E}'_{\orb}$ inside $\mathcal{E}^{**}_{\orb}$. 
We define $\mathcal{F}'_{\orb}$ as the saturation in $\mathcal{E}_\orb'$ 
of the image of $f_{\orb}^*\mathcal{H}_\orb$.   
Then these constructions are  converse with each other.

Since $f_{\orb}^*\mathcal{F}_{\orb}$ is isomorphic to $\mathcal{F}'_{\orb}$ on   $U'$, 
the image of the class $f^*\hat{c}_1(\mathcal{F}_{\orb})- \hat{c}_1(\mathcal{F}'_\orb)$ in $H^{2}(U',\mathbb{R})$ is zero.  
In other words, it is a class in the relative cohomology group $H^2(X',U',\mathbb{R})$.  
If $[X']\in H_{2n}(X',\mathbb{R})$ is the fundamental class of $X'$, 
then we have 
\begin{eqnarray*} 
& &(f^*\hat{c}_1(\mathcal{F}_{\orb})- \hat{c}_1(\mathcal{F}'_\orb) )\cdot (f^*_\orb \alpha)^{n-1} \\
&=& ( ( f^*\hat{c}_1(\mathcal{F}_{\orb})- \hat{c}_1(\mathcal{F}'_\orb) )\smallsmile (f^*_\orb \alpha)^{n-1} ) \smallfrown [X']  \\ 
&=& \alpha^{n-1} \smallfrown  f_*( (f^*\hat{c}_1(\mathcal{F}_{\orb})- \hat{c}_1(\mathcal{F}'_\orb)) \smallfrown [X'])  \\ 
&=& 0,  
\end{eqnarray*}
since $f_*( (f^*\hat{c}_1(\mathcal{F}_{\orb})- \hat{c}_1(\mathcal{F}'_\orb)) \smallfrown [X']) \in H_{2n-2}(X\setminus U, \mathbb{R}) = \{0\}$ 
(for example, see  Section \ref{subsec:BM-homology}).   
This completes the proof of the item (1).

Finally, we note that the ranks of $\mathcal{F}_{\orb}$ and $\mathcal{F}'_{\orb}$ are the same  in the previous correspondence.  
Hence the item (2) also follows.  
\end{proof}

The following lemma is a key ingredient for our openness of stability. 

\begin{lemma}
\label{lemma:bounded-c1} 
Let $X_{\orb}$ be an effective compact complex orbifold with quotient space $X$, let $\omega$ be an orbifold  K\"ahler class, and let $n$ be the dimension of $X$. 
Assume that $\mathcal{E}_{\orb}$ is a torsion-free  orbifold  coherent   sheaf  on $X_{\orb}$, 
and that $C$ is a real number. 
Then the following subset of $H^2(X,\mathbb{R})$ is finite 
\[
\{
\hat{c}_1(\mathcal{F}_{\orb}) \ : \  0\neq \mathcal{F}_{\orb} \subseteq \mathcal{E}_{\orb} \mbox{ is  saturated    with }  \hat{c}_1(\mathcal{F}_{\orb}) \cdot \omega^{n-1} \ge C
\}.
\]
\end{lemma}

\begin{proof}
The proof is to descend $\mathcal{E}_{\orb}$ to the quotient space $X$ and then apply the boundedness results of \cite{Toma2019}.   
We fix a  finite atlas of chart $X_{\orb}=\{(V_i,G_i,\pi_i)\}$.  
Let $q$ be the lcm of the cardinalities   $|G_i|$. 
Let  $U\subseteq X$ be the open subset of points over which the  actions of $G_i$  are free, 
and let $E_1,...,E_s$ be the irreducible components of $X\setminus U$ of codimension  one. 
Then, by taking the preimages  $\pi_i^{-1}(E_j)$ in the orbifold charts,  
each $E_j$ defines an orbifold subvariety $(E_j)_{\orb}$ of $X_{\orb}$, of codimension one. 
In particular, the  ideal sheaf $\mathcal{I}_j$ of $(E_j)_{\orb}$  is an orbifold invertible coherent ideal sheaf. 
We set  $\hat{c}_1(E_j) := - \hat{c}_1(\mathcal{I}_j)$. 
\\

\textit{Step 1.} In this step, we prove the following statement. 
Let $\mathcal{F}_{\orb}$ be a torsion-free coherent  sheaf on $X_{\orb}$ of rank $r$,  
let $\mathcal{G}$ be the coherent sheaf on $X$ defined by $((\pi_i)_*\mathcal{F}_i)^{G_i}$, 
and let  $\mathcal{G}_{orb}$ be the torsion-free coherent sheaf on $X_{\orb}$ defined by $(\pi_i)^*\mathcal{G}/\mathrm{torsion}$.   
Then the rank of $\mathcal{G}$ is equal to $r$ as well.  
Since $\mathcal{F}_{\orb}$ is torsion-free, we obtain a natural morphism from $\mathcal{G}_{\orb}$ to $\mathcal{F}_{\orb}$, which is an isomorphism over $U$.  
Then 
\[
\hat{c}_1(\mathcal{F}_{\orb}) = \hat{c}_1(\mathcal{G}_{\orb}) +  \sum_{j=1}^s d_j \hat{c}_1(E_j),
\]
where $d_j$ is the vanishing order of the natural morphism 
$\det \mathcal{G}_{\orb} \to \det \mathcal{F}_{\orb}$ along $(E_j)_{\orb}$. 
We claim that  $0\le  d_1,...,d_s \le r(q-1)$.

To prove this assertion, it is sufficient to work locally around a general point $x$ of $E_j$.  
Thus we may assume that, around $x$, the sheaf $\mathcal{G}$ is   free.  
Let $(V,G,\pi)$ be an orbifold chart over a neighborhood of $x$ such that $\pi^{-1}(\{x\}) = \{x'\}$ is a singleton.  
We set $D_j= \pi^{-1}(E_j)$. 
Since $x$ is a general point of $E_j$, we may further assume that $\mathcal{F}_{\orb}$ corresponds to some free coherent sheaf $\mathcal{F}$ on $V$.
By Lemma \ref{lemma:determinant-compute} at the end of this subsection, it is enough to show that, for any local section $\eta$ of $\pi^*\mathcal{G}$, 
which does not vanish at $x'$,   
the vanishing order of $\varphi(\eta)$  along $D_j$ is at most $|G|-1\le q-1$, 
where $\varphi\colon \pi^*\mathcal{G} \to \mcF$ is the natural morphism.  
We remark that, if $\eta'$ is the sum of the $G$-orbit of $\eta$, 
then it does not vanish at  $x'$, 
for the action of $G$ on the fiber $(\pi^*\mathcal{G})_{x'}$ is trivial. 
Furthermore,  the vanishing order of $\varphi(\eta')$ along $D_j$ is at least the one of $\varphi(\eta)$.   
Hence, it is enough to show that, for any local section $\sigma$ of $\mathcal{G}$, 
which does not vanish at  $x$,   
the vanishing order of  $\varphi(\pi^*\sigma)$ along $D_j$ is at most $|G|-1$.  
Up to shrinking $V$,  
there is a $G$-invariant holomorphic function $u$ on $V$ which vanishes along $D_j$ with order   equal to $|G|$. 
If the vanishing order of $\varphi(\pi^*\sigma)$ along $D_j$  is at least $|G|$, 
then $ u^{-1}\cdot \varphi(\pi^*\sigma)$ is a $G$-invariant section of $\mathcal{F}$. 
Thus it induces a section $\sigma'$ of $(\pi_*\mathcal{F})^G$, which is a local section of  $\mcG$ around $x$.  
Furthermore, if $v$ is the local holomorphic function around $x$ induced by $u$, then it vanishes along $E_j$ and we have  $\sigma' = \sigma \cdot v^{-1}$. 
We obtain  a contradiction since $\sigma$ does not vanish at  $x$.  
This completes the proof of Step 1. 
\\

\textit{Step 2}. 
In this step, we reduce the problem to the case of sheaves on $X$.  
We recall that, by Lemma \ref{lemma:construction-orb-kahler}, 
$\omega$ is also a K\"ahler class on the quotient space $X$. 
We define
\[
\Sigma_1:=\{0 \neq \mathcal{F}_{\orb} \ : \  \mathcal{F}_{\orb} \subseteq \mathcal{E}_{\orb} \mbox{ is  saturated    with }  \hat{c}_1(\mathcal{F}_{\orb}) \cdot \omega^{n-1} \ge C\}, 
\]
\[
\Sigma_2:=\{ \mbox{coherent sheaves }  \mathcal{G}   \mbox{ on } X \mbox{ defined by } ((\pi_i)_*\mathcal{F}_{\orb})^{G_i}  \mbox{ for }   \mathcal{F}_{\orb} \in \Sigma_1 \}.  
\]
Thanks to Step 1, it is sufficient to show that the  orbifold first Chern classes of the torsion-free orbifold coherent sheaves $\mathcal{G}_{\orb}$ induced by the elements $\mathcal{G}$ in $\Sigma_2$ are finite.   
We note that $((\det \mathcal{G})^{\otimes q})^{**}$ is  locally free on $X$, and we have 
\[
c_1(((\det \mathcal{G})^{\otimes q})^{**}) =  \hat{c}_1(((\det \mathcal{G})^{\otimes q})^{**}) = q \cdot  \hat{c}_1( \mathcal{G}_{\orb}  ), 
\]
as shown in (6) and (7) of Remark \ref{rmk:quotient-singularities}. 
Step 1 also implies that  the intersection numbers  $c_1(((\det \mathcal{G})^{\otimes q})^{**}) \cdot \omega^{n-1}$, for $\mathcal{G} \in \Sigma_2$, are bounded from below by a constant.   

Let $\mathcal{E}'$ be the coherent sheaf on $X$ defined by $((\pi_i)_*\mathcal{E}_{\orb})^{G_i}$.  
Then for any $\mathcal{G}\in \Sigma_2$, it is a subsheaf of $\mathcal{E}'$,  saturated  over $U$. 
We define 
\[
\Sigma_3 = \{ \mathcal{H} \ : \  \mathcal{H} \mbox{ is the saturation of } \mathcal{G} \mbox{ in } \mathcal{E}' \mbox{ where }  \mathcal{G} \in \Sigma_2\}.  
\]
Since there is  a natural  morphism from $\mathcal{G}$ to its saturation $\mathcal{H}$, which is isomorphic over $U$, 
we see that 
\[
c_1(((\det \mathcal{H})^{\otimes q})^{**}) = c_1(((\det \mathcal{G})^{\otimes q})^{**}) + \sum_{j=1}^s a_j \hat{c}_1(E_j)
\]
for some integers $a_j \ge 0$.  
Since  $\hat{c}_1(E_j) \cdot \omega^{n-1}>0$ for all $j=1,...,s$,  
 the numbers $c_1(((\det \mathcal{H})^{\otimes q})^{**}) \cdot \omega^{n-1}$, for $\mathcal{H} \in \Sigma_3$, are bounded from below by a constant.    
Furthermore,  it is   sufficient to show that the  first Chern classes $c_1(((\det \mathcal{H})^{\otimes q})^{**})$ for  $\mathcal{H} \in \Sigma_3$ are finite. 
\\

\textit{Step 3}.  
We completes the proof in this step. 
For any coherent sheaf $\mathcal{H}\in \Sigma_3$ of rank $r$, we see that 
\[
((\det \mathcal{H})^{\otimes q})^{**} \subseteq  ((\bigwedge^r \mathcal{E}')^{\otimes q})^{**}
\]
is a saturated subsheaf, 
since the inclusion is saturated in codimension 2, 
and since $((\det \mathcal{H})^{\otimes q})^{**}$ is a line bundle.     
In order to apply the results of  \cite{Toma2019}, we need to  show that,     
the $(n-1)$-degree $\deg_{n-1}(\mathcal{L})$ of $\mathcal{L} := ((\det \mathcal{H})^{\otimes q})^{**}$ with respect to $\omega^{n-1}$,  
in the sense of \cite[Definition 2.5]{Toma2019}, 
is bounded from below by a constant,  
for all $\mathcal{H}\in \Sigma_3$ of rank $r$.  

We fix   a desingularization   $\rho\colon \widetilde{X} \to X$.   
Since $X$ has rational singularities, we have $R^k\rho_*(\rho^*\mathcal{L}) = 0$ for $k>0$, 
and hence  $\rho_! (\rho^*\mathcal{L}) = \mathcal{L}$ in the Grothendieck group $K_0(X)$ generated by coherent sheaves. 
If $\tau$ signifies the homological Todd class  as  in \cite[Section 2.2]{Toma2019}, 
then the component $\tau_{n-1}(\mathcal{L}) \in H_{2n-2}(X, \mathbb{R})$ satisfies
\[
\tau_{n-1}(\mathcal{L}) 
=  \tau_{n-1}(\rho_! (\rho^*\mathcal{L}))   
= \rho_*(\tau_{n-1} ( \rho^*\mathcal{L} )).    
\] 
Since $\widetilde{X}$ is smooth, we have 
$\tau_{n-1} ( \rho^*\mathcal{L} ) =  
(c_1(\rho^*\mathcal{L}) + \frac{1}{2}c_1(\widetilde{X}))  \smallfrown [\widetilde{X}]$, 
where  $[\widetilde{X}] \in H_{2n}(\widetilde{X}, \mathbb{R})$ is the fundamental class.   
In particular, by the projection formula, we have 
\[
\deg_{n-1}(\mathcal{L}) 
=  \tau_{n-1}(\mathcal{L}) \smallfrown \omega^{n-1}  
=   c_1(\mathcal{L}) \cdot \omega^{n-1}  + \frac{1}{2}c_1(\widetilde{X}) \cdot (\rho^*\omega)^{n-1}.
\] 
This implies the boundedness on the $(n-1)$-degrees that we require.

Since $\mathcal{L}$ is saturated in $ ((\bigwedge^r \mathcal{E}')^{\otimes q})^{**}$,  
by \cite[Lemma 5.7]{Toma2019}, the set of quotient sheaves
\[
\{  ((\bigwedge^r \mathcal{E}')^{\otimes q})^{**} / \mathcal{L}  : \mathcal{H}\in \Sigma_3 \mbox{ of rank } r \}
\] 
is bounded in the sense of \cite[Definition 5.1]{Toma2019}. 
Using flattening and Noetherian induction (see \cite[Section 3]{Toma2016}), we deduce from \cite[Proposition 2.6]{Toma2019} that the set of their homological Todd classes is finite.  
From the computation above, it follows that the 
first Chern classes $ c_1(((\det \mathcal{H})^{\otimes q})^{**})$ for $\mathcal{H}\in \Sigma_3$ with rank $r$ are finite.  
Since $r$ is an integer between $1$ and rank of $\mathcal{E}_{\orb}$,    
this completes the proof of the lemma.  
\end{proof}

\begin{lemma}
\label{lemma:A-perturbation-stability}
Let $f_{\orb}\colon (X_{\orb}, \omega)  \to (Y_{\orb},\alpha)$ be a  proper bimeromorphic morphism between effective compact K\"ahler orbifolds. 
Let $\alpha_\epsilon= f^*\alpha+\epsilon \omega$ for all $\epsilon >0$. 
Assume that $\mcG_{\orb}$ is an orbifold coherent torsion-free sheaf on $X_{\orb}$ which  is $(f^*\alpha)^{n-1}$-stable.
Then  it  is $(\alpha_\epsilon)^{n-1}$-stable for all $\epsilon>0$ small enough.
\end{lemma}

\begin{proof}  
We claim that there is a number $C>0$, such that 
\[
\mu_{(f^*\alpha)^{n-1}}(\mcF_{\orb}) \< \mu_{(f^*\alpha)^{n-1}}(\mcG_{\orb}) - C. 
\]
for any non zero proper saturated orbifold subsheaf
$\mcF_{\orb} \subseteq \mcG_{orb}$.
To prove this claim, by considering $(f_*\mcG)/\mathrm{torsion}$ on  $Y_{\orb}$, 
and by applying Lemma \ref{lemma:modification-stable}, 
we may assume that $f$ is the identity map.  
In particular, $f^*\alpha = \alpha$ is a K\"ahler class. 
Then by Lemma \ref{lemma:bounded-c1}, there are only finitely many possible values for the slopes $\mu_{\alpha^{n-1}}(\mathcal{F}_{\orb})$ which are greater than $\mu_{\alpha^{n-1}}(\mathcal{G}_{\orb}) - 1$, where $\mathcal{F}_{\orb}$ are saturated subsheaves of $\mathcal{G}_{\orb}$. 
The existence of $C$ then follows from the stability of $\mathcal{G}_{\orb}$.  

By developing $(\alpha_\epsilon)^{n-1} = (f^*\alpha+\epsilon \omega)^{n-1}$, 
and by Lemma \ref{lemma:A-bounded-above-v2},   
we deduce that, for all $\epsilon>0$ small enough, 
for all saturated subsheaves $\mathcal{F}_{\orb}$ of $\mathcal{G}_{\orb}$, 
we have 
\[
\mu_{(\alpha_\epsilon)^{n-1}}(\mcF_{\orb})  \< 
\mu_{(f^*\alpha)^{n-1}}(\mcF_{\orb}) + \frac{1}{3} C. 
\]
Moreover,  for such $\epsilon$ small enough, we have 
\[
\mu_{(\alpha_\epsilon)^{n-1}}(\mcG_{\orb}) \> \mu_{(f^*\alpha)^{n-1}}(\mcG_{\orb}) - \frac{1}{3} C. 
\]
This completes the proof of  the lemma.
\end{proof}

We can now deduce Lemma \ref{lemma:A-BG-inequality-orbifold}.

\begin{proof} [{Proof of Lemma \ref{lemma:A-BG-inequality-orbifold}}]
There is an 
orbifold K\"ahler class on $X_{\orb}$ of the shape 
\[
\omega=f^*\alpha -\sum a_k \hat{c}_1( \mcO_{X}(E_k)),
\]
where the $E_k$'s are the irreducible components of $f$-exceptional locus, and the $a_k$'s are positive rational numbers.  
Let $\alpha_\epsilon = f^*\alpha + \epsilon \omega$ for $\epsilon >0$.  
Then $\alpha_\epsilon$ is an orbifold K\"ahler class on  $X_{\orb}$. 
We note that $\mcG_{\orb}$ is $(f^*\alpha)^{n-1}$-stable by Lemma \ref{lemma:modification-stable}. 
Hence it is $(\alpha_\epsilon)^{n-1}$-stable for $\epsilon>0$ small enough by Lemma \ref{lemma:A-perturbation-stability}. 
By \cite[Theorem 1]{Faulk2022}, $\mcG_{\orb}$ admits an orbifold Hermitian-Einstein metric with respect to $\alpha_\epsilon$. 
Hence by a similar argument as in  L\"ubke's theorem in \cite{Lub82}  (see also \cite[Theorem 4.4.7]{Kob14}), we deduce that
\[
\left(\hat{c}_2(\mcG_{\orb}) - \frac{r-1}{2r}\hat{c}_1(\mcG_{\orb})\right)\cdot \alpha_\epsilon^{n-2} \> 0.
\] 
We can then conclude the lemma by taking the limit as $\epsilon\to 0$.
\end{proof}

Now we  are ready to prove   Proposition \ref{prop:psef-c2-general}. 

\begin{proof}
[{Proof of Proposition \ref{prop:psef-c2-general}}] 
We first remark the following estimate, 
\[
 \frac{1}{\alpha_\epsilon^3} \cdot (\alpha - \beta) \cdot \alpha_\epsilon^{2} 
 = \frac{2 \epsilon \cdot \alpha^{2} \omega + \epsilon^2 \cdot  (\alpha-\beta)\omega^2}{3\epsilon \cdot \alpha^{2}\omega + 3\epsilon^2 \cdot \alpha \omega^2 + \epsilon^3 \cdot \omega^3} 
 \sim_{\epsilon \to 0} \frac{ 2 \epsilon \cdot \alpha^{2} \omega }{3\epsilon \cdot \alpha^{2}\omega} = \frac{2}{3}.
\] 
Therefore, if $\epsilon>0$ is sufficiently small, we have  
\[\frac{1}{\alpha_\epsilon^3} \cdot (\alpha - \beta) \cdot \alpha_\epsilon^{2} \< 1.\]

We fix some $\epsilon>0$   sufficiently small so that the previous inequality holds.
Since $\alpha_\epsilon$ is a K\"ahler class on $X$, it is  an orbifold K\"ahler class on $X_{\orb}$ by Lemma \ref{lemma:construction-orb-kahler}. 
By taking the Harder-Narasimhan filtration of $\mcE_{\orb}$, and then taking a Jordan-H\"older filtration on each subquotient, we obtain  a  filtration of  orbifold reflexive coherent sheaves
\[ 0 = \mcF^0_{\orb} \subseteq  \cdots \subseteq \mcF^k_{\orb} = \mcE_{\orb} \]
such that each subquotient $\mcG^j_{\orb}:=\mcF^j_{\orb}/\mcF^{j-1}_{\orb}$ is $(\alpha_\epsilon)^{2}$-stable. 
Furthermore, if $\delta_{\epsilon,j}$ is the slope of $\mcG_{\orb}^j$ with respect to $(\alpha_\epsilon)^{2}$, then we have
\[
\delta_{\epsilon,1} \> \cdots \> \delta_{\epsilon, k}.
\]
Since $X_{\orb}$ is standard, $\mcF^{k-1}_\orb$ induces a coherent  subsheaf $\mcF^{k-1}\subseteq \mcE$, such that $\hat{c}_1(\mcG_{\orb}^k) = \hat{c}_1(\mcE) - \hat{c}_1(\mcF^{k-1})$, see (5) of Remark \ref{rmk:quotient-singularities}. 
Since $\mcE$ is generically nef with respect to $\alpha_\epsilon$, we deduce that $\delta_{\epsilon,k} \> 0$.  

By applying Lemma \ref{lemma:orbifold-res-torfree} to the subquotients $\mcG^k_\orb,..., \mcG^1_\orb$ successively, we obtain  an orbifold $X'_{\orb}$ with the natural projection  $f_{\orb}\colon  X'_{\orb} \to X_{\orb}$  so that
we have an induced filtration
\[ 0 = \mcF'^0_{\orb} \subseteq  \cdots \subseteq  \mcF'^k_{\orb} = \mcE'_{\orb} := f_{\orb}^*\mcE_{\orb}\]
such that each  subquotient $\mcG'^j_{\orb} = \mcF'^j_{\orb}/\mcF'^{j-1}_{\orb}$ is  an orbifold vector bundle.  
Here, each $\mathcal{F}'^j_\orb$ is the saturation of $f_\orb^*\mathcal{F}^j_\orb$ inside $f_\orb^*\mathcal{E}_\orb$.
By Lemma \ref{lemma:A-BG-inequality-orbifold} and Lemma \ref{lemma:modification-stable}, we have 
\[  
\left(\hat{c}_2(\mcG'^j_{\orb}) - \frac{r_j-1}{2r_j}\hat{c}_1(\mcG'^j_{\orb} )^2 \right) \cdot f^*\alpha_\epsilon\> 0, 
\] 
where  $r_j$ is the rank of $\mcG^j_{\orb}$, and $f\colon X'\to X$ is the induced morphism on quotient spaces. 
Furthermore, we have $ \hat{c}_1(\mcG_{\orb}'^{j}) \cdot (f^*\alpha_\epsilon)^{2}
=
\hat{c}_1(\mathcal{G}_\orb^j) \cdot \alpha_{\epsilon}^2. $

As a consequence, we get   
\begin{eqnarray*}
& & \hat{c}_2( \mcE ) \cdot  \alpha_\epsilon  \\
&=& \hat{c}_2(f^*_{\orb}\mcE_{\orb} ) \cdot f^*\alpha_\epsilon  \\
&=& \left(\sum_j \hat{c}_2(\mcG_{\orb}'^j ) + \sum_{j<l}\hat{c}_1(\mcG_{\orb}'^j ) \hat{c}_1(\mcG_{\orb}^{'l} ) \right)\cdot f^*\alpha_\epsilon   \\
&\> &  \frac{1}{2} \left( \sum_j \hat{c}_1(\mcG_{\orb}'^j)^2 + 2\sum_{j<l}\hat{c}_1(\mcG_{\orb}'^j ) \hat{c}_1(\mcG_{\orb}^{'l} )  
-  \sum_j \frac{1}{r_j}\hat{c}_1(\mcG_{\orb}'^j )^2  \right) \cdot f^*\alpha_\epsilon  \\
&=&\frac{1}{2} \left(\hat{c}_1(f_{\orb}^*\mcE_{\orb})^2 -  \sum_j \frac{1}{r_j}\hat{c}_1(\mcG_{\orb}'^j )^2 \right) \cdot
f^*\alpha_\epsilon  \\
&=& \frac{1}{2} \hat{c}_1(\mcE)^2 \cdot \alpha_{\epsilon} 
     - \frac{1}{2}\left( \sum_j \frac{1}{r_j}\hat{c}_1(\mcG_{\orb}'^j )^2 \right)  
     \cdot  f^*\alpha_\epsilon 
\end{eqnarray*}
Applying the Hodge index theorem (see Lemma \ref{lemma:hodge-index} below) on  $X'$, and noting that $(f^*\alpha_\epsilon)^3= \alpha_\epsilon^3$,  we obtain that  
\[
\alpha_\epsilon^{3} \cdot \left(\hat{c}_1(\mcG_{\orb}'^j) ^2 \cdot f^*\alpha_\epsilon  \right) 
\<    \left(\hat{c}_1(\mcG_{\orb}'^j) \cdot (f^*\alpha_\epsilon)^{2}\right)^2
=
(\hat{c}_1(\mathcal{G}_\orb^j) \cdot \alpha_{\epsilon}^2)^2. 
\] 
Hence we deduce that 
\begin{eqnarray*}
\left(\sum_j \frac{1}{r_j}\hat{c}_1(\mcG_{\orb}'^j )^2 \right)\cdot f^*\alpha_\epsilon 
&\<& \frac{1}{\alpha_\epsilon^3}\sum_j \frac{1}{r_j} \left(\hat{c}_1(\mcG_{\orb}'^j) \cdot (f^* \alpha_\epsilon)^{2}\right)^2 \\
&=& \frac{1}{\alpha_\epsilon^3} \sum_j r_j\delta_{\epsilon, j}^2 \\
&\<&   \frac{\delta_{\epsilon,1}}{\alpha_\epsilon^3} \sum_j r_j\delta_{\epsilon, j} \\
&=&  \frac{\delta_{\epsilon,1}}{\alpha_\epsilon^3} \cdot \hat{c}_1(\mcE_{\orb}) \cdot \alpha_\epsilon^{2}\\
 &=&   
 \frac{\delta_{\epsilon,1}}{\alpha_\epsilon^3} \cdot (\alpha - \beta) \cdot \alpha_\epsilon^{2}.
\end{eqnarray*}
%Since $\alpha^3=0$, $\alpha\beta \num 0$ and $\omega\alpha^{2} \neq 0$, we see that 
%\[
 %\frac{1}{\alpha_\epsilon^3} \cdot (\alpha - \beta) \cdot \alpha_\epsilon^{2} 
 %= \frac{2 \epsilon \cdot \alpha^{2} \omega + \epsilon^2 \cdot  (\alpha-\beta)\omega^2}{3\epsilon \cdot \alpha^{2}\omega + 3\epsilon^2 %\cdot \alpha \omega^2 + \epsilon^3 \cdot \omega^3} 
 %\sim_{\epsilon \to 0} \frac{ 2 \epsilon \cdot \alpha^{2} \omega }{3\epsilon \cdot \alpha^{2}\omega} = \frac{2}{3}.
%\] 
Since $\epsilon>0$ is small enough, from the first paragraph, %and since $\delta_{\epsilon,1} \> 0$,  
we obtain that 
\[
\left(\sum_j \frac{1}{r_j}\hat{c}_1(\mcG_{\orb}'^j )^2 \right)\cdot f^*\alpha_\epsilon 
\< \delta_{\epsilon,1} \cdot  \frac{ (\alpha - \beta) \cdot \alpha_\epsilon^{2} }{\alpha_\epsilon^3} 
\< \delta_{\epsilon, 1} \cdot 1  = \delta_{\epsilon,1}. 
\]
 Since  $ \sum_j r_j \delta_{\epsilon,j} =  \hat{c}_1(\mcE ) \cdot \alpha_\epsilon^2 $ and  since $\delta_{\epsilon,j} \> 0$ for $j=1,...,k$, we deduce that $ r_1 \delta_{\epsilon,1} \< \hat{c}_1(\mcE ) \cdot \alpha_\epsilon^2$. 
Hence 
\[
\left(\sum_j \frac{1}{r_j}\hat{c}_1(\mcG_{\orb}'^j )^2 \right)\cdot f^*\alpha_\epsilon 
\< \frac{1}{r_1} \hat{c}_1(\mcE) \cdot \alpha_\epsilon^{2}   
\< \hat{c}_1(\mcE) \cdot \alpha_\epsilon^{2}.
\]
Therefore, we get
\[
\hat{c}_2(\mcE) \cdot \alpha_\epsilon  
=
\hat{c}_2(f_{\orb}^*\mcE_{\orb}) \cdot f^*\alpha_\epsilon  
\> 
\frac{1}{2} \hat{c}_1(\mcE )^2 \cdot \alpha_\epsilon - \frac{1}{2} \hat{c}_1(\mcE ) \cdot \alpha_\epsilon^2 
\]
By tending $\epsilon$ to zero, we conclude that 
\[
\hat{c}_2(\mcE) \cdot \alpha  
\> 
\frac{1}{2}  \hat{c}_1(\mcE)^2 \cdot \alpha - \frac{1}{2 }\hat{c}_1(\mcE) \cdot \alpha^2  = 0.
\]
This completes the proof of the proposition. 
\end{proof}

The following version of Hodge index theorem was used in the previous proof.

\begin{lemma}
\label{lemma:hodge-index}
Let $X$ be a compact K\"ahler variety of dimension $n$ with rational singularities and let $\alpha \in H^2(X, \mathbb{R})$ be a nef class.
Then for any $\beta \in H^2(X, \mathbb{R})$, we have 
\[
\alpha^n \cdot (\beta^2 \cdot \alpha^{n-2}) \< ( \beta \cdot \alpha^{n-1} )^2.
\]
\end{lemma}

\begin{proof}
Let $r\colon \widetilde{X}  \to  X$ be a desingularization such that there is a K\"ahler class $\omega$ on $\widetilde{X}$.  
Let $\alpha_\epsilon = r^*\alpha+\epsilon \omega$ for $\epsilon >0$. 
Then each $\alpha_\epsilon$ is a K\"ahler class on $\widetilde{X}$. 
The classic Hodge index theorem implies that 
\[
\alpha_\epsilon^n \cdot ((r^*\beta)^2 \cdot \alpha_\epsilon^{n-2}) \< ( r^*\beta \cdot \alpha_\epsilon^{n-1} )^2.
\]
Then the lemma follows by taking the limit as $\epsilon\to 0$.
\end{proof}

The following  computation was used in the proof of Lemma \ref{lemma:bounded-c1}. 

\begin{lemma}
\label{lemma:determinant-compute}
Let $X$ be a  complex manifold and let $D\subseteq X$ be a prime divisor. 
Let $\varphi\colon \mathcal{G} \to \mathcal{F}$ be a morphism of free  coherent sheaves of rank $r$ on $X$,  which is generically isomorphic. 
We can choose  a basis $(\sigma_1,...,\sigma_r)$  of $\mathcal{G}$   around a  general point  of $D$, where $\sigma_1,...,\sigma_r$ are local sections of $\mathcal{G}$, 
such that it  satisfies  the following inductive properties. 
\begin{enumerate}
    \item $\varphi(\sigma_1)$ has the highest vanishing order  along $D$. 
    \item For any $j=1,...,r-1$, if $\sigma_1,...,\sigma_j$ have been chosen, 
          then $\sigma_{j+1}$ is chosen so that   $\varphi(\sigma_{j+1})$ has the highest vanishing order  along $D$.
\end{enumerate}
Let $d_1,...,d_r$ be the vanishing orders of $\varphi(\sigma)_1,...,\varphi(\sigma_r)$ along $D$, 
and let  
\[\det \varphi \colon \det \mathcal{G} \to \det \mathcal{F} \] 
be the natural morphism of line bundles.  
Then the vanishing order of $\det \varphi $ along $D$ is equal to $d_1+\cdots+d_r$.  
\end{lemma}

\begin{proof}   
We first remark that such a choice of basis always exists. 
Indeed, for every section $\sigma$ of $\mathcal{G}$, which does not vanish along $D$,  
the vanishing order of $\varphi(\sigma)$ along $D$ is at most equal to the one of $\det \varphi $. 

We will prove by induction on $r$. The lemma is true if $r=1$.  
Assume that it holds for ranks smaller than $r$. 
Let $\delta_j=\varphi(\sigma_j)$ for $j=1,...,r$.  
We are free to shrink $X$ around a general point  of $D$ in the following argument. 
In particular, we may assume that $D$ is defined by an equation $f=0$ for some holomorphic function $f$ on $X$, and that $\varphi$ is isomorphic outside $D$.   
Then $f^{-d_1} \cdot \delta_1$ is a   section of $\mathcal{F}$, which does not vanish along $D$. 
We can consider the following morphism $\psi$
\[
\mathcal{E}  = \mathcal{O}_X \cdot e_1 \oplus \cdots \oplus \mathcal{O}_X \cdot e_r \to \mathcal{F}
\]
such that $\psi(e_1) = f^{-d_1} \cdot \delta_1$ and $\psi(e_j) = \delta_j$ for $j=2,...,r$. 
There is a natural morphism $\eta$ from $\mathcal{G}$ to $\mathcal{E}$ which associates $\sigma_1$ to $f^{d_1}\cdot e_1$ and $\sigma_j$ to $e_j$ for $j=2,...,r$.  
Then we have $\varphi = \psi \circ \eta$, 
and the vanishing  order along $D$ of $\det \varphi $ is equal to the one of $\det \psi$ plus $d_1$.  

We note that 
\[\mathcal{F}' := \mathcal{F}/<f^{-d_1} \cdot \delta_1> \mbox{ and } \mathcal{G'} :=  \mathcal{O}_X \cdot \sigma_2 \oplus \cdots \oplus \mathcal{O}_X \cdot \sigma_r 
\cong \mathcal{G}/<\sigma_1> 
\] are free coherent sheaves  around a general point of $D$. 
Moreover, there is a  natural  morphism  $\varphi' \colon \mathcal{G}' \to \mathcal{F}' $  induced by $\varphi$, which is  an isomorphism outside $D$.  
The vanishing order along $D$ of $\det \varphi'$ is equal to the one of $\det \psi$.

Let  $e$ be a section of $\mathcal{G'}$ which does not vanish along $D$,  
let $\delta = \varphi(e)$ be the section of $\mathcal{F}$, 
and let $d$ be its vanishing order along $D$.  
Then $d\le d_1$ by the condition (1) on $\sigma_1$. 
The vanishing order   $d' $ of $ \varphi'(e)$ along $D$ is equal to 
\[
\max \{ \mbox{vanishing order along } D \mbox{ of } \delta  + h \cdot f^{-d_1} \cdot \delta_1 \ : \ h \mbox{ holomorphic around } D \}.  
\]
We claim that $d'=d$. 
To prove the claim, we take a holomorphic function $h$ around   $D$, 
so that the vanishing order of $\delta' := \delta + h\cdot f^{-d_1} \cdot \delta_1$ is equal to  $d'$. 
Assume by contradiction that $d' > d$.  
Then  the vanishing order   of $h$ along $D$ is equal to  $d$. 
In particular,  $h^{-1} \cdot  f^{d_1}$ is a holomorphic function around general points of $D$. 
Then  we have   
\[
h^{-1} \cdot  f^{d_1} \cdot \delta'   
= \delta_1 +  h^{-1} \cdot  f^{d_1} \cdot \delta   
= \varphi(\sigma_1  + h^{-1} \cdot  f^{d_1} \cdot  e)
\]
The LHS above is a section of $\mathcal{F}$ with vanishing order  
$d_1 + d'-d > d_1$   
along $D$. 
The RHS shows that this section is in the image of $\varphi$.  
Moreover, $\sigma_1  + h^{-1} \cdot  f^{d_1} \cdot  e$   does not vanish along $D$, for $ (\sigma_1,...,\sigma_r)$ is a basis $\mcG$. 
This contradicts   the condition (1) on $\sigma_1$.  

The claim then implies that the basis $(\sigma_2,...,\sigma_r)$ of $\mathcal{G}'$ satisfies the two conditions of the lemma, 
with respect to the morphism $\varphi'$. 
By induction hypothesis, the vanishing order of  $\det \varphi'$ along $D$ is equal to $d_2+\cdots + d_r$. 
This completes the induction and the proof of the lemma.  
\end{proof}

\section{Comparison of Chern classes} 
\label{section:compare-Chern}

Consider a standard complex  orbifold  $X_{\orb}$ and its quotient space $X$. 
We have seen in (5) of Remark \ref{rmk:quotient-singularities} that every orbifold vector bundle $\mcE_{\orb}$ induces a reflexive sheaf $\mcE$ on $X$.  
The objective of this section is to compare the orbifold Chern classes of $ \mcE_{\orb} $ and the  Chern classes of $\widetilde{\mcE}$, where $\widetilde{\mcE}$ is a vector bundle on  a desingularization $r\colon \widetilde{X} \to X$ such that $(r_*\widetilde{\mcE})^{**} = \mcE$.

\subsection{Relative de Rham cohomology}
Let $M$ be a $\mathcal{C}^\infty$ manifold. 
We denote by $\mathscr{D}'^i$  the sheaf of real currents of degree $i$, i.e.
 differential $i$-forms with  real distribution coefficients. 
Then there is a complex 
\[ \cdots \to D^{i}(M) \overset{\mathrm{d}}{\to} D^{i+1}(M) \to \cdots,
\] 
where $D^i(M) = \Gamma(M, \mathscr{D}'^i(M))$ is the group of   global sections of the sheaf $\mathscr{D}'^i$. 
The cohomology groups of the previous complex are denoted by  $H^{i}_{\mathscr{D}'}(M, \mathbb{R})$. 
Then there are natural isomorphisms $H^{i}_{\mathscr{D}'}(M, \mathbb{R}) \cong H^{i}(M, \mathbb{R})$ known as the  de Rham isomorphisms. 
We consider the following complex
\[ \cdots \to D^{i}(M,U) \overset{\mathrm{d}}{\to} D^{i+1}(M,U) \to \cdots  \]
where $U\subset M$ is an open set,  $D^{i}(M,U) = D^i(M) \oplus D^{i-1}(U)$ and the differentials are defined in the following way: 
\[ \mathrm{d}(\omega, \eta) = (\mathrm{d}\omega, \omega|_U - \mathrm{d}\eta). \]
We  denote the cohomology  groups of the previous complex by $H^{i}_{\mathscr{D}'}(M,U, \mathbb{R})$.  
Then there is a long exact sequence 
\[
\cdots \to H^{i}_{\mathscr{D}'}(M,U, \mathbb{R}) \to H^{i}_{\mathscr{D}'}(M, \mathbb{R}) \overset{\beta}{\to} H^{i}_{\mathscr{D}'}(U, \mathbb{R}) \to \cdots
\]
such that $\beta$ is the restriction map. 
We refer to \cite[Page 78-79]{BottTu1982} for detailed discussion on relative de Rham cohomology. 
We note that  \cite{BottTu1982}  only treats de Rham cohomology of differential forms. 
But the proof of \cite[Claim 6.48]{BottTu1982} also holds for de Rham cohomology of currents in our setting. 
By applying the five lemma on the long exact sequences of relative cohomology,  
we obtain  natural  isomorphisms 
\[ H^{i}_{\mathscr{D}'}(M,U, \mathbb{R}) \cong H^{i}(M,U, \mathbb{R}). \]

\subsection{Relative Chern classes}\label{subsec:relative-class}

We aim to compare orbifold Chern classes and Chern classes on a desingularization. 
Similar topics have been studied before, especially for surfaces, see for example  \cite{Wahl1993}, \cite{Bla96} and \cite{Langer2000}.  
Inspired by their works, we will introduce the relative Chern classes in any dimension. \\

Let $X$ be a normal irreducible complex analytic variety with quotient singularities,  and $r\colon \widetilde{X} \to X$ a desingularization which is an isomorphism over  $X_{\sm}$. 
We denote by $E\subseteq \widetilde{X}$ the exceptional locus of $r$, and define $\widetilde{X}^\circ = \widetilde{X}\setminus E$.

%We denote the $r$-exceptional locus by $E\subseteq \widetilde{X}$, which is assumed   of pure dimension $1$. 

We denote the standard orbifold structure of $X$  with an atlas of charts by $X_{\orb} = \{(V_i, G_i) \}$. 
Let $\mcE$ be a reflexive sheaf on $X$ which induces an orbifold vector bundle $\mcE_{\orb}= \{\mcE_i\}$ on $X_{\orb}$. 
Assume that there is a vector bundle $\widetilde{\mcE}$ on $\widetilde{X}$ such that $(r_*(\widetilde{\mcE}))^{**} = \mcE$. 

Let $h_{X_{\sm}}$ be a Hermitian metric on the vector bundle $\mcE|_{X_{\sm }}$ which induces an orbifold Hermitian metric $\{h_j\}$ on $\mcE_{\orb}=\{\mcE_j\}$. On one hand, by \cite[Lemma 1.15]{Bla96}, the metric Chern classes $c_i(h_{X_{\sm }})$ induce closed currents $r^*c_i(h)$ on $\widetilde{X}$. 
On the other hand, by definition, $c_i(h_{X_{\sm}})$ defines the orbifold Chern classes $\hat{c}_i(\mcE) = \hat{c}_i(\mcE_{\orb})$. 
See  \cite[Section 2]{Bla96} for more details.

Let $\tilde{h}$ be a Hermitian metric on $\widetilde{\mcE}$, and $c_i(\tilde{h})$ the metric Chern classes.  
Then $c_i(\tilde{h})-r^*c_i(h)$ is exact on  $\widetilde{X}^\circ$. 
That is, there is some $\eta \in D^{2i-1}(\widetilde{X}^\circ)$ such that 
\[(c_i(\tilde{h})-r^*c_i(h) )|_{\widetilde{X}^\circ}  = \mathrm{d}\eta.
\]
Hence it corresponds to the relative de Rham current cohomology class
\[ ( c_i(\tilde{h})-r^*c_i(h) , \eta) \in H^{2i}_{\mathscr{D}'}(\widetilde{X},   \widetilde{X}^\circ, \mathbb{R}).
\]  
We notice that this class is independent of the choices of the metrics. 
We define the relative Chern class $c_i(\widetilde{\mcE}, r)$ as  the relative cohomology class 
\begin{equation}\label{eqn:relative-class-E}
 c_i(\widetilde{\mcE}, r):=   c_i(\tilde{h})-r^*c_i(h) = c_i(\widetilde{\mcE}) - r^*\hat{c}_i(\mcE) \in H^{2i}(\widetilde{X},  \widetilde{X}^\circ, \mathbb{R})
\end{equation}
 via  the de Rham isomorphism.  
Here, by abuse of notation,  \eqref{eqn:relative-class-E} signifies that $c_i(\widetilde{\mcE}, r)$ 
belongs to the natural image of $H^{2i}(\widetilde{X},  \widetilde{X}^\circ, \mathbb{R})$ in $H^{2i}(\widetilde{X}, \mathbb{R})$.

Let $T_{\widetilde{X}}$ be the tangent bundle of $\widetilde{X}$. 
Then $(r_*T_{\widetilde{X}})^{**}=T_X$. 
We set  $\hat{c}_i(X) := \hat{c}_{i}(T_X)$. 
Then  we can define the following relative Chern classes  
\begin{equation}\label{eqn:relative-class-T}
    c_i(\widetilde{X},r):=c_i(T_{\widetilde{X}}, r).
\end{equation}
We also define the following classes
\begin{equation*}
    \label{eqn:c1-square}
    c_1^2(\widetilde{\mcE},r)  := (c_1( \widetilde{\mcE}))^2-(r^*\hat{c}_1(\mcE))^2 \mbox{ and }
c_1^2(\widetilde{X},r)  : = (c_1(\widetilde{X}))^2-(r^*\hat{c}_1(X))^2  
\end{equation*}
as elements  in $H^4(\widetilde{X},  \widetilde{X}^\circ, \mathbb{R})$, as well as 
\begin{equation}\label{eqn:l-function}
     \ell(\widetilde{\mcE},r):= c_1^2(\widetilde{\mcE},r) + c_2(\widetilde{\mcE},r)  
\mbox{ and }
\ell(\widetilde{X},r):= c_1^2(\widetilde{X},r) + c_2(\widetilde{X},r).
\end{equation}

There  are natural isomorphisms 
\[
H^{2i}(\widetilde{X},  \widetilde{X}^\circ, \mathbb{R}) \cong H^{BM}_{2n-2i}(E, \mathbb{R}),
\]
where  $H^{BM}$ is the Borel-Moore homology group, see Subsection \ref{subsec:BM-homology} below for more details.  
Particularly, when $\widetilde{X}$  is a surface, we have 
\[  H^{BM}_{0}(E, \mathbb{R}) \cong \mathbb{R}^d,\]
where $d$ is the number of singular points in $X$. 
Furthermore, since $c_1^2  + c_2$ is  invariant under blowups of points on smooth surfaces, we see that 
$\ell(\widetilde{X},r)$ depends only on the singularities of $X$ when $X$ is a surface.

\subsection{Borel-Moore homology}\label{subsec:BM-homology}
As we have seen in the previous subsection, the relative Chern classes lie in some groups of relative cohomology, which are canonically isomorphic to certain groups of Borel-Moore homology. 
In this subsection, 
we recall some elementary results concerning Borel-Moore homology.  
For more details, we refer to \cite[Chapter 19]{Fulton} and the references therein.

Let $X$ be a complex analytic space of pure dimension $n$. 
For simplicity, we always assume that it is an open subset of a compact complex analytic space $M$. 
Furthermore, if ${X}'$ is the Euclidean closure of $X$ in $M$, we assume that $( {X}',  {X}'\setminus X)$ is homeomorphic to a   CW-complex pair.  
There are several interpretation of $H^{BM}_\bullet(X, \mathbb{R})$: 
\begin{enumerate}
\item Assume that $\overline{X}$ is a compact space  containing $X$ as an open subset, such that $(\overline{X}, \overline X\setminus X)$ is homeomorphic to a CW-complex pair. 
Then $H^{BM}_\bullet(X, \mathbb{R}) = H_\bullet (\overline{X}, \overline{X}\setminus X, \mathbb{R}).$

\item Assume that $X$ is embedded into an oriented manifold $Y$ of real dimension $d$, as a closed subset. 
Then  $H^{BM}_\bullet(X, \mathbb{R}) = H^{d-\bullet} (Y, Y\setminus X, \mathbb{R}).$

\item $H^{BM}_\bullet(X, \mathbb{R})$ is dual to the cohomology with compact support $H^\bullet_c(X,\mathbb{R})$. 
\end{enumerate}

The top Borel-Moore homology group $H^{BM}_{2n}(X, \mathbb{R})$ is generated by the fundamental classes  $[X_i]$, where the $X_i$'s are the irreducible components of $X$.  
The dimension of the $0$-th  Borel-Moore homology group $H^{BM}_{0}(X, \mathbb{R})$ is equal to the number of compact connected component of $X$.  
%If it is non zero, then it is generated by the fundamental classes of  points in $X$. 
Particularly, when $X$ is compact connected,  we fix the natural isomorphism $H^{BM}_{0}(X, \mathbb{R}) \cong \mathbb{R}$ by sending the class of a point to $1$.

\begin{lemma}
\label{lemma:BM-commute-basechange}
Let $f\colon X\to Y$ be a proper morphism of complex analytic varieties. 
Let $V\subseteq Y$ be a non empty  connected Euclidean open subset and $U=f^{-1}(V)$. 
Assume that $V$ is regular enough, \textit{e.g.} its closure and boundary are homeomorphic to finite CW-complexes. 
Then  for any integer $k\> 0$, we have the following commutative diagram 
\begin{equation*}
  \xymatrixcolsep{3pc}
  \xymatrix{ 
  H^{BM}_k(X,\mathbb{R}) \ar[r] \ar[d] & H^{BM}_k(U,\mathbb{R}) \ar[d]\\
  H^{BM}_k(Y,\mathbb{R}) \ar[r] &  H^{BM}_k(V,\mathbb{R})
  }
\end{equation*}
where the horizontal arrows are restriction maps and the vertical arrows are proper pushforwards.
\end{lemma} 

\begin{proof}
We set $S=X\setminus U$ and $T = Y\setminus V$. 
Let $\overline{X} = X \cup \{ \infty_X\}$ and   $\overline{Y} = Y \cup \{ \infty_Y\}$  be the one-point compactifications.    
Then $f$ extends naturally to a proper continuous map from $\overline{X}$ to $\overline{Y}$, mapping $\infty_X$ to $\infty_Y$.
We have the following natural isomorphisms 
\[
H^{BM}_k(X,\mathbb{R}) \cong H_k(\overline{X}, \{\infty_X\},\mathbb{R}), \ H^{BM}_k(U,\mathbb{R}) \cong H_k(\overline{X}, S \cup \{\infty_X\},\mathbb{R}),
\]
and 
\[
H^{BM}_k(Y,\mathbb{R}) \cong H_k(\overline{Y}, \{\infty_Y\},\mathbb{R}), \ H^{BM}_k(V,\mathbb{R}) \cong H_k(\overline{X}, T \cup \{\infty_X\},\mathbb{R}).
\]
We remark that there is a commutative diagram 
\begin{equation*}
  \xymatrixcolsep{3pc}
  \xymatrix{ 
   H_k(\overline{X}, \{\infty_X\},\mathbb{R}) \ar[r] \ar[d] & H_k(\overline{X}, S \cup \{\infty_X\},\mathbb{R}) \ar[d]\\
  H_k(\overline{Y}, \{\infty_Y\},\mathbb{R}) \ar[r] &  H_k(\overline{Y}, T \cup \{\infty_Y\},\mathbb{R})
  }
\end{equation*}
where the horizontal arrows are natural morphisms of relative homology, and the virtical arrows are pushforwards of singular homology. 
This implies the diagram of the lemma.
\end{proof}

For a closed analytic subspace $E\subseteq X$ and for any integers $p,q \> 0$, there is a cap product operation
\[
\smallfrown\colon  H^{p}(X,X\setminus E, \mathbb{R}) \times H^{BM}_{p+q}(X, \mathbb{R}) \to H^{BM}_q(E, \mathbb{R}). 
\]
We have the following statement. 

\begin{lemma}
\label{lemma:BM-projection-formula}
Let $f\colon X\to Y$ be a proper morphism of complex analytic varieties. 
Let $ V\subseteq Y$ be a compact analytic subspace.  
Let $\sigma \in H^{BM}_{\bullet}(X, \mathbb{R})$ and 
$\ell \in H^{\bullet}(Y, Y\setminus V, \mathbb{R})$. 
Then the following projection formula holds:
\[
f_*\sigma \smallfrown \ell = f_*(\sigma \smallfrown f^*\ell) \in H^{BM}_{\bullet}(V, \mathbb{R}).
\]
\end{lemma}

\begin{proof}
%\cite[Equation (57) of section V.10]{Bredon}
By taking one-point compactifications of $X$ and $Y$, the problem is reduced to the usual projection formula for homology and cohomology.
\end{proof}

Assume that $X$ is a complex manifold of dimension $n$. 
%Then we have more interpretation on the Borel-Moore homology of subvarieties of $X$.  
For any closed analytic subspace $S\subseteq X$, 
for each integer $p\>0$, there is an isomorphism 
\begin{equation}\label{eqn:BM-isomorphism}
H^{2n-p}(X, X\setminus S, \mathbb{R}) \to H^{BM}_p(S,\mathbb{R})
\end{equation}
given by the cap product with  $[X]$. 
If $S$ is irreducible and has dimension $k$, then  there is a unique cohomology class 
\[ cl^X(S) \in H^{2n-2k}(X,X\setminus S, \mathbb{R}) \]
such that the fundamental class of $[S]$ is given by 
\[[S] = cl^X(S) \smallfrown [X]   \in H^{BM}_{2k}(S, \mathbb{R}).\]  
In particular, when $S$ is a submanifold of $X$, then $cl^X(S)$ of just the Thom class of $S$ in $X$.

Let  $E\subseteq X$ be a closed analytic subspace.  
We set $\Delta = E\cap S$. 
For any class $\ell\in H^{q}(X,X\setminus E, \mathbb{R})$, any class $\sigma\in H^{BM}_p(S, \mathbb{R})$, we define      
\begin{equation}
\ell \smallfrown  \sigma : =   (\ell \smallsmile  u) \smallfrown [X].
\end{equation} 
as an element in 
$H^{BM}_{p-q}(\Delta, \mathbb{R}),$ where $u\in H^{2n-p}(X, X\setminus S, \mathbb{R})$ is the unique element such that $\sigma= u \smallfrown [X]$, and $\smallsmile$ is the  cup product of cohomology.

\subsection{Comparison with $\mathbb{Q}$-Chern classes.}

In this subsection, we will compare the orbifold Chern classes in our context and the $\mathbb{Q}$-Chern classes defined in \cite[Section 3]{Mumford83} or \cite[Chapter 10]{Kol92}, on projective surfaces $X$ with quotient singularities.   Let $\mcE$ be a reflexive sheaf on $X$ which is also an orbifold vector bundle on the standard orbifold structure $X_{\orb}$. 
We note that the $\mathbb{Q}$-Chern classes, which we denote by $\overline{c}_i(\mcE)$, of \cite[Chapter 10]{Kol92}, is a cycle class in $A_i(X)\otimes \mathbb{Q}$, where $A_i$ is the Chow group of cycles of dimension $i$. 
However, the orbifold Chern classes $\hat{c}_i(\mcE)$ are cohomology classes.  
In order to compare them, we will map them into a same set. 
Let $[X]\in H_{4}(X, \mathbb{R})$ be the fundamental class of $X$, and let $[\overline{c}_i(\mcE)]$ be the fundamental class  of $\overline{c}_i(\mcE)$ in $H_{2i}(X,\mathbb{R})$. 

\begin{lemma}
    \label{lemma:Q-Chern-class}
With the notation above, we have 
\begin{equation}
[\overline{c}_i(\mcE) ] = \hat{c}_i(\mcE) \smallfrown [X] \in H_{4-2i}(X, \mathbb{R}).  
\end{equation} 
As a consequence, we can apply the calculation of \cite[Chapter 10]{Kol92} for $\mathbb{Q}$-Chern classes  to orbifold Chern classes in our context for projective surfaces. 
\end{lemma}

%claim that, up to canonical isomorphisms,  the orbifold Chern classes defined by orbifold metrics are the same as the one defined in \cite[Section 3]{Mumford83} or \cite[Section 10]{Kol92}, on projective surfaces $X$ with quotient singularities.  
%As a result, we can apply the calculation of \cite{Kol92} in our context for projective surfaces. 

\begin{proof}
Let $i\in \{0,1,2\}$.     
We recall the definition in  \cite[Chapter 10]{Kol92}. 
There is a finite   cover $p\colon \widetilde{X} \to X$, called global cover in \cite[Section 2]{Mumford83},  satisfying the following conditions. %(See also \cite[Notation 3.1]{DruelOu2022}).
\begin{enumerate}
    \item $p$ is Galois of group $G$. 
    \item $\widetilde{X}$ is a normal surface, hence Cohen-Macaulay. 
    \item  There is a finite family of orbifold charts  $X_{\orb}=\{(V_i,G_i)\}$, such that for every $\tilde{x}\in \widetilde{X}$, there is an open Euclidean neighborhood $\widetilde{U}$ of $\tilde{x}$, such that $p|_{\widetilde{U}}$ factors through some $V_i$.
   \end{enumerate}
%Applying the Galois action of $G$ on the open subset $U$ in the condition (3), we obtian that 
%\begin{enumerate}
%    \item[(4)] for every $x\in X$, there is an open Euclidean neighborhood $U$ of $x$, such that $p|_{p^{-1}(U)}$ factors through some $V_i$.
%   \end{enumerate}
As a consequence, we deduce that  $ \widetilde{\mcE}  = (p^*\mcE)^{**}$ is locally free.  
We can view the Chern class of $\widetilde{\mcE}$, which we denote by $\tilde{c}_i(\widetilde{\mcE})$, 
as an element in $A^i(\widetilde{X})$, where $A^i$ is the $i$-th Chow cohomology group, in the sense of \cite[Chapter 17]{Fulton}.  
Note that $\tilde{c}_i(\widetilde{\mcE})$  is invariant under the natural action of $G$.  
The Chern classes in \cite[Chapter 10]{Kol92} are defined as 
\[
\overline{c}_i(\mcE) = 
\frac{1}{|G|} \cdot  p_* ( \tilde{c}_i(\widetilde{\mcE})  \smallfrown [ \widetilde{X}]'  ) \in A_{2-i}(X)\otimes \mathbb{R}, 
\]
where $[ \widetilde{X}]' \in A_2(\widetilde{X})$  is the fundamental class of $\widetilde{X}$ (see also \cite[Lemma 3.5]{Mumford83}).

Let $\widetilde{Y} \to \widetilde{X}$ be an equivariant desingularization, and let $Y= \widetilde{Y}/G$.  
We have the following commutative diagram. 
\[
  \xymatrixcolsep{3pc}
  \xymatrix{ \widetilde{Y} \ar[r]^f \ar[d]_q & \widetilde{X} \ar[d]^p\\
  Y \ar[r]_g &  X
  }
\]  
Let $h_{\sm}$ be a Hermitian metric on $\mcE|_{X{\sm}}$ which extends to an orbifold metric on $\mcE_{\orb}$.
Then by the condition (3)  of the global cover,  
the pullback of $h_{\sm}$ under $p\circ f$ extends to  a  $G$-invariant metric $h'$ on $\widetilde{\mcF}:= f^*\widetilde \mcE$. 
Since $\widetilde{Y}$ is smooth projective,  various definitions of Chern classes are consistent.
More precisely, there is a cycle $T$ of  dimension $2-i$ in $\widetilde{Y}$ such that 
\[
[T] =  c_i( \widetilde{\mcF}, h') \smallfrown [\widetilde{Y}] \in H_{4-2i}( \widetilde{Y} ,\mathbb{R}), 
\]
%where by abuse of notation, the $[\widetilde{Y}]$ above is the fundamental class of $\widetilde{Y}$ in $H_4( \widetilde{Y} ,\mathbb{R}),$
and that
\[
T = \tilde{c}_i( \widetilde{\mcF}) \smallfrown [\widetilde{Y}]' \in A_{2-i}(\widetilde{Y}) \otimes \mathbb{R}. 
\]
Furthermore, we may assume that $T$ is $G$-invariant.

On the one hand, we note  that $(\widetilde{Y},G,q)$ is an  orbifold chart of $Y$, and that $\widetilde{\mcF}$ is an orbifold vector bundle.  
With this orbifold structure, we have 
\[
\hat{c}_i(\widetilde{\mcF})  =  c_i( \widetilde{\mcF}, h').
\]
Moreover, the condition (3) above implies that 
$
\hat{c}_i(\widetilde{\mcF}) = g^* \hat{c}_i(\mcE_{\orb}) = g^* \hat{c}_i(\mcE). 
$
By the projection formula, this shows that  
\[
 \hat{c}_i(\mcE) \smallfrown  [ X]= g_*(\hat{c}_i(\widetilde{\mcF}) \smallfrown  [ {Y}] ) \in H_{4-2i}(X,\mathbb{R}).
\]
Hence we get 
\[
  \hat{c}_i(\mcE) \smallfrown  [ X]  =  g_*  \left( \frac{1}{|G|} \cdot     q_*(  c_i( \widetilde{\mcF}, h') \smallfrown [\widetilde{Y}]  ) \right)
        =\frac{1}{|G|} \cdot p_*f_*[T].
\]

On the other hand, by the projection formula, we  have 
\[
f_*T = f_* (\tilde{c}_i( \widetilde{\mcF})  \smallfrown [\widetilde{Y}]'  ) 
=  \tilde{c}_i(\widetilde{\mcE}) \smallfrown [\widetilde{X}]' \in A_{2-i}(\widetilde{X})\otimes \mathbb{R}.
\]
It follows that  
\[
[\overline{c}_i(\mcE) ] = 
\left[ 
\frac{1}{|G|} \cdot p_*  ( \tilde{c}_i(\widetilde{\mcE})  \smallfrown [ \widetilde{X}]'  )  \right]
= \frac{1}{|G|}  \cdot p_*f_*[T]   \in H_{4-2i}(X, \mathbb{R}).     
\]
This completes the proof of the lemma. 
\end{proof}

%We remark that 
%\[c_i( \widetilde{\mcF}, h') = q^*\hat{c}_i(\mcF) \in H^{2i}( \widetilde{Y} ,\mathbb{R}). \]
%Therefore, by the projection formula, we get 
%\[\frac{1}{|G|} \cdot q_*[T] =  \hat{c}_i(\mcF) \smallfrown  [ {Y}] \in   H_{4-2i}( Y ,\mathbb{R}).\] 
%Moreover, by definition, we have 
%\[\overline{c}_i(\mcF) = \frac{1}{|G|}  \cdot q_*T \in  A_{2-i}({Y})\otimes \mathbb{R}.\]
%By taking the fundamental class, we deduce that 
%\[
%[\overline{c}_i(\mcF) ] = \hat{c}_i(\mcF) \smallfrown [Y] \in H_{4-2i}(Y, \mathbb{R}).  
%\]
%
%Now we will push the previous equality on $X$.  
%On the one hand, by the projection formula, we obtain that 
%$ f_* (c_i( \widetilde{\mcF})  \smallfrown [\widetilde{Y}]'  ) =  c_i(\widetilde{\mcE}) \smallfrown [\widetilde{X}]' \in A_{2-i}(\widetilde{X})\otimes \mathbb{R}.$ 
%It follows that  
%$[\overline{c}_i(\mcE) ] = g_* [\overline{c}_i(\mcF) ].   $
%On the other hand, the morphism $g$ induces a morphism of orbifolds $Y_{\orb} \to X_{\orb}$. 
%It follows that $\hat{c}_i(\mcF) = g^* \hat{c}_{i}(\mcE)$. 
%Thus 
%$g_*(\hat{c}_i(\mcF) \smallfrown  [ {Y}] ) = \hat{c}_i(\mcE) \smallfrown  [ X].$
%Hence 
%$ [\overline{c}_i(\mcE) ] = \hat{c}_i(\mcE) \smallfrown [X] \in H_{4-2i}(X, \mathbb{R}).  $ 
%This completes the proof of the lemma.  

\subsection{The case of complex analytic threefolds.} 
We will focus on the case when $X$ is a complex analytic threefold and study the relative class $\ell(\widetilde{X},r) =  c_1^2(\widetilde{X},r) + c_2(\widetilde{X},r)$ defined in  \eqref{eqn:l-function}. 
For later applications, we need to  compute its pushforward   to $X$.  
Throughout this subsection, in order to have a better interpretation of Borel-Moore homology, 
we will assume that $X$ is a Zariski open subset of a compact complex analytic threefold.

% \begin{lemma}\label{lem:quotient-klt}
%     If $X$ is a variety with quotient singularities, then $X$ has klt singularities.
% \end{lemma}

% \begin{proof}
%     Let $x\in X$; then there is an open neighborhood $U$ of $x$, a open subset $V\subset \mbC^{\dim X}$ and a finite group $G$ such that $U\cong V/G$. From \cite[Proposition 5.15]{KM98} it follows that $U$ is normal and $\mbQ$-factorial, in particular, $K_U$ is $\mbQ$-Cartier. 
%     Let $\pi:V\to U$ be the quotient map and $K_V\sim_{\mbQ}\pi^*K_U+R$, where $R$ is the ramification divisor. 
%     Since $V$ is smooth, and since $(V, -R)$ is klt,   from \cite[Proposition 5.20]{KM98}, it follows that $U$ has klt singularities. 
% \end{proof}

\begin{lemma}
\label{lemma:c_2-difference}
Let $X$ be a complex analytic  threefold with quotient singularities, and $r\colon \widetilde{X}\to X$ a desingularization such that $r$ is an isomorphism over $X_{\rm sm}$.  
We denote by $C_1,...,C_r$ the $1$-dimensional irreducible components of $X_{\rm sing}$. 
Then there are real numbers $a_i$ such that for any class $\sigma \in H^{2}_c(X, \mathbb{R})$, we have 
\begin{equation*}\label{equation:compare}
 \Big(( c_1(\widetilde{X}))^2 + c_2(\widetilde{X}) \Big) \cdot r^*\sigma = \Big((\hat{c}_1(X))^2  + \hat{c}_2(X) + \sum_{i=1}^r a_i[C_i]\Big) \cdot \sigma   
\end{equation*} 
as real numbers. 
In other words, we have 
\[
\ell(\widetilde{X}, r) \cdot r^*\sigma = \Big(\sum_{i=1}^r a_i[C_i]  \Big) \cdot \sigma.
\]
\end{lemma}
 
The dot products in the lemma are defined as follows. 
For example, we note that $H^6_c(X,\mathbb{R})$ is canonically isomorphic to $\mathbb{R}$. 
Then for any  $a\in H^4(X,\mathbb{R})$, the product $a\cdot \sigma$  is the number corresponding to $a\smallsmile \sigma \in H^6_c(X,\mathbb{R})$. 
For any $b\in H^{BM}_2(X,\mathbb{R})$, the product $b\cdot \sigma$ is then defined by the duality between $H^{BM}_2(X,\mathbb{R})$ and $H^2_c(X,\mathbb{R})$.

\begin{proof}%[{Proof of Lemma \ref{lemma:c_2-difference}}]
Let $E:=\Ex(r)$ be the exceptional locus of $r$  and let $\widetilde{X}^\circ =\widetilde{X}\backslash {E}$. 
Then 
$\ell(\widetilde{X},r)  \in H^4(\widetilde{X}, \widetilde{X}^\circ, \mathbb{R})$ 
and 
\[\ell(\widetilde{X},r) \smallfrown [\widetilde{X}] \in  H^{BM}_2(E,\mathbb{R}).\] 
On the one hand, the Poincar\'e duality shows that 
\[
\ell(\widetilde{X}, r) \cdot r^*\sigma = (\ell(\widetilde{X},r) \smallfrown [\widetilde{X}]) \cdot r^*\sigma 
\]
as real numbers. 
On the other hand, since $r$ is proper,  there are real numbers $a_1,...,a_r$ such that
\begin{equation}\label{eqn:pushforward-of-relative-class} 
r_*{(\ell(\widetilde{X},r) \smallfrown [\widetilde{X}])} = \sum_{i=1}^r a_i[C_i],
\end{equation}
where $[C_i]\in H_2^{BM}(X_{\sing}, \mathbb{R})$ is the fundamental class of $C_i$. 
This completes the proof of the lemma. 
\end{proof}

In the following lemmas, we show that one can compute the numbers  $a_i$ by taking a local slice intersecting $C_i$ transversally. 
From  \cite[Lemma 5.8]{GK20} it follows that for $x\in C_i$ a general point, there is an Euclidean neighborhood $U\subseteq X$ of $x$ such that $U\cong S\times \mathbb{D}$, where $(o\in S)$ is a germ of surface quotient singularities and $\mathbb{D} \subset \mathbb{C}$ is the open unit disk.   
Let $\widetilde{U}= r^{-1}(U)$. 
%and $\widetilde{U}^\circ = \widetilde{U} \cap \widetilde{X}^\circ$. 
Shrinking $U$ if necessary we may assume that $\widetilde{U} \cong \widetilde{S}\times \mathbb{D}$, 
where $\mu\colon \widetilde{S} \to S$ is a desingularization. 
Let $\Delta:=\Ex(\mu)$ and  let $\pi\colon \widetilde{U}\to \mathbb{D}$ be the natural projection. 
We denote by $h\colon \widetilde{S} \injective \widetilde{U}$   the inclusion induced by the zero section of $\mathbb{D}$.

\begin{lemma}
\label{lemma:a_i-intersection}   
With the notations and hypothesis as above, 
the number $a_i$ represents the class
 $h^*\ell(\widetilde{U}, r) \smallfrown [\widetilde{S}] \in H_0^{BM}(\Delta, \mathbb{R})$  via  the natural isomorphism $H_0^{BM}(\Delta, \mathbb{R}) \cong \mathbb{R}$. 
\end{lemma}

\begin{proof} 
Let $\widetilde{T}:=h(\widetilde{S}) \subseteq \widetilde{U}$. 
By \cite[Equation (8), Section 19.1]{Fulton}, we see that 
\[
\ell(\widetilde{U}, r) \smallfrown [\widetilde{T}] =
 (\ell(\widetilde{U}, r) \smallsmile cl^{\widetilde{U}}(\widetilde{T}) ) \smallfrown [\widetilde{U}]=
h^*\ell(\widetilde{U}, r) \smallfrown [\widetilde{S}] \in H_0^{BM}(\Delta, \mathbb{R}).
\]
Let $\delta \in H^2(\mathbb{D}, \mathbb{D}\setminus \{ 0 \}, \mathbb{R})$ be the Thom class of  $\{0\}$ in $\mathbb{D}$.   
Then $ cl^{\widetilde{U}}(\widetilde{T})  = \pi^*\delta$. 
Therefore, we have 
\[
\ell(\widetilde{U}, r) \smallfrown [\widetilde{T}]  
= (\ell(\widetilde{U}, r) \smallsmile \pi^* \delta) \smallfrown [\widetilde{U}]
= \pi^*\delta \smallfrown (\ell(\widetilde{U}, r) \smallfrown  [\widetilde{U}] ).
\]
We note that $ \ell(\widetilde{U}, r) \smallfrown  [\widetilde{U}]  \in H^{BM}_2(E\cap \widetilde{U}, \mathbb{R})$, where $E=\Ex(r)$.  

On one hand, by applying the projection formula of Lemma \ref{lemma:BM-projection-formula} to the proper morphism $\pi\colon E\cap \widetilde{U} \to \mathbb{D}$, we deduce that 
\begin{equation} \label{eqn:proj-for}
 \delta \smallfrown \pi_*(\ell(\widetilde{U}, r) \smallfrown  [\widetilde{U}] )
 =  \pi_*(\ell(\widetilde{U}, r) \smallfrown [\widetilde{T}])\in H^{BM}_0(\{0\}, \mathbb{R}).   
\end{equation}
On the other hand, by Lemma \ref{lemma:BM-commute-basechange}, we have 
\[r_*(\ell(\widetilde{U}, r) \smallfrown  [\widetilde{U}] ) = a_i[C_i\cap U],\]
which implies that 
\[\pi_*(\ell(\widetilde{U}, r) \smallfrown  [\widetilde{U}] )   = a_i[ \mathbb{D}].\] 
Hence we deduce that the LHS of the equation (\ref{eqn:proj-for})  is equal to $a_i$. 
This shows the statement of the lemma.
\end{proof}

\begin{lemma}
\label{lemma:c2-calculation}
With the notations and hypothesis as in Lemma \ref{lemma:a_i-intersection}, we have $a_i = \ell(\widetilde{S}, \mu)\smallfrown [\widetilde{S}]$,   via  the natural isomorphism $H_0^{BM}(\Delta, \mathbb{R}) \cong \mathbb{R}$.  
\end{lemma}

\begin{proof} 
From Lemma \ref{lemma:a_i-intersection}, we see that 
$a_i=h^*\ell(\widetilde{U}, r) \smallfrown [\widetilde{S}]$.  
Since the Chern classes commute with pullbacks, we obtain that 
$h^*\ell(\widetilde{U}, r)  = \ell(h^*T_{\widetilde{U}}  , \mu)$.  
We note that  $h^*T_{\widetilde{U}} \cong T_{\widetilde{S}} \oplus \mathcal{O}_{\widetilde{S}}$.
Thus we have 
\[
h^*\ell(\widetilde{U}, r)  =  \ell( h^*T_{\widetilde{U}}, \mu) = \ell(\widetilde{S}, \mu).
\] 
This completes the proof of the lemma.
\end{proof}

By the same techniques, we obtain the following lemma on comparison of second Chern classes.

\begin{lemma}
\label{lemma:compare-log-c_2}
Let $X$ be a compact analytic threefold with quotient singularities. Let $B$ be a reduced boundary on $X$ such that $(X,B)$ is lc. 
Assume that $\Omega_X^{[1]}(\mathrm{log}\, B)$ induces orbifold vector bundle on the standard orbifold structure of $X$. 
Then there exists an effective $1$-cycle $C$ in $X$ such that 
\[
\hat{c}_2(X) - \hat{c}_2(\Omega_X^{[1]}(\mathrm{log}\, B)) + (K_X+B) \cdot B = [C] \in H^{4}(X, \mathbb{R}). 
\]
\end{lemma}

\begin{proof}
We first assume that $(X,B)$ is log smooth.  
Then there are exact sequences 
\[
0\to \Omega_X^1 \to \Omega_X^1(\log\, B) \to \bigoplus \mcO_{B_i} \to 0, 
\]
\[
0 \to \mcO_X(-B_i) \to \mcO_X \to \mcO_{B_i}\to 0,
\]
where the $B_i$'s are the irreducible components of $B$. 
Thanks to the theory of Chern classes of coherent sheaves on compact complex manifolds, see  \cite[Section 3]{Gri10}, we can deduce that 
\[c_2(X) - c_2(\Omega_X^{1}(\mathrm{log}\, B)) + (K_X+B) \cdot B\num 0.\] 

In general,  let $r\colon \widetilde{X} \to X$ be a log resolution of $(X,B)$ and $\widetilde{B} = r^{-1}_*(B)$. 
We consider the relative class 
\begin{eqnarray*}
 q(\widetilde{X}, r) &=& \Big({c}_2(\widetilde{X}) -  {c_2}(\Omega_{\widetilde{X}}^{1}(\mathrm{log}\, \widetilde{B})) + (K_{\widetilde{X}}+\widetilde{B}) \cdot \widetilde{B}\Big) \\
 && - r^*\Big(\hat{c}_2(X) - \hat{c}_2(\Omega_X^{[1]}(\mathrm{log}\, B)) + (K_X+B) \cdot B \Big).
\end{eqnarray*}
The same argument of Lemma \ref{lemma:c_2-difference}  shows that, 
for any $\sigma \in H^2_c(X,\mathbb{R})$, 
\[
 q(\widetilde{X}, r) \cdot r^*\sigma = \Big(\sum_{i=1}^r b_iC_i \Big) \cdot \sigma,  
\]
where the $C_i$'s  are the 1-dimensional irreducible components of $X_{\sing}$.  
It suffices to show that $b_i\< 0$ for all $i$. To this end, we use the method of Lemma \ref{lemma:a_i-intersection} and Lemma \ref{lemma:c2-calculation}. 
Let $x$ be a general point in $C_i$. 
 Then by  \cite[Proposition 16.6]{Kol92},   there is a neighborhood $U\subseteq X$ of $x$ such that the following properties hold. 
There is a lc surface pair $(S,F)$ such that $(U,B)\cong (S,F)\times \mathbb{D}$, where $\mathbb{D}$ is the unit disk. 
There is a log smooth surface pair $(W,\Theta)$ and a group $G$ acting on it, freely in codimension 1 of $W$, such that 
$(W,\Theta)/G \cong (S,F)$.

Let $\widetilde{U}=r^{-1}(U)$. 
Shrinking $U$ if necessary, we may assume that 
$(\widetilde{U}, \widetilde{B}\cap \widetilde{U})$ is isomorphic to $(\widetilde{S}, \widetilde{F})\times \mathbb{D}$, where $\mu \colon \widetilde{S} \to S$ is a log resolution of $(S,F)$ and $\widetilde{F} = \mu^{-1}_*F$. 
Now by the same arguments of  Lemma \ref{lemma:a_i-intersection} and Lemma \ref{lemma:c2-calculation}, 
we see that $b_i$ can be calculated as follows,
\begin{eqnarray*}
 b_i &=& \Big( \Big({c}_2(\widetilde{S}) -  {c_2}(\Omega_{\widetilde{S}}^{1}(\mathrm{log}\, \widetilde{F})) + (K_{\widetilde{S}}+\widetilde{F}) \cdot \widetilde{F}\Big) \\
 & & - r^* \Big( \hat{c}_2(S) - \hat{c}_2(\Omega_S^{[1]}(\mathrm{log}\, F)) + (K_S+F) \cdot F \Big)  \Big) \smallfrown [\widetilde{S}] \\
 &\in & H^{BM}_0(\Delta, \mathbb{R}).
\end{eqnarray*}
  via  the natural isomorphism $H^{BM}_0(\Delta, \mathbb{R}) \cong \mathbb{R}$. 

%We note that the previous number depends only on the singularity of $(S,F)$. 
 Furthermore, by \cite[Lemma 9.9]{Kaw88},   we may assume that 
 $W\subseteq \mathbb{C}^2$ is an Euclidean open neighborhood of the origin, that the action of $G$ extends to a linear action on $\mathbb{C}^2$, and that $\Theta=\Theta'\cap U$, where $\Theta'\subseteq \mathbb{C}^2$ is a union of coordinates hyperplanes.  
Let $(T, \overline{F})$ be a projective compactification of $(\mathbb{C}^2/G, \Theta'/G).$ 
Then there is an Euclidean open embedding $\iota\colon (S,F) \to (T, \overline{F})$. 
Let $o\in S$ be the unique singular point. 
By taking a  partial desingularization, we may assume that $(T,\overline{F})$ is log smooth outside $\iota(o)$.  
By abuse of notation, we still denote by $\mu\colon \widetilde{T} \to T$ a desingularization, such that $\mu^{-1}(S) = \widetilde{S}$. 
Since the Chern classes commute with pullbacks,   in order to compute $b_i$ we may replace $S$ by $T$ and assume that $S$ is a projective surface and $(S,F)$ is log smooth outside $o$.   
Then  we can apply  \cite[Equation (10.8.8)]{Kol92} to deduce that $b_i \< 0$. 
 This completes the proof of the lemma. 
\end{proof}

\section{Positivity of logarithmic cotangent sheaves}
\label{section:log-cotan}

The main objective of this section is to establish the following proposition. 
We will also study its variant for non-uniruled threefolds, 
see Proposition \ref{prop:psef-c2} in the last subsection, whose proof is simpler.     

\begin{proposition}
\label{prop-log-orbifold-c2-semipositive}
Let $Y$ be a  non-algebraic uniruled compact K\"ahler threefold with quotient singularities.  
Assume the following conditions hold:
\begin{enumerate}
    \item There is a bimeromorphic morphism $\pi\colon Y \to X$ to a normal compact K\"ahler threefold  such that $L=k(K_X+\D)$ is Cartier and nef for some  reduced boundary $\D$ and some integer $k>0$.  
    \item $L^2\not\equiv 0$ and $L^3 = 0$. 
    \item  $\pi^*L|_P \equiv 0$, in particular  $\pi^*L \cdot P \equiv 0$,  for any $\pi$-exceptional prime divisor $P$. 
    \item If $B = \pi^{-1}_*\D$, then $K_Y+(1-\lambda)B$ is pseudoeffective for some $\lambda >0$.
%    \item  $(Y, B)$ is lc. 
    \item $\Omega_Y^{[1]}(\mathrm{log}\,  B)$ induces an orbifold vector bundle on the standard orbifold structure of $Y$. 
\end{enumerate}    
Then  $\hat{c}_2(\Omega_Y^{[1]}(\mathrm{log}\,  B)) \cdot  \pi^*L \> 0$.   
\end{proposition}
 
We remark that the condition that $L^3\equiv 0$ is automatic under the assumption that $Y$ is K\"ahler and non-algebraic.  
This proposition is a consequence of the generic nefness of the logarithmic cotangent sheaf  $\Omega_{Y}^{[1]}(\mathrm{log}\, B)$, see Lemma \ref{lem-orbifold-semi-positivity} below.  
A key ingredient is the following result.

\begin{proposition}
\label{prop-effective-adjoint-foliation-canonical}
Let $X$ be a smooth uniruled non-algebraic   compact  $\mathbb{Q}$-factorial K\"ahler threefold. 
Let $\mcF$ be the foliation corresponding to the MRC fibration.   
Assume that there is a $\mbQ$-divisor $\D$  such that $(X,\D)$ is klt  and  $K_X+\D$  is pseudoeffective.  
Then   $(K_{\mcF}+\D_{\rm hor}) \cdot \omega^{2} \> 0$ for any K\"ahler class $\omega$ on $X$, where $\D_{\rm hor}$ is the horzontal part of $\D$ over the base of the MRC fibration.  
\end{proposition}

If we assume that  $X$ is a projective threefold and the base of its MRC fibration is a surface, then the same conclusion of the previous proposition still holds, and is  a special case of \cite[Proposition  4.1]{Dru17}. 
Indeed, our proof follows the ideas there, which we sketch as follows.  
{
For simplicity, we assume that the MRC fibration of $X$ is an equidimensional morphism $f\colon X \to Y$. 
Then from the positivity results of direct images, we may deduce certain positivity on $K_{X/Y}+\Delta_{\rm hor}$. 
We remark that  $K_{X/Y}-K_{\mcF}$ is an effective divisor, supported in the locus where $f$ has non reduced fibers. 
To deduce the positivity of $K_{\mcF}$ from the one of $K_{X/Y}$, the proof of \cite[Proposition  4.1]{Dru17} applies 
Kawamata's cyclic covering trick on $Y$, and reduces the problem to the case when $f$ has reduced fibers in codimension one. 
In such a case, we have $K_{X/Y}=K_{\mcF}$  and we can  conclude.  
In the situation of Proposition \ref{prop-effective-adjoint-foliation-canonical}, Kawamata's  covering trick is not known for compact K\"ahler varieties. 
Therefore, in this section,  we will use the basechange method of \cite[Lemma 2.11]{DasOu2022}, and argue with some elementary geometric computation. 
}

\subsection{$\mathbb{P}^1$-fibirations}
We will first prove some   results on fibrations whose general fibers are isomorphic to $\mathbb{P}^1$.

\begin{lemma}\label{lem:reflexive-pushforward}
	 Let $f:X\to Y$ be a proper fibration of complex manifolds. 
     %such that the general fibers of $f$ are isomorphic to $\mbP^1$. 
  Let $\mcL$ be a line bundle on $X$ whose restrictions on   the general fibers of $f$ are  trivial bundles. 
	Then there is a $f$-exceptional divisor $E_1$ and an effective divisor $E_2$ such that $f^*((f_*\mcL)^{**})\cong \mcL \otimes \mcO_X(E_1-E_2)$.
\end{lemma}

\begin{proof}
We observe that $f_*\mcL$ is a torsion-free sheaf of rank 1 on $Y$. 
Since $Y$ is smooth, the reflexive hull $(f_*\mcL)^{**}$ is a line bundle on $Y$. 
%, see \cite[Lemma 1.1.15, Chap. 2]{OSS11}. 
Let $U\subseteq Y$ be the largest open subset over which $f_*\mcL$ is locally free. 
Then $\codim_Y(Y\setminus U)\ge 2$. 
Let $V=f^{-1}(U)$. 
Then every prime Weil divisor contained in $X\setminus V$ is $f$-exceptional. 

On one hand, we notice that the natural morphism $f^*f_*\mcL \to \mcL$   is generically an isomorphism.  
Hence the induced  morphism $(f^*f_*\mcL)^{**} \to \mcL$ of line bundles is  generically an isomorphism as well.
It follows that there is an effective divisor $E_2$, such that 
\[
(f^*f_*\mcL)^{**} \cong \mcL\otimes \mcO_X(-E_2). 
\]

On the other hand, the natural morphism $f_*\mcL \to (f_*\mcL)^{**}$ induces natural  morphisms
\[
f^*f_*\mcL \to f^*((f_*\mcL)^{**}) 
\]
and 
\[
(f^*f_*\mcL)^{**} \to f^*((f_*\mcL)^{**}) 
\]
which are isomorphisms on $V$. 
Hence there is a divisor $E_1$ with support in $X\setminus V$ such that 
\[
(f^*f_*\mcL)^{**} \otimes \mcO(E_1) \cong f^*((f_*\mcL)^{**}). 
\]

Therefore, we obtain that 
\[
f^*((f_*\mcL)^{**}) \cong \mcL\otimes \mcO_X(E_1-E_2). 
\]
Since $E_1$ is contained in $X\setminus V$, it is $f$-exceptional. 
This completes the proof of the lemma.  
\end{proof}

%In the rest of the section we will use the results obtained in Subsection \ref{subsec:holomorphic-foliation} without reference.  

\begin{lemma}
\label{lem-blowup-more-adjoint-foliation}
Let $f\colon X\to Y$ be a projective fibration from a  smooth analytic surface to a germ $(y\in Y)$ of smooth curve. 
Assume that the general fibers of $f$ are isomorphic to $\mathbb{P}^1$. 
Let $\mcF$ be the foliation on $X$ induced by $f$.  
Suppose  that $\D$ is a  $\mbQ$-divisor such that $(X, \Delta)$ has canonical singularities and  every component of $\Delta$ is isomorphic to $Y$  via  $f$.  

Let $g\colon X' \to X$ be the blowup of a point $x  \in X$ lying over $y$, $f'\colon X'\to Y$   the composite morphism, $\Delta':=g^{-1}_*\Delta$, and $\mcF'$   the foliation induced by $f'$. 
Then the difference $K_{\mcF'}+\D' - g^*(K_{\mcF}+\D)$ is a $g$-exceptional effective $\mbQ$-divisor. 
In particular, there is a natural isomorphism 
\[g_*\mathcal{O}_{X'}(m(K_{\mcF'}+\D')) \cong \mathcal{O}_{X }(m(K_{\mcF} +\D ))\]
%is an isomorphism 
for any positive integer $m$ such that $m\D$ is a $\mbZ$-divisor. 
\end{lemma}

\begin{proof} 
We claim that $x$ is contained in at most two components of the fiber $X_y := f^{-1}(y)$. To see this claim we   run a $K_X$-MMP over $Y$  and end with a $\mathbb{P}^1$-bundle $\varphi:Z\to Y$.  
We note that, a MMP here is just a sequence of contractions of $(-1)$-curves.  
Such a MMP exists as we can assume that $f$ has only finitely many singular fibers, up to shrinking $Y$ around $y$. 
Then every step of the MMP is a   contraction of a component of a singular fiber. 
Let $\psi:X\to Z$ be the induced morphism. Then $\psi$ can be seen as the composition of  blowups  at smooth points. 
Hence $x$ is contained in at most two components of $X_y$.  
In addition, we also see that every irreducible component of $X_y$ is a smooth rational curve. 

We first assume that $x$ is contained in   two irreducible components $F_1$ and $F_2$ of $X_y$. 
Let $m_1$ and $m_2$ be the coefficients of $F_1$ and $F_2$  in the  divisor  $f^*y$. 
Then locally around $x\in X$ we have
\[K_{\mcF}=K_{X/Y} - (m_1-1)F_1 -(m_2-1)F_2.\]
Let $E\subset X'$ be the $g$-exceptional divisor. 
 Then locally around $E$ in $X'$ we have 
\begin{eqnarray*}
K_{\mcF'} &=&  K_{X'/Y}-(m_1+m_2-1)E - (m_1-1)g_*^{-1}F_1 -(m_2-1)g_*^{-1}F_2\\
         &=& g^*K_{X/Y} -(m_1+m_2-2)E - (m_1-1)g_*^{-1}F_1 -(m_2-1)g_*^{-1}F_2\\
         &=& g^*K_{X/Y} - (m_1-1)g^* F_1 -(m_2-1)g^* F_2\\
         &=& g^*K_{\mcF}. 
\end{eqnarray*} Since $g$ is an isomorphism over $X\backslash \{x\}$, the equality $K_{\mcF'} \sim g^*K_{\mcF}$ holds  globally on $X'$. Moreover, since every component of $\D$ is isomorphic to $Y$ via $f$, the components of $\Delta$ are sections of $f:X\to Y$. 
In particular, the components of $\Delta$ do not pass through the singular points of   $X_y$, 
and hence $x$ is not contained in the support of $\Delta$. 
Therefore $\Delta'=g^{-1}_*\D=g^*\Delta$ and  we have 
\[K_{\mcF'}+ \D'=g^*(K_{\mcF}+\D).\]

Now we assume that $x$ is contained in exactly one component $F_1$ of $X_y$. Let $m_1$ be the coefficient of $F_1$ in $f^*y$. Then a similar computation as above shows that 
\[K_{\mcF'}=g^*K_{\mcF}+E.\] 
Then we have  
\[K_{\mcF'}+ \D'=g^*(K_{\mcF}+\D) + ({1-\beta}) E,\]
where $\beta$ is the multiplicity of $\D$ at $x$. Since $(X,\D)$ is canonical, we have $\beta \< 1$. 
It follows that  \[K_{\mcF'}+ \D' - g^*(K_{\mcF}+\D)  \> 0.\]
This completes our proof.
\end{proof}

\begin{lemma}
\label{lem-same-direct-image-horizontal-bimeromorphic}
Let $f\colon X\to Y$ be a projective fibration between  complex manifolds. Assume that general fibers are isomorphic to $\mathbb{P}^1$.  
Let $\mcF$ be the foliation of $f$. 
Assume  that there is a  $\mathbb{Q}$-divisor  $\D $ on $X$ whose components are all bimeromorphic to $Y$  via  $f$ such that  $(X,\D)$ is terminal and that $K_X+\D$ has intersection number $0$ with the general fibers of $f$.  Then for any positive integer $m$ such that $m\D$ is an integral divisor, the natural morphism  $$\varphi\colon f_*\mathcal{O}_{X}(m(K_{\mcF}+ \D)) \to  f_*\mathcal{O}_{X}(m(K_{X/Y}+ \D))$$ is an isomorphism in codimension $1$.
\end{lemma}

\begin{proof} 
Since the question is local on the base $Y$, we may assume that $Y$ is a germ of complex manifold. 
Cutting $Y$ by $\dim Y-1$ number of general sections of $H^0(Y, \mcO_Y)$, 
we assume that $Y$ is a smooth curve and $X$ is a smooth  surface.   
Furthermore, we may also assume that $f$ has only finitely many singular fibers.  
Now   we can run a $f$-relative $(K_X+\D)$-MMP over $Y$  and obtain a pair $(\widehat{X},\widehat{\D})$ such that $K_{\widehat{X}}+\widehat{\D}$ is nef over $Y$.  

Since all  the  components of $\D$ are horizontal over $Y$, none of them are contracted by this $f$-relative MMP. 
Hence, $(\widehat{X},\widehat{\D})$ is still terminal. In particular, $\widehat{X}$ is still a smooth surface. 
Now we run a $K_{\widehat{X}}$-MMP over $Y$, trivial with respect to  $K_{\widehat{X}}+\widehat{\D}$. 
Since  $K_{\widehat{X}}+\widehat{\D}$ is numerically trivial on the general fibers of $\widehat{X} \to Y$, and is nef over $Y$, 
it must be numerically trivial over $Y$. 
Hence the MMP terminates with a $\mathbb{P}^1$-bundle $f'\colon X'\to Y$. 
Let $\D'$ be the strict  transform of $\widehat{\D}$ onto $X'$. 
Then  $(X',\D')$ has canonical singularities,  since the MMP is $(K_{\widehat{X}}+\widehat{\D})$-trivial and does not contract any component of $\widehat{\D}$. 

Let $g\colon X\to X'$ be the induced bimeromorphic morphism.  
Since $f' \colon X'\to Y$ is a   $\mathbb{P}^1$-bundle, for an  integer $m>0$ sufficiently large and divisible,  there is a  natural isomorphism
\[
f'_*\mathcal{O}_{X'}(m(K_{\mcF'}+ \D')) \to  f'_*\mathcal{O}_{X'}(m(K_{X'/Y}+ \D')). 
\]  
Observe that the   morphism $g\colon X\to X'$ can be seen as the composition of  blowups of points. 
Thus by repeated application of Lemma \ref{lem-blowup-more-adjoint-foliation} we obtain the following isomorphism  
\[ 
g_*\mathcal{O}_{X}(m(K_{\mcF}+ \D)) \cong \mathcal{O}_{X'}(m(K_{\mcF'}+ \D')). \]
Since $(X',\D')$ is canonical, we also get  
\[
g_*\mathcal{O}_{X}(m(K_{X/Y}+ \D)) \cong  \mathcal{O}_{X'}(m(K_{X'/Y}+ \D')).
\]
Hence we deduce that 
\[ 
f_*\mathcal{O}_{X}(m(K_{\mcF}+ \D)) \cong  f_*\mathcal{O}_{X}(m(K_{X/Y}+ \D)). 
\]
This completes our proof.
\end{proof}

\subsection{Positivity of foliated canonical divisors}

We are now ready to prove Proposition \ref{prop-effective-adjoint-foliation-canonical} and  Proposition \ref{prop-log-orbifold-c2-semipositive}.    
We first consider a special case of Proposition \ref{prop-effective-adjoint-foliation-canonical} in the following lemma.
 
\begin{lemma}
\label{lem-effective-adjoint-foliation-canonical-horizontal-bimeromorphic}
Let $f\colon X \to Y$ be a projective fibration between normal  compact K\"ahler varieties such that general fibers of $f$ are isomorphic to $\mathbb{P}^1$   
and that $X$ is smooth. 
Assume that  there is a   $\mathbb{Q}$-divisor  $\D$  with  $\lfloor \D \rfloor =0$, 
such that   every component of  $\D$ is bimeromorphic  to $Y$  via  $f$, and that  $K_X+\D$ intersects the general fibers of $f$ non-negatively.
Let $\mcF$ be the foliation induced by  $f$ and $n$ the dimension of $X$. 
Then   $(K_{\mcF}+\D) \cdot \omega^{n-1} \> 0$ for any K\"ahler class $\omega$ on $X$.  
\end{lemma}

\begin{proof}
Since $\Delta$ is an effective  divisor, it is enough to assume that $K_X+\Delta$ is numerically trivial on general fibers of $f$.   
By using Lemma \ref{lemma:p-q-exceptional},  we construct the following commutative diagram
\begin{equation}\label{eqn:terminal-reduction}
\xymatrixrowsep{3pc}\xymatrixcolsep{3pc}\xymatrix{
(X',\D')\ar[d]_{g} \ar[r]^p  & (X,\D) \ar[d]^{f} \\
 Y'\ar[r]_{q} & Y
}
\end{equation}

\noindent such that 
\begin{enumerate}
\item $X'$ and $Y'$ are smooth,
\item  $p$ and $q$ are projective bimeromorphic,
\item every $g$-exceptional divisor is also $p$-exceptional,
\item $\D'$ is the strict transform of $\D$ such that the components of $\Delta'$ are smooth and pairwise disjoint; in particular, $(X', \Delta')$ has terminal singularities.
%\item $(X', \D')$ is terminal.
\end{enumerate}
Observe that the general fiber of $g$ are isomorphic to $\mbP^1$ and $K_{X'}+\D'$ is numerically trivial on the general fibers of $g$.
In particular, there is a sufficiently divisible positive integer $m$ such that $g_*\mathcal{O}_{X'}(m(K_{X'/Y'}+ \D'))$ is a rank one torsion-free sheaf   on $Y'$.  
Let $\mcG$ be the foliation induced by $g$. 
Then $K_{\mcG}+\Delta'$ is trivial on the general fibers of $g$. By Lemma \ref{lem-same-direct-image-horizontal-bimeromorphic}, 
\begin{equation}
g_*\mathcal{O}_{X'}(m(K_{\mcG}+ \D')) \to  g_*\mathcal{O}_{X'}(m(K_{X'/Y'}+ \D'))
\end{equation}
is an isomorphism in codimension one. 
Taking reflexive hulls we get the following isomorphism of line bundles on $Y'$
\begin{equation}\label{eqn:isomorphism-in-codim-one}
    (g_*\mathcal{O}_{X'}(m(K_{\mcG}+ \D')))^{**} \to  (g_*\mathcal{O}_{X'}(m(K_{X'/Y'}+ \D')))^{**}.
 \end{equation}
By \cite[Corollary 5.2.1]{PT18}, the torsion-free sheaf $g_*\mathcal{O}_{X'}(m(K_{X'/Y'}+ \D'))$  has a singular metric with semipositive curvature current. Thus by \cite[Remark 2.4.2(3)]{PT18}, the line bundle $\mcM:=(g_*\mathcal{O}_{X'}(m(K_{X'/Y'}+ \D')))^{**}$ has a singular metric with semipositive curvature current; in particular, the line bundle $\mcM$ is pseudoeffetive,  and so is $\mcL:=(g_*\mathcal{O}_{X'}(m(K_{\mcG}+ \D')))^{**}$. 

By Lemma \ref{lem:reflexive-pushforward}, there is a $g$-exceptional $\mbQ$-divisor  $E_1$ and an effective divisor $E_2$ on $X'$ such that 
\[
g^*\mcL\cong \mathcal{O}_{X'}(m(K_{\mcG}+ \D'+E_1-E_2)).
\] 
Also, observe that there is a $p$-exceptional $\mbQ$-divisor $F$ such that  
\[p^*\mcO_X(m(K_{\mcF}+\Delta))\cong\mcO_{X'}(m(K_{\mcG}+\Delta'+F)).\]
Thus for any K\"ahler class $\omega$ on $X$ we have 
\begin{eqnarray*}
(K_{\mcF}+\D)\cdot \omega^{n-1} &=&  p^*(K_{\mcF}+\D)\cdot (p^*\omega )^ {n-1} \\
&=&(K_{\mcG}+\D'+F) \cdot (p^*\omega)^{n-1}\\
&=& (K_{\mcG}+\D') \cdot (p^*\omega)^{n-1}.
\end{eqnarray*}
Since $E_1$ is $g$-exceptional, it is also $p$-exceptional by our construction. 
Thus we obtain that
\begin{eqnarray*}
    (K_{\mcG}+\Delta')\cdot (p^*\omega)^{n-1} &=& (K_{\mcG}+\Delta'+E_1) \cdot (p^*\omega)^{n-1}  \\ 
     &\>& (K_{\mcG}+\Delta'+E_1-E_2)\cdot (p^*\omega)^{n-1}\\
     &= &  \frac{1}{m} c_1(g^*\mcL) \cdot (p^*\omega)^{n-1}\\ &\>&0,
\end{eqnarray*}
This completes the proof of the lemma.        
\end{proof}

We improve this lemma in the following assertion, 
which only requires that every component of $\D$ is horizontal over $Y$. 

\begin{lemma}
\label{lem-effective-adjoint-foliation-canonical-regular}
Let $f\colon X \to Y$ be a projective fibration between normal  compact K\"ahler varieties such that general fibers of $f$ are isomorphic to $\mathbb{P}^1$ and that  $X$ is smooth. 
Assume that there is a  $\mbQ$-divisor $\D$ with $\lfloor \D \rfloor =0$,  
such  that   every component of  $\D$ is horizontal over $Y$,
and that  $K_X+\D$  has non-negative  intersection number  with the general fibers of $f$. 
Let $\mcF$ be the foliation corresponding to $f$ and $n$   the dimension of $X$.  
Then   $(K_{\mcF}+\D) \cdot \omega^{n-1} \> 0$ for any  nef class  $\omega$ on $X$.  
\end{lemma} 

\begin{proof}
By continuity, we may assume that $\omega$ is K\"ahler. 
The idea is to reduce to the case of Lemma \ref{lem-effective-adjoint-foliation-canonical-horizontal-bimeromorphic}.  
By applying \cite[Lemma 2.11]{DasOu2022}  to the irreducible components of $\D$ successively, 
and by taking a desingularization in the end, we can obtain  a commutative diagram 

\centerline{
\xymatrix{
%(X_3,\D_3)\ar[d]_{f_3} \ar[r]^{p_3}  & 
(X_2,\D_2)\ar[d]_{f_2} \ar[r]^{p_2}  & (X_1,\D_1) \ar[d]^{f_1}   \ar[r]^{p_1} 
& (X,\D) \ar[d]^{f}   \\
% Y_3\ar[r]_{q_3} &  
Y_2\ar[r]_{q_2} & Y_1 \ar[r]_{q_1}   & Y
}
}

\noindent such that 
\begin{enumerate}
\item $q_1$ and $p_1$ are finite,
\item $Y_1$ is normal, 
\item $X_1$ is the normalization of the main component of $X\times_{Y} Y_1$,
\item $X_2$ and $Y_2$ are smooth,
\item $p_2$ and $q_2$ are projective bimeromorphic,
\item %$p_1$ is \'etale over general points of each component of $\D$ and  
       {$\D_1 = p^*_1 \D$},
\item every components of $\D_1$ is bimeromorphic to $Y_1$  via  $f_1$,
\item $\D_2$ is the strict transform of $\D_1$. 
%\item $(X_2, \D_2)$ is klt.
\end{enumerate}

Let $p_1^{-1}\mcF$ be the pullback foliation on $X_1$ and let $\mcF_1$ be the foliation induced by $f_1$.  Since $p_1^{-1}\mcF$ is induced by $q_1\circ f_1$, it follows that $\mcF_1\subset p_1^{-1}\mcF$, and hence $\mcF_1=p_1^{-1}\mcF$ as both of them have same rank. 
Since  $p_1$ is \'etale over a dense Zariski open subset of $X$ of the form $f^{-1}(U)$, where $U\subset Y$ is a dense Zariski open subset of $Y$, it follows that the ramification divisor, say $R$,  of $p_1$ is vertical over  $Y_1$. 
Thus $R$ is $\mcF_1$-invariant. 
From   \cite[Lemma 3.4]{Dru21} and \cite[Proposition 2.2]{CS20} it follows that $K_{\mcF_1}=p_1^*K_{\mcF}$. 
Hence  we have
\[ 
 p_1^*(K_{\mcF }+\D ) = p_1^*K_{\mcF } + p_1^*\D  = K_{\mcF_1} + \D_1.
\]

Let $\mcF_2$ be the foliation induced by $f_2$ and let $g: X_2\to X$ be the composite morphism. 
Then from the computation above it follows that there is a $g$-exceptional $\mbQ$-divisor $E$ on $X_2$ such that 
\[g^* (K_{\mcF}+\D) = K_{\mcF_2} + \D_2+E.\] 
%for some sufficiently divisible $m\in\mbZ^+$.\\
Since  $(K_{\mcF_2}+\D_2) \cdot (g^*\omega)^{n-1} \> 0 $  by  Lemma \ref{lem-effective-adjoint-foliation-canonical-horizontal-bimeromorphic},  
we obtain that 
\[ 
(K_{\mcF}+\D) \cdot \omega^{n-1} \> 0.
\]
This completes our proof. 
\end{proof}

Now we can deduce Proposition \ref{prop-effective-adjoint-foliation-canonical}.

\begin{proof}[{Proof of Proposition \ref{prop-effective-adjoint-foliation-canonical}}] 
By Theorem \ref{thm-non-vanishing}, $K_X+\D$ is pseudoeffective if and only if $K_X+\D_{\hor}$ is pseudoeffective. Hence, without loss of generality, we may assume that $\D=\D_{\hor}$. 
Moreover, $K_X+\D$ has nonnegative intersection number with the general fibers of the MRC fibration.  
By \cite[Theorem 1.1]{HP15}, $X$ is bimeromorphic to a   Mori fiber space $\hat{f}\colon \widehat{X} \to \widehat{Y}$ 
such that  $\widehat Y$ is a normal compact   surface.  
Furthermore,  the  general  fibers of $\hat{f}$  
map to the general leaves of $\mcF$. 
By blowing up the graph of $X\bir \widehat X$,  we get the following diagram, 
\centerline{
\xymatrix{
(X',\D')\ar[d]_{f'} \ar[r]^p  & (X,\D)   \\
 Y' &  
}
}
\noindent such that 
\begin{enumerate}
\item $X'$ is smooth and $Y'=\widehat Y$,
\item  $p$ is proper  bimeromorphic,
\item $f'$ is projective with general fibers isomorphic to $\mathbb{P}^1$,
\item general fibers of $f'$ map to the leaves of $\mcF$  via  $p$,
%\item every $f'$-exceptional divisor is also $p$-exceptional,
\item $\D'$ is the strict transform of $\D$,  
\item $\lfloor \D' \rfloor = 0$.
%$(X', \D')$ is klt.
\end{enumerate}
Let $\omega'=p^*\omega$ and let $\mcF'$ be the foliation induced by $f'$. 
Observe that $\mcF'=p^{-1}\mcF$.
%, since both are foliation of same rank and isomorphic on a dense Zariski open subset of $X'$. 
Then there is a $p$-exceptional $\mbQ$-divisor $E$ on $X'$ such that
%\[
%\mcO_{X'}(m(K_{\mcF'}+\D'+E)) \cong p^*\mcO_X(m(K_{\mcF}+\D )) 
%\] 
%for some sufficiently divisible integer $m>0$. 
\[
K_{\mcF'}+\D'+E = p^*(K_{\mcF}+\D ). 
\] 
We note that $p:X'\to X$ is an isomorphism around the general fibers of the MRC fibration of $X$. Thus  $K_{X'}+\D'$ has nonnegative intersection number with the general fibers of $f'$.  
By Lemma \ref{lem-effective-adjoint-foliation-canonical-regular}, we have $(K_{\mcF'}+\D') \cdot \omega'^{2} \> 0$.  Since  $E$ is $p$-exceptional, we  conclude that    $(K_{\mcF}+\D) \cdot \omega^{2} \> 0$.
\end{proof}

Now we  can establish the generic nefness of some logarithmic cotangent sheaves in the following lemma.

\begin{lemma}
\label{lem-orbifold-semi-positivity}
Let $(X,B)$ be a lc pair with a reduced boundary $B$, where $X$ is a  uniruled non-algebraic compact K\"ahler threefold, with quotient singularities. 
Assume that there is some $0<\lambda <1$ such that   $K_X+ (1-\lambda) B$ is pseudoeffective.  
Let $\omega$ be any K\"ahler class on $X$. 
Then the reflexive logarithmic cotangent  sheaf $\Omega_X^{[1]}(\mathrm{log}\, B)$ is $\omega$-generically nef.
\end{lemma}

\begin{proof}
Let $\mathcal{E} \subseteq \Omega_X^{[1]}(\mathrm{log}\, B)$ be a proper non zero saturated subsheaf. 
We need to show that 
\[ {c}_1(\mathcal{E}) \cdot \omega^2 \< (K_X+B)\cdot \omega^2.\] 
Let $r\colon\widetilde{X} \to X$ be a log resolution of $(X,B)$ and let $\widetilde{B} = r^{-1}_*B$. 
Then   $K_{\widetilde{X}}+(1-\lambda) \widetilde{B}$ is pseudoeffective by Theorem \ref{thm-non-vanishing}.   
We remark that 
\[
(r_*\Omega_{\widetilde{X}}^1(\mathrm{log}\, \widetilde{B}))^{**} = \Omega_X^{[1]}(\mathrm{log}\, B).
\]
Hence, there is a saturated subsheaf $\widetilde{\mcE} \subseteq \Omega_{\widetilde{X}}^1(\mathrm{log}\, \widetilde{B})$ such that 
$ (r_* \widetilde{\mcE})^{**} = \mcE$.   
As a consequence, by the projection formula,  we only need to show that
\[
c_1(\widetilde{\mathcal{E}}) \cdot (r^*\omega)^2 \< (K_{\widetilde{X}}+\widetilde{B})\cdot (r^*\omega)^2.
\]

Thus, replacing $(X,B)$ by $(\widetilde{X}, \widetilde{B})$  if necessary, we  may assume that $(X,B)$ is log smooth. 
If the rank of $\mcE$ is $1$, then the inequality follows from \cite[Theorem 1.1]{CP16}. 
Now we assume that $\mathcal{E}$ has rank $2$. 
Assume by contradiction that 
\[
c_1({\mathcal{E}}) \cdot  \omega^2 > (K_{X}+ B)\cdot  \omega^2.
\]
Let $\mathcal{G} = (\Omega^{1}_{X}(\mathrm{log}\,  {B}) / \mathcal{E})^*$.  
Then $\mcG$ is a rank one saturated subsheaf of  the logarithmic tangent bundle  $T_{X}(-\log\,  {B})$, and we have 
\[
c_1(\mcG) \cdot  \omega^2  = -(K_X+B)\cdot  \omega^2  + c_1(\mathcal{E})\cdot  \omega^2 >0.
\] 
Let $\mcG \injective \mcF$ be the saturation of $\mcG$ in the tangent bundle $T_{X}$.  
Then $\mcF$ is a foliation on $X$, and  we have   
\[
c_1( \mcF) \cdot   \omega ^2 \> c_1(\mcG) \cdot   \omega^2 >0.  
\] 
Since $\mcF$ has rank one, it follows that $c_1(\mcF) = -K_{\mcF}$. 
By a criterion of Brunella (see \cite[Theorem 1.1]{Bru06}), $ \mcF $ is a foliation by rational curves.  
Therefore, it must be the foliation corresponding to  the MRC fibration of $X$.  

We remark that  $\mcG = \mcF \cap  T_{X}(-\log\, B)$. 
By \cite[Proposition 2.17]{Cla17}, we obtain that 
\[ 
c_1(\mcG) = -(K_{\mcF} + B_{\hor}), 
\] 
where $B_{\hor}$ is the  horizontal part of $B$ over the base of the MRC fibration.  
Since $(X, (1-\lambda) B)$ is klt and $K_{X}+(1-\lambda) B$ is pseudoeffective, 
by Proposition \ref{prop-effective-adjoint-foliation-canonical}, we obtain that 
\[
(K_{\mcF}+B_{\hor} )\cdot  \omega^2 \> (K_{\mcF}+ (1-\lambda)  {B}_{\hor}) \cdot  \omega^2 \> 0. 
\] 
Hence $c_1(\mcG) \cdot  \omega^2 \< 0$. We obtain a contradiction.
\end{proof}

We can now conclude Proposition \ref{prop-log-orbifold-c2-semipositive}

\begin{proof}[{Proof of Proposition \ref{prop-log-orbifold-c2-semipositive}}]
By Lemma \ref{lem-orbifold-semi-positivity}, $\Omega_{Y}^{[1]}(\mathrm{log}\, B)$ is  generically nef with respect to any nef class on $Y$. 
Furthermore, we remark that 
\[
\hat{c}_1(\Omega_{Y}^{[1]}(\mathrm{log}\, B)) = K_Y+B = \pi^*(K_X+\Delta) - E,
\]
where $E$ is some $\pi$-exceptional divisor.  
By assumption, we have $\pi^*L\cdot E\num 0$. 
Thus we can apply Proposition \ref{prop:psef-c2-general}, with $\mcE= \Omega_{Y}^{[1]}(\mathrm{log}\, B)$,  $\alpha=   \pi^*(K_X+\D)$ and $\beta=E$, 
 to conclude that  $\hat{c}_2(\Omega_{Y}^{[1]}(\mathrm{log}\, B)) \cdot \pi^*L \> 0$.  
\end{proof}

\subsection{The case of non-uniruled threefolds} 
 
We also have  the following variant of Proposition \ref{prop-log-orbifold-c2-semipositive}.

\begin{proposition}
\label{prop:psef-c2}
Let $Y$ be a   compact K\"ahler threefold with quotient singularities.
Assume the   following properties hold:
\begin{enumerate}
    \item There is a bimeromorphic morphism $\pi\colon Y \to X$ to a normal compact K\"ahler threefold  such that $L=k(K_X+\D)$ is Cartier and nef for some  reduced boundary $\D$ and some integer $k>0$. 
    \item $L^2\not\equiv 0$ and $L^3 = 0$. 
    \item $\pi^*L|_P \equiv 0$, in particular  $\pi^*L \cdot P \equiv 0$,  for any $\pi$-exceptional prime divisor $P$. 
    \item $Y$ is not uniruled. 
    \item  If  $B=\pi_*^{-1}\D$, then $\Omega_Y^{[1]}(\mathrm{log}\,  B)$ induces an orbifold vector bundle on the standard orbifold structure $Y_{\orb}$. 
\end{enumerate} 
Then  $\hat{c}_2(\Omega_Y^{[1]}(\mathrm{log}\,  B)) \cdot  \pi^*L \> 0$.   
\end{proposition}

\begin{proof}
By Lemma \ref{lemma:generic-nef-non-uniruled} below, we see that $\Omega_Y^{[1]}(\mathrm{log}\,  B)$ is generically nef with respect to   any nef class.   
There is a $\pi$-exceptional divisor $E$ such that  
\[\hat{c}_1(\Omega_Y^{[1]}(\mathrm{log}\,  B)) = K_Y+ B =  \pi^*(K_X+ \D) - E.\] 
Hence we can deduce the proposition by applying Proposition \ref{prop:psef-c2-general}, with $\mcE= \Omega_{Y}^{[1]}(\mathrm{log}\, B)$,  $\alpha=   \pi^*(K_X+\D)$ and $\beta=E$.
\end{proof}

\begin{lemma}
\label{lemma:generic-nef-non-uniruled}
Let $X$ be a compact K\"ahler threefold with quotient singularities. 
Assume that $X$ is not uniruled. 
Let $B$ be a reduced divisor. 
Then for any nef class $\alpha$ and any  torsion-free quotient $\Omega_X^{[1]}(\log\, B) \to \mathcal{Q}$ of coherent sheaves, we have $\hat{c}_1(\mathcal{Q})\cdot \alpha^2   \> 0$.
\end{lemma}

\begin{proof}
Let $\mcN\subseteq \mathcal{Q}$ be the subsheaf induced by the inclusion $\Omega_X^{[1]} \subseteq \Omega_X^{[1]}(\log\, B)$. 
Then 
\[
 \hat{c}_1(\mathcal{Q})\cdot \alpha^2   \>   \hat{c}_1(\mathcal{N})\cdot \alpha^2. 
\]
Hence, it is enough to prove the case when $B=0$.

Assume that $B=0$. 
Let $\widetilde{X} \to X$ be a desingularization. 
Since $X$ is not uniruled, $K_{\widetilde{X}}$ is pseudoeffective. 
By Theorem \ref{thm-non-vanishing-general-setting}, there is some effective divisor $D\sim_{\mathbb{Q}} K_{\widetilde{X}}$. 
Hence, by \cite[Theorem 1.4]{Enoki1988}, if $\Omega^1_{\widetilde{X}} \to \mathcal{Q}'$ is the induced quotient, then we have  $c_1(\mathcal{Q}') \cdot (r^*\alpha)^2 \> 0$. 
This implies that 
$\hat{c}_1(\mathcal{Q})\cdot \alpha^2   \> 0$.
\end{proof}

\section{Lower bounds on Euler characteristics}
\label{section:Euler-char}

In this section we will work under the following setup.

\begin{setup}
\label{set:setup-1}
    Let $X$ be a  $\mathbb{Q}$-factorial compact  K\"ahler threefold, and $\Delta$ a reduced boundary on $X$. 
Assume that the following properties hold.
\begin{enumerate}
\item $(X ,\Delta )$ is lc.
\item $X$ has terminal singularities outside the support of  $\Delta$.
\item $L:=k(K_X+\Delta)$ is Cartier  and nef for some integer $k>0$.
\item There is an effective $\mbQ$-divisor $D \sim_\mathbb{Q} K_{X }+\Delta $ such that   $D_{\red}=\Delta$.
\item $\nu(X, K_X+\Delta) =2$. 
\item If $C\subset X$ is a curve such that $(K_X+\Delta)\cdot C>0$, then $(X, \Delta)$ is dlt near the general points of $C$. 
\end{enumerate}
\end{setup}

Our goal is to prove the following proposition. 
When $X$ is projective, this is an important step in \cite[Chapter 14]{Kol92}, where the algebraic abundance was established. 
In the case of K\"ahler threefolds, the difficulty is that   Bogomolov-Gieseker type inequalities for klt K\"ahler varieties are still unknown. 
Our solution is to construct a modification $Y$ of $X$, which has quotient singularities only, see Lemma \ref{lemma:construction-of-Y-1}. 
Afterwards, we can apply orbifold Bogomolov-Gieseker type inequalities on $Y$.
%as in \cite{Faulk2022}.

\begin{proposition}
\label{prop:Euler-characteristic-combined-1} 
With the notations and hypothesis as in Setup \ref{set:setup-1},  we assume further that $K_X+(1-\epsilon)\D$ is nef for any $0<\epsilon\ll 1$. 
Then there is a constant $\lambda$ such that  $\chi(X, \mathscr{O}_{{X}}(n L) ) \> \lambda$ for all $n \> 0$. 
Furthermore, if there is a curve $C\subseteq X_{\sing}$ such that $C\cdot L >0$, then $\chi(X, \mathscr{O}_{{X}}(n L) ) $ grows at least linearly in $n$.
\end{proposition}

\subsection{Elementary properties}
Before going into the proof of Proposition \ref{prop:Euler-characteristic-combined-1}, we  first show some   consequences of the hypothesis in Setup \ref{set:setup-1}.

\begin{lemma}\label{lem:nu-2-computation-1}
  Assume that we are in  Setup \ref{set:setup-1}.  
  Then $L^2 \cdot \Delta_i = 0$ for any irreducible component $\Delta_i$ of $\D$. 
  If in addition  $K_X+(1-\epsilon)\D$ is nef for any $0<\epsilon\ll 1$,
  then $K_X^2\cdot L = 0$. 
\end{lemma}

\begin{proof}
 Since $\nu(X, K_X+\Delta)=2$, it follows that 
   $L^2\cdot D = L^3 = 0.$ 
  Since $L$ is nef and   $D$ is effective, by the property (4),  
  we deduce that $L^2\cdot \Delta_i = 0$ for any irreducible component $\Delta_i$ of $\D$.  

  Now we assume further that $K_X+(1-\epsilon)\D$ is nef for any $0< \epsilon \ll 1$. 
  Then $\nu(X, K_X+(1-\eps)\Delta)\<2$, and thus 
  \[(K_X+(1-\eps)\Delta)^2\cdot(D-\epsilon\Delta) = (K_X+(1-\eps)\Delta)^3=0\] for all $0<\eps\ll 1$.
  Since $(D-\epsilon\Delta)$ is effective for all $0<\eps\ll 1$,   we deduce that  
  \[(K_X+(1-\eps)\Delta)^2\cdot\Delta_i=0\] for every component $\Delta_i$ of $\Delta$. 
  The LHS above is a  polynomial in $\eps$ with infinitely many solutions,  so its coefficients must all be $0$: 
\[K_X^2\cdot \Delta_i = \Delta^2 \cdot \Delta_i= K_X\cdot \Delta \cdot \Delta_i = 0. %\quad\mbox{for all } i.
\]  
It  then follows that $K_X^2\cdot L =0$.  
\end{proof}

The following  results   are extracted from  \cite[Lemma 14.2]{Kol92}. 

\begin{setup}\label{set:setup-2}
    Assume that we are in  Setup \ref{set:setup-1}. 
Let $\Gamma$ be an irreducible component of $\D$ such that $\nu(\Gamma, L|_\Gamma) >0$. 
By adjunction, there is a boundary divisor $\Theta$ on $\Gamma$ such that $(\Gamma, \Theta)$ is slc (see \cite[Remark 1.2(5)]{Fuj00}) and  
\[  
(K_X+\D)|_\Gamma \sim_{\mathbb{Q}} K_{\Gamma} + \Theta. 
\] 
Let $\pi\colon \widehat{\Gamma} \to \Gamma$ be the normalization. 
Then there is a boundary $\widehat{\Theta}$ such that  $(\widehat{\Gamma}, \widehat{\Theta})$ is lc and that $K_{\widehat{\Theta}} + \widehat{\Gamma} \sim_{\mathbb{Q}} \pi^*(K_X+\D)$.  
By log abundance for K\"ahler surfaces (see \cite[Theorem 1.33]{DasOu2022}), 
$K_{\widehat{\Theta}} + \widehat{\Gamma}$ induces a fibration $\hat{f}\colon \widehat{\Gamma} \to B$.
\end{setup}

\begin{lemma} 
\label{lemma:surface-in-Delta-Theta-hor}
With the notations and hypothesis of Setup \ref{set:setup-2}, assume that a general fiber $\widehat{F}$ of $\hat{f}$ is a  rational curve. 
\begin{enumerate}
    \item Then  $\Gamma$ is normal  around $F:= \pi(\widehat{F})$. 
\end{enumerate}
In particular, we have an almost proper meromorphic map $f\colon \Gamma \dashrightarrow B$ (i.e. the general fibers are proper) induced by $\hat{f}$. 
Let $\Theta_{\hor}$ be the horizontal part of $\Theta$ over $B$. 
We decompose  
\[\Theta_{\hor} = \lfloor \Theta_{\hor} \rfloor  + \Theta_{\hor}^{<1}\] 
into the sum of its integral part and its fractional part.  
Then the following properties hold. 
\begin{enumerate}
\item[(2)] The pair $(X,\D)$ is log smooth around general points of $\lfloor \Theta_{\hor} \rfloor$. 
\item[(3)] 
Let $\Xi$ be an  irreducible  component   of $\Theta_{\hor}$. 
Then $\Xi$ is contained in the support of $\D-\Gamma$ if and only if it is a component of $\lfloor \Theta_{\hor} \rfloor$.  
\item[(4)] $\Theta_{\hor}^{<1} \subseteq X_{\sing}$. 
\item[(5)] If $C \subseteq X_{\sing}$ is a curve which intersects $F$, then $C$ is a component of $\Theta_{\hor}^{<1}$. 
\item[(6)]
Assume that $C$ is an irreducible component of $\Theta_{\hor}^{<1}$. 
Then the coefficient of $C$ is $1-\frac{1}{m}$ for some positive integer $m$.
In an Euclidean neighborhood of a general point of  $C$, we have an analytic isomorphism $(X,\Delta, C)\cong (S\times \mathbb{D}, E\times \mathbb{D},  \{o\}\times \mathbb{D})$, 
where $\mathbb{D} \subseteq \mathbb{C}$ is the open unit disk, $(o\in S)$ is a germ of surface singularities isomorphic to $(0\in \mathbb{C}^2/\mathbb{Z}_m)$. 
The action of $\mathbb{Z}_m$ on $\mathbb{C}^2$ is of weight $(1,q)$ with $q$ relatively prime to $m$. 
The curve $E\subseteq S$ corresponds to $\{(a,b)\in \mathbb{C}^2 \ |\ a=0 \}$. 
Furthermore $m$ is also the  Cartier index of $\Delta$ around a general point of $C$.  
\end{enumerate}
\end{lemma}

\begin{proof}
We note that $\pi(\Supp \widehat{\Theta})$ is contained in the union of $X_{\sing}$ and the non-normal locus of $\Delta$.  
We also observe that 
 \[K_{\widehat \Gamma}+\widehat{\Theta}\num a\widehat{F}\] in $\widehat{\Gamma}$ for some $a>0$.   
Let $ \widehat{\Theta}_{\hor} $ be the horizontal part of $\widehat{\Theta}$ over $B$.  
Assume that $\widehat{\Xi}$ is an irreducible component of $\lfloor \widehat{\Theta}_{\hor} \rfloor$.
Then we write $\Xi=\pi(\widehat{\Xi})$ and we  get $L\cdot  \Xi  >0$. 
Hence $(X,\D)$ is dlt around general points of $\Xi$ by the property (6) of Setup \ref{set:setup-1}. 
From the discussion of \cite[Proposition 16.6]{Kol92}, we deduce that $(X,\D)$ is log smooth around the general points of $\Xi$,  
and that there is a  component $\Gamma_1$ of $\D-\Gamma$ such that $\Xi\subseteq \Gamma_1 \cap \Gamma$. 
Conversely, assume that there is a component $\Gamma_1$ of $\D-\Gamma$ such that there is a curve $\Xi \subseteq \Gamma_1\cap \Gamma$ with $L\cdot \Xi >0$. 
Then, since  $(X,\D)$ is dlt around the general points of $\Xi$,  and since $\Xi$ is a lc center of $(X, \D)$, it follows that $(X, \D)$ is log smooth  around general points of $\Xi$. 
Therefore, if $\widehat{\Xi}$ is the pullback of $\Xi$ in $\widehat{\Gamma}$, then it is a component of $\lfloor\widehat \Theta_{\hor} \rfloor$. 
We hence obtain a characterization of the irreducible components of $\lfloor \widehat \Theta_{\hor} \rfloor$.

We notice that $\Gamma$ has normal crossing singularities in codimension one as $(\Gamma, \Theta)$ is a slc pair. 
% singularities (see \cite[Proposition ]{Fuj00}).   
Thus the one-dimensional component of the non-normal locus of $\Gamma$ is contained in  $\pi(\lfloor \widehat{\Theta} \rfloor)$. 
Since $\Gamma$ is smooth around general points of each components of $\pi(\lfloor  \widehat{\Theta}{\hor} \rfloor)$, 
we   deduce that $\Gamma$ is normal around $F$. 
This proves the item (1). 

Since $\Gamma$ is normal around $F$, the argument of the first paragraph implies the items (2)-(4). 
The item (6) follows from \cite[Proposition 16.6]{Kol92}, more specifically, from \cite[(16.6.3)]{Kol92}. 
To see the item (5), we notice that 
\[ 
C\cdot (K_X+\D) = a\cdot C\cdot F >0. 
\]
Hence by  the property (6) of Setup \ref{set:setup-1}, $(X,\D)$ is dlt around $C$. 
In addition, since $\widehat{F}$ is a general fiber, we see that $C\subseteq \Gamma$. 
From \cite[(16.6.1) and (16.6.3))]{Kol92}, we deduce that $C$ is a component of $\Theta_{\hor}^{<1}$. 
This completes our proof. 
\end{proof}

\begin{lemma}
\label{lemma:surface-in-Delta-classification}
With the assumption in Lemma \ref{lemma:surface-in-Delta-Theta-hor}, we can classify the following three cases. 

\begin{enumerate}
  \item[1]   We have $\Theta_{\hor} = \lfloor \Theta_{\hor} \rfloor$. 
In this case, $(X,\D)$ is log smooth around $F$. 

\item[2]    We have $  \lfloor \Theta_{\hor} \rfloor \cdot F =1$.  
In this case,   we can decompose  
\[\Theta_{\hor}^{<1} = \sum_i \left(1-\frac{1}{m_i}\right) \Theta_i\] 
into irreducible components. 
If $b_i=\Theta_i \cdot F$, then we have 
\begin{equation} \label{eqn:adjunction-2}
     \sum_i \left(1-\frac{1}{m_i}\right)b_i = 1.
\end{equation}   
As a consequence, $m_i=2$ for each $i$, and $X$ has canonical singularities around $F$.

\item[3]  We have $ \lfloor \Theta_{\hor}\rfloor  =0$. 
   In this case, we can decompose 
   \[\Theta_{\hor}^{<1} = \sum_i \left(1-\frac{1}{m_i}\right) \Theta_i.\] 
   If $b_i=\Theta_i \cdot F$, then we have 
   \begin{equation}\label{eqn:adjunction-3}
       \sum_i \left(1-\frac{1}{m_i}\right)b_i = 2.
   \end{equation}
   As a consequence, $m_i\in \{2,3,4,6\}$ for each $i$. 
\end{enumerate} 
\end{lemma}

\begin{proof}
We see that $F\cong \mathbb{P}^1$ and $F\cdot \Theta_{\hor} = 2$. 
Thus   the intersection number $ \beta = F\cdot \lfloor \Theta_{\hor} \rfloor$ belongs to $\{2,1,0\}$. 
If $\beta = 2$, then we are in the Case $1$. 
From Lemma \ref{lemma:surface-in-Delta-Theta-hor}, we deduce that $(X,\D)$ is log smooth around $F$.  

If $\beta=1$, then we are in the Case 2, and  get the equation (\ref{eqn:adjunction-2}). 
A direct computation shows that $m_i=2$ for all $i$. 
Since any surface quotient singularities by the  group $\mathbb{Z}_2$ is Du Val, from the classification of Lemma \ref{lemma:surface-in-Delta-Theta-hor},  we conclude that $X$ has canonical singularities around $F$. 

If $\beta=0$, then we are in the Case 3, and  get the equation (\ref{eqn:adjunction-3}).  
This completes the proof of the lemma.
\end{proof}

Thanks to the previous two lemmas, we can deduce the following assertions on curves contained in $X_{\sing}$.

\begin{lemma}
\label{lemma:singularities-are-cyclic-1}
Assume that we are in the Setup \ref{set:setup-1}. 
Let $C\subseteq X_{\sing}$ be a curve such that $L\cdot C >0$. 
Then in an Euclidean neighborhood of a general point of  $C$, we have an analytic isomorphism $(X,\Delta, C)\cong (S\times \mathbb{D}, E\times \mathbb{D},  \{o\}\times \mathbb{D})$, 
where $\mathbb{D} \subseteq \mathbb{C}$ is the open unit disk, $(o\in S)$ is a germ of surface singularities isomorphic to $(0\in \mathbb{C}^2/\mathbb{Z}_m)$ with $m\in \{2,3,4,6\}$.  
The action of $\mathbb{Z}_m$ on $\mathbb{C}^2$ is of weight $(1,q)$ with $q$ relatively prime to $m$. 
The curve $E\subseteq S$ corresponds to $\{(a,b)\in \mathbb{C}^2 \ |\ a=0 \}$. 
Furthermore, the index $m$ is the  Cartier index of $\Delta$ around a general point of $C$. 

Moreover, if $m=6$, then there is a unique component $\Gamma$ of $\D$ containing $C$, there are unique curves $C_1, C_2 \subseteq X_\textsubscript{\rm sing} \cap \Gamma$  such that 
\[
L\cdot C =L\cdot C_1 = L\cdot C_2. 
\] 
If $m_1$ and $m_2$ are the indices of $C_1$ and $C_2$, i.e. the Cartier indices of $\Delta$ around the general points of $C_1$ and $C_2$ respectively, then we have $m_1=2$ and $m_2=3$. 
Additionally, such associations $C \mapsto \Gamma$, $C \mapsto C_1$ and $C\mapsto C_2$ are one-to-one. 
\end{lemma} 

\begin{proof}
Since $X$ has terminal singularities outside of the support of $\Delta$, 
 there is a component $\Gamma$ of $\Delta$ such that $C \subseteq \Gamma$. 
Let $\pi \colon \widehat{\Gamma} \to \Gamma$ be the normalization.
Since $(X,\D)$ is dlt around  general points of $C$ (see (6) of Setup \ref{set:setup-1}), we see that $\Gamma$ is normal around general points of $C$. 
Thus we may set $\widehat{C} = \pi^{-1}_*C$. 

By adjunction, there is a boundary $\widehat{\Theta}$ on $\widehat{\Gamma}$ such that 
\[\pi^* (K_X+\Delta)  \sim_{\mathbb{Q}} K_{\widehat{\Gamma}} + \widehat{\Theta}\] 
and that $(\widehat{\Gamma}, \widehat{\Theta})$ is lc. 
We claim that the numerical dimension $\nu(\Gamma, K_{\widehat \Gamma}+\widehat{\Theta})=1$. 
Indeed, if not, then it is either $0$ or $2$. 
If $\nu(\widehat{\Gamma}, K_{\widehat \Gamma}+\widehat{\Theta})=0$, then 
\[ L\cdot C= k \cdot \pi^*(K_X+\Delta) \cdot\widehat C=0,\]  which is a contradiction. 
If $\nu(\widehat{\Gamma}, K_{\widehat \Gamma}+\widehat{\Theta})=2$, then $L^2\cdot \Gamma>0$. 
This contradicts Lemma \ref{lem:nu-2-computation-1}. 

Thus by log abundance for K\"ahler surfaces (see \cite[Theorem 1.33]{DasOu2022}), $\pi^*L$ induces a fibration $\hat{f}\colon \widehat{\Gamma} \to B$ to a smooth projective curve $B$. 
From \cite[Proposition 16.6]{Kol92}, we deduce that $\widehat{C}$ is a component of $\widehat{\Theta}$. 
The condition $L\cdot C >0$ implies that $\widehat{C}$ is horizontal over $B$. 
It then follows that general fibers of $\hat{f}$ are smooth rational curves. 
Therefore, for the surface $\Gamma$, we are in the situation of Lemma \ref{lemma:surface-in-Delta-Theta-hor} and Lemma \ref{lemma:surface-in-Delta-classification}.  
We can hence deduce the first part of the lemma.

Let $F$ be a general fiber of the  map $\Gamma \dashrightarrow B$ induced by $\hat{f}$. 
Then there is some $a>0$ such that 
\[(K_X+\D)|_\Gamma \equiv aF.\]
For the second part of the lemma, we assume that the coefficient of ${C}$ in $\Theta_{\hor}$ is $(1-\frac{1}{6})$. 
Then we must be in the Case 3 of Lemma \ref{lemma:surface-in-Delta-classification}, and the equation (\ref{eqn:adjunction-3}) takes the following form:
 \[
 \left(1-\frac{1}{2}\right)+\left(1-\frac{1}{3}\right)+\left(1-\frac{1}{6}\right)=2.
 \]
Therefore,  $\Theta_{\hor}$ has exactly $3$ components including $ C$.  
 Let the other two be $C_1$ and $ C_2$. 
 Then \[\Theta_{\hor}=\left(1-\frac12\right)  C_1+\left(1-\frac13\right) C_2+\left(1-\frac 16\right) C,\] 
 and 
 \[ C\cdot F= C_1\cdot F= C_2\cdot F=1.\]  
 Then by Lemma \ref{lemma:surface-in-Delta-Theta-hor} again, it follows that $C_1, C_2 \subseteq X_{\sing}$, and the Cartier index of $\Delta$ near the general points of $C_1$ and $C_2$ are $2$ and $3$, respectively. 
 Furthermore 
 \[L\cdot C_1= k(K_{ \Gamma}+\Theta)\cdot C_1 =ka\cdot F\cdot C_1=ka=L\cdot C_2=L\cdot C>0.\] 
 Finally, to see that the associations are one-to-one, we first observe that, by (3) of Lemma \ref{lemma:surface-in-Delta-Theta-hor},  $\Gamma$ is uniquely determined. 
 Afterwards,   the curves $C,C_1,C_2$ are exactly the irreducible components  of $\Theta_{\hor}$, with distinct coefficients. 
 This completes the proof of the lemma.
\end{proof}

The next two lemmas show that the conditions of Setup  \ref{set:setup-1} are  preserved under certain MMP.  
We will discuss  divisorial contractions and flips separately.

\begin{lemma}
\label{lemma:contraction-dlt}
Assume that we are in the Setup \ref{set:setup-1}.  
Suppose that $\varphi \colon (X,\D) \to (X', \D')$ is an elementary divisorial contraction, which is $K_X$-negative and $(K_X+\D)$-trivial. 
Then $(X',\D')$ also satisfies the conditions of Setup \ref{set:setup-1}. 
\end{lemma}

\begin{proof}
Since the contraction $\varphi$ is $(K_X+\D)$-trivial, we deduce that $(X', \D')$ satisfies the conditions (1)-(5) of Setup \ref{set:setup-1}.  
It remains to show the condition (6). 
Without loss of generality, we may assume that the exceptional locus of $\varphi^{-1}$ is a curve $C'\subseteq X'$ with 
\[ (K_{X'}+\D') \cdot C' >0. \]
We need to show that $(X', \D')$ is dlt around general points of $C'$. 
Let $\Gamma$ be the exceptional divisor of $\varphi$. 
Then there is some $b>0$ such that 
\[ K_X=\varphi^*K_{X'} +b\Gamma \mbox{ and }  K_X+\D = \varphi^*(K_{X'}+\D') .\]
Hence, we only need to consider the case when $\Gamma$ is a component of $\D$. 

Let $\pi\colon \widehat{\Gamma} \to \Gamma$ be the normalization. 
By adjunction, there is a boundary divisor $\widehat{\Theta}$ on $\widehat{\Gamma}$ such that 
\[
\pi^*(K_X+\D) \sim_{\mathbb{Q}} K_{\widehat{\Gamma}} + \widehat{\Theta}, 
\]
and $(\widehat{\Gamma}, \widehat{\Theta})$ is lc. 
Since $(K_{X'}+\D') \cdot C' >0$, by Lemma \ref{lem:nu-2-computation-1},  
we see that $\nu(\widehat{\Gamma}, K_{\widehat{\Gamma}}+\widehat{\Theta}) =1$. 
Hence, by log abundance for K\"ahler surfaces (see \cite[Theorem 1.33]{DasOu2022}), 
$\pi^*L$ induces a fibration $\hat{f}\colon \widehat{\Gamma} \to B$ to a smooth projective curve $B$. 
Let $F$ be a general fiber of $\Gamma \to C'$. 
Then it is a rational curve since $\varphi\colon X\to X'$ is an elementary contraction. 
Since $L\cdot F= 0$, we see that  $\widehat{F}:= \pi^{-1}(F)\subseteq \widehat{\Gamma}$ is a  fiber of $f\colon \widehat{\Gamma} \to B$.    We note that $\widehat F$ is a rational curve as well.  
Hence, for the surface $\Gamma$, we are in the situation of Lemma \ref{lemma:surface-in-Delta-Theta-hor} and Lemma \ref{lemma:surface-in-Delta-classification}.  
In the remainder of the proof,  we will go through each case of   Lemma \ref{lemma:surface-in-Delta-classification}. 

The condition $K_X\cdot F <0$ implies that $(D-\epsilon \D) \cdot F<0$   for $\epsilon>0$ small enough. 
Since $D-\epsilon \D$ is effective, it follows that $\Gamma\cdot F<0$. 
Moreover, since $D\cdot F=0$, $F$ must intersect a component of $\D-\Gamma$. 
In particular, by the item (3) of Lemma \ref{lemma:surface-in-Delta-Theta-hor}, we are not in the Case 3 of the proof of Lemma \ref{lemma:surface-in-Delta-classification}. 

If we are in the Case 1, then $X$ is   smooth in an open neighborhood of $F$. 
It follows that $X'$ has terminal singularities around $\varphi(F)$, hence around general points $C'$. 
Since $(X',\D')$ is a reduced lc pair, from the discussion of \cite[Proposition 16.6]{Kol92}, 
we obtain that $(X', \D')$ is log smooth around general points of $C'$. 

If we are in the Case 2, then $X$ has canonical singularities along  $F$. 
Thus we deduce that $X$ has terminal singularities in an open neighborhood of $\varphi(F)$. 
Hence $X'$ is smooth around general points of $C'$. 
Now we can argue as in the Case 1 to show that $(X', \D')$ is log smooth around general points of $C'$.  
This completes the proof of the lemma. 
\end{proof}

\begin{lemma}\label{lem:flip-invariance}
 Assume that we are in the Setup \ref{set:setup-1}. 
 Suppose that $\vphi \colon  X\bir X'$ is a $(K_X+\Delta)$-trivial $K_X$-negative flip and $\Delta':=\vphi_*\Delta$. 
 Then for every curve $C' \subseteq X'$ such that $ (K_{X'}+\Delta')\cdot C'  >0$, the map $\varphi^{-1}$ is an isomorphism around the general points of $C'$.  
 In particular, $(X', \Delta')$ also satisfies the conditions of Setup \ref{set:setup-1}.    
\end{lemma}

\begin{proof}
One sees that the conditions (1)-(5) of Setup \ref{set:setup-1} are automatically satisfied by $(X', \Delta')$. 
Let  $f\colon X\to Y$ be the $K_X$-negative flipping contraction and $f' \colon  X'\to Y$ the corresponding flip. 
Observe that $K_X+\Delta\num_f 0$ and $K_{X'}+\Delta'\num_{f'} 0$. 
So if there is a curve $C'\subset X'$ such that $(K_{X'}+\Delta')\cdot C'>0$, then $C'$ is not contained in $\Ex(f')$.  
In particular, $\vphi^{-1}\colon X'\bir X$ is an isomorphism near the general points of $C'$. 

Let $C\subset X$ be the strict transform of $C'$ under $\vphi^{-1}$. 
Then by passing to the graph of $\vphi$, we see that 
\[ (K_X+\Delta)\cdot C  = (K_{X'}+\Delta')\cdot C'  >0.\]
Thus, if $C'\subseteq X'_{\sing}$, then  $(X, \Delta)$ has dlt singularities near the general points of $C$  by our hypothesis.   
Hence $(X', \Delta')$ has dlt singularities near the general points of $C'$ as $\vphi^{-1}\colon X' \bir X$ is an isomorphism there.  
This completes our proof. 
\end{proof}

The following lemma computes some relative Chern classes.

\begin{lemma}
\label{lemma:possilbe-a_i-1}
Let $(o\in S,E)$ be a germ of the surface quotient singularities obtained in Lemma \ref{lemma:singularities-are-cyclic-1}, 
and let $\mu\colon \widetilde{S} \to S$ be a desingularization. 
Let $\Delta_S:=\Ex(\mu)$, and let $\ell(\widetilde{S}, \mu) = c_1^2(\widetilde{S}, \mu) + c_2(\widetilde{S}, \mu)$ be the relative class defined in \eqref{eqn:l-function}.
Then we have the following possible values for $\ell(\widetilde{S}, \mu)$,  via  the natural isomorphisms  
$H^4(\widetilde{S}, \widetilde{S}\setminus \Delta_S, \mathbb{R}) \cong H^{BM}_0(\Delta_S, \mathbb{R}) \cong \mathbb{R}.$ 
\begin{enumerate}
\item $\ell(\widetilde{S}, \mu) = \frac{3}{2}$ if $m=2$.
\item $\ell(\widetilde{S}, \mu) \in \{ \frac{4}{3}, \frac{8}{3} \}$ if $m=3$.
\item $\ell(\widetilde{S}, \mu) \in \{ \frac{3}{4}, \frac{15}{4} \}$ if $m=4$. 
\item $\ell(\widetilde{S}, \mu) \in \{ \frac{35}{6}, -\frac{5}{6} \}$ if $m=6$.
\end{enumerate}
\end{lemma} 
\begin{proof}
By \cite[Lemma 9.9]{Kaw88}, there is an Euclidean open embedding from $ S$ to  $\mathbb{C}^2/G$, where $G$ is a finite group acting linearly on $\mathbb{C}^2$. 
Let $T$ be a projective compactification of $\mathbb{C}^2/G.$ 
Then there is an Euclidean open embedding $\iota\colon S \to T$. 
By taking a partial desingularization, we may assume that $T$ has exactly one singular point $\iota(o)$.  
By abuse of notation, we still denote by $\mu\colon \widetilde{T} \to T$ a desingularization, such that $\mu^{-1}(S) = \widetilde{S}$. 
Since the Chern classes commute with pullbacks, we obtain that 
\[
\ell(\widetilde{S}, \mu) = \iota^*\ell(\widetilde{T}, \mu) \in H^4(\widetilde{S}, \widetilde{S}\setminus \D_S, \mathbb{R}). 
\]
%Moreover, by excision, the pullback
%\[  \iota^*\colon  H^4(\widetilde{T}, \widetilde{T}\setminus  \D_S, \mathbb{R}) \to H^4(\widetilde{S}, \widetilde{S}\setminus \{o\}, \mathbb{R}) \]
%is an isomorphism. 
Therefore, replacing $S$ by $T$, we may assume that $S$ is a projective surface with exactly one singular point $o$. 
Now we can apply the formula of \cite[Equation (14.3.1.2)]{Kol92} to conclude the lemma. 
\end{proof}

We will need the following lemma in the next subsection. 

\begin{lemma}
    \label{lemma:orbifold-snc}
   Let $X$ be a complex manifold and $G$ a finite group acting faithfully on $X$. 
  % Let $Z\subseteq X$ be the smallest closed subset such that $G$ acts freely on $X\setminus Z$,  and let $\D'$ be the codimension one part of  $Z$. 
  Let $\D'$ be a reduced divisor in $X$ such that the action of $G$ is free in codimension one on $X\setminus \mathrm{Supp}\, \D'$. 
   Assume that  $\D$ is a reduced $G$-invariant divisor such that $(X,\D+\D')$ is a reduced snc pair. 
   
   Let $(Y,B) = (X/G, \D/G)$ be the  pair of quotient spaces.  
   If $(V, H,\pi)$ is a standard orbifold chart over $Y$, then $\pi^{-1}(B)$  has normal crossing  singularities. 
   In particular, $\Omega_Y^{[1]}(\log\, B)$ induces  an orbifold vector bundle on the standard orbifold structure of $Y$. 
\end{lemma}

\begin{proof} 
Without loss of generality, we may assume that $\D'$ is $G$-invariant.  
We remark that the problem is local on $X$.  
Indeed, for a point $x\in X$, there is an open  connected neighborhood $U$ of $x$ such that $g(U)= U$ whenever $g(x)=x$ for any element $g\in G$.  
Let $G'\subseteq G$ be the maximal subgroups fixing $x$.
%, and let  $G_U\subseteq G_x$ be the maximal subgroup acting trivially on $U$.  
Then $U/G' \cong X/G$ around the image of $x$ in $Y$. 
%, where $G'=G_x/G_U$.   
Hence, up to replacing $(X,G)$ by $(U,G')$, and by \cite[Lemma 9.9]{Kaw88}, we may assume that $X$ is an open neighborhood of the origin in $\mathbb{C}^n$, that $G$ is a subgroup of $\mathrm{GL}_n(\mathbb{C})$, and that $\D+\D'$ is the union of coordinates hyperplanes.  
Let $R\subseteq G$ be the   subgroup generated by pseudo-reflections, which must be a normal subgroup.  
Thanks to the Chevalley–Shephard–Todd theorem, we deduce that $(V:=X/R,  H:= G/R)$ is a standard orbifold chart  of  $Y$.   

Without loss of generality, we can assume that $R$ is not trivial. 
In particular, $\D'\neq 0$. 
Let $\D_i$ be an irreducible component of $\D$. 
We claim that every pseudo-reflections $g\in G$ leaves $\D_i$ invariant. 
Suppose the opposite, then we must have $g(\D_i)=\D_j$ for some different component $\D_j$ of $\D+\D'$. 
If $H_g$ is the hyperplane of fixed points of $g$, then $H_g\neq \Delta_i$, $H_g\neq \Delta_j$, $H_g$ is a component of $\Delta'$ and $g(\D_i\cap H_g) \subseteq \D_j \cap H_g$. 
Since $g$ fixes every point in $H_g$, we get  $\D_i\cap H_g \subseteq \D_j \cap H_g$.  
This is a contradiction, for $\D_i+\D_j+H_g$ is snc and reduced.

We denote by $q\colon X\to V$ the natural morphism. 
Then $\pi^{-1}(B) = q(\D)$. 
Let   $S$ be a stratum of $\D$,  it is enough to show that the image $q(S)$ is smooth.  
The previous paragraph implies that $R$ leaves $S$ invariant, and the action of $R$ on $S$ is generated by quasi-reflections. 
Hence $q(S)=S/R$ is smooth by the  Chevalley–Shephard–Todd theorem.  
This completes the proof of the lemma.
\end{proof}

\subsection{Existence of lower bound for the Euler characteristics}

As mentioned at the beginning of the section, to prove Proposition \ref{prop:Euler-characteristic-combined-1},   we will construct a modification $Y$ of $X$, which has only quotient singularities.

\begin{lemma}
\label{lemma:construction-of-Y-1}
Assume that we are in the Setup \ref{set:setup-1}. 
Then there exists a  projective bimeromorphic morphism $\pi\colon Y \to X$ such that the following properties hold:  
\begin{enumerate}
\item $Y$ has cyclic quotient singularities only, and $(Y,B)$ is lc, where $B=\pi_*^{-1}\Delta$.  
\item The reflexive sheaf $\Omega_Y^{[1]}(\log\, B)$ induces an orbifold vector bundle on the standard orbifold structure.
\item If $C\subseteq X_{\sing}$ is a curve with $C\cdot L >0$, then $\pi^{-1}$ is an isomorphism  around the general points of $C$.
\item If $C_Y\subseteq Y_{\sing}$ is a curve with    $C_Y \cdot \pi^* L >0$, 
then $\pi$ is an isomorphism around the general points of $C_Y$. 
\item For any $\pi$-exceptional divisor $P$, we have $\pi^*L|_P \equiv 0$.
\item If we suppose further that $K_X+(1-\epsilon)\D$ is nef for all $0<\epsilon\ll 1$, then  $K_Y^2 \cdot \pi^*L = 0$. 
\end{enumerate}
\end{lemma}

\begin{proof}
Let $\rho \colon Z \to X$ be a $\mathbb{Q}$-factorial terminalization of $X$ such that $K_Z+\Gamma = \rho^* K_X$ (see \cite[Theorem 1.15]{DasOu2022}). 
We consider the following set of prime $\rho$-exceptional divisors, 
\[ \mathcal{S} = \{ E \ |\  \rho(E)=C \mbox{ is a curve such that } L\cdot C >0 \}.\] 
Let $F = \sum_{E\in \mathcal{S}} E$. 
By running a $\rho$-relative MMP for $(Z, \Gamma + \eta F)$, for some  $0<\eta\ll 1$, 
we obtain a  $\mathbb{Q}$-factorial bimeromorphic model $\rho' \colon Z' \to X$ such that the divisors contracted by $\varphi: Z\bir Z'$ are exactly  those of $\mathcal{S}$. 
We note that, if $C\subseteq X_{\sing}$ such that $L\cdot C > 0$, then $\rho'$ is an isomorphism over the general points of $C$.

We claim that if $C' \subseteq Z'_{\sing}$ is a curve, then $\rho'$ is an isomorphism around the general points of $C'$ and $L\cdot C > 0$, where $C = \rho'(C')$. 
Indeed, since $Z'$ is singular along the curve  $C'$, it cannot have terminal singularities around the general  points of $C'$. 
In particular, there is an exceptional divisor $E$ over $Z'$ with $\Center_{Z'}(E)=C'$ such that the discrepancy $a(E, Z')\< 0$. This implies that $a(E, X)=a(E, Z', \Gamma')\<a(E, Z')\< 0$, where $\Gamma':=\varphi_*\Gamma$. 
Consequently, $E_Z:=\Center_Z(E)\subset Z$ is a divisor. 
We note that $E_Z$ is contracted by $\varphi:Z\bir Z'$, as $\Center_{Z'}(E)=C'$. 
Therefore $E_Z\in \mathcal{S}$; in particular, $\Center_X(E)= C\subset X_{\sing}$ is a curve such that $L\cdot C >0$. 
Hence, $\rho'$ is an isomorphism around the general points of $C'$. 
\\

%Let $Z'' \to Z'$ be a partial resolution of $Z'$ which resolves exactly the isolated singularities of $Z'$.  
%Let $\rho''\colon Z'' \to X$ be the composite morphism and  $K_{Z''}+ \Gamma''= \rho''^*(K_X+\Delta)$.
%We note that the singular locus of $Z''$ is of pure codimension $2$. 
%Moreover, if $C''$ is a component of $Z''_\textsubscript{sing}$, then $\rho''$ is an isomorphism around general points of $C''$, and $\rho''(C'')=C$ is a curve contained in %$X_\textsubscript{sing}$ such that $L\cdot C >0$. 

We set $\Theta = (\rho')^{-1}_*\Delta$.  
Let $z\in Z'$ be any point and $U \subseteq Z'$ a small enough analytic open neighborhood of $z$. 
% Note that $\Gamma'|_U$ is effective.
Let $\gamma\colon V\to U$ be the index-one cover of the divisor $\Theta|_U$ with Galois group $G$.  
The previous paragraph implies that if $C'\subseteq Z'_{\sing}$ is a curve, then $\rho'$ is an isomorphism around the general points of $C'$.  
By the description on the singularities of $X$ along $\rho'(C')$ in Lemma \ref{lemma:singularities-are-cyclic-1}, 
and by using the universal property of index-one cover in Lemma \ref{lemma:cyclic-factor}, 
we deduce that  $V$ is smooth around the general points of $\gamma^{-1}(C')$.  
It follows that $V$ has at most isolated singularities. 
Let $\Xi$ be the pullback of $\Theta|_U$ in $V$. 
Then $(V,\Xi)$ is lc. 
Since $V$ has isolated singularities, from \cite[Proposition 16.6]{Kol92},  up to shrinking $U$, 
we may assume that the pair $(V,\Xi)$ is log smooth outside $\gamma^{-1}(z)$. 

From the compactness of $Z'$, 
%we can  then deduce that, there are only finitely many points $z\in Z'$ such that the pair $(V, \Xi)$ constructed above is not log smooth. 
%As a consequence, 
we can find finitely many open subsets $U_1,...,U_n$ of $Z'$ with $\bigcup_{i=1}^n U_i =Z'$ such that the following properties hold. 
Each $U_i$ is an open neighborhood of some point $z_i\in Z'$, as described   in the previous paragraph.  
Furthermore, if $w \in U_i\cap U_j$ for some $i\neq j$, then $w$ is not equal to any of the $z_i$'s. 
In particular,   $(V_i,\Xi_i)$ is log smooth at $\gamma_i^{-1}(w)$, where $\gamma_i \colon (V_i, \Xi_i )\to (U_i, \Theta|_{U_i})$ is the index-one cover constructed in the previous paragraph, with Galois group $G_i$.

We note that,  thanks to the universal property of cyclic cover in Lemma \ref{lemma:cyclic-factor}, the collection  $\{ (V_i \setminus \gamma_i^{-1}(z_i), G_i )\}$ induces an orbifold structure on $Z'\setminus \{z_1,...,z_n\}$.  
Let $\overline{V}_i \to V_i$ be a $G_i$-equivariant log resolution of $(V_i,\Xi_i)$,  which is an isomorphism on the log smooth locus of $(V_i,\Xi_i)$.  
Then the overlaps among the $\overline{V}_i$'s coincide with those among the $V_i$'s.  
Thus  the collection  $\{(\overline{V}_i,G_i)\}$  induces a compact complex orbifold $Y_{\orb}'$ with quotient space $Y$.  
Here we use the notation $Y_{\orb}'$ as it may not be standard.
There is a natural bimeromorphic morphism $g\colon Y  \to Z'$, as explained in (4) of Remark \ref{rmk:quotient-singularities}.
Let $\pi \colon Y\to X$    be the composite morphism.   
We note that $(Y,B)$ is lc by \cite[Proposition 5.20]{KM98}. 
We write $\overline{\Xi}_i$ for the strict transform of $\Xi_i$ in $\overline{V}_i$. 
\\

We will prove the condition (2) for $Y$.  
We remark that, on an orbifold chart $\overline{V}_i$ of $Y_{\orb'}$, the preimage of $B $ is exactly $\overline{\Xi}_i$.
Furthermore, the $G_i$ action on $\overline{V}_i$ is free in codimension one outside the exceptional locus $W_i$ of $\overline{V}_i \to V_i$. 
Since $W_i+\overline{\Xi}_i$ is snc, by Lemma \ref{lemma:orbifold-snc},  
%Hence  $\Omega_{ \overline{V}_i }^{1}(\log\, \overline{\Xi}_i)$ is locally free. 
%Since the reflexive sheaf on $\overline{V}_i$ corresponding  to  $\Omega_Y^{[1]}(\log\, B)$  is just $\Omega_{ \overline{V}_i }^{1}(\log\, \overline{\Xi}_i)$,  it follows 
We deduce  that  $\Omega_Y^{[1]}(\log\, B)$ induces an orbifold vector bundle on the standard orbifold structure of  $Y$.

We will  verify the conditions  (3) and (4). 
If $C\subseteq X_{\sing}$ is a curve with $C\cdot L >0$, then $\rho'$ is an isomorphism over general points of $C$ by the construction of $Z'$ in the first paragraph. 
We notice that $g^{-1}\colon Z' \dashrightarrow Y$ is an isomorphism outside $\{z_1,...,z_n\}$. 
Hence $\pi$ is an isomorphism over general points of $C$. This shows the condition (3). 
If $C_Y$ is a curve in  $Y_{\sing}$ with $C_Y \cdot \pi^*L >0$, then  $g$ must be an isomorphism around general points of $C_Y$. 
Thus $g(C_Y)$ is a curve contained in $Z'_{\sing}$, 
and $\rho'$ is an isomorphism around general points of $g(C_Y)$ by the second paragraph.  
This shows the condition (4). 

For (5), observe that if $\pi(P)$ is not a point, then $g(P)$ is  a divisor in $Z'$, and $\pi(P)$  is a curve   contained in $X_{\sing}$. Thus by construction of $Z'$,  we have $L\cdot \pi(P) =0$. 
Hence $\pi^*L|_P \equiv 0$. This also implies that $K_Y^2 \cdot \pi^*L = K_X^2 \cdot L$ and then (6) follows from Lemma \ref{lem:nu-2-computation-1}.
%This completes our proof.
\end{proof}

%Until the end of this section, we fix such bimeromorphoic model $Y$ of $X$. 
The following lemma is adapted from \cite[Lemma 14.3.1]{Kol92}.

\begin{lemma}
\label{lemma:comparison-c2-1}
Assume that we are in the Setup \ref{set:setup-1}. Let $\pi:Y\to X$ be the proper bimeromorphic morphism constructed in Lemma \ref{lemma:construction-of-Y-1}. Let $\rho\colon \widetilde{X} \to Y$ be a desingularization of $Y$ and $\gamma\colon \widetilde{X} \to X$ the composite morphism. 
Then $(K_{\widetilde{X}}^2+c_2(\widetilde{X})) \cdot \gamma^*L  \ge (K_{Y}^2+\hat{c}_2(Y)) \cdot \pi^*L$. 
Furthermore,  if the equality holds, then for any curve $C \subseteq X_{\sing}$, we have $L\cdot C = 0$.
\end{lemma}

\begin{proof}
Let $\ell(\widetilde{X}, \rho) = c_1^2(\widetilde{X}, \rho) + c_2(\widetilde{X}, \rho)$ be the  relative class defined in Subsection \ref{subsec:relative-class}. 
Then by Lemma \ref{lemma:c_2-difference},  there are real numbers $a_i$ and curves $C_i\subseteq Y_{\sing}$ such that 
\[ \ell(\widetilde{X}, \rho) \cdot \gamma^*L = \pi^*L\cdot \left(\sum_i a_iC_i\right).\]
We only care about the $a_i$'s which correspond to $C_i\cdot \pi^*L >0$. 
By Lemma \ref{lemma:c2-calculation} and  the property (4) of Lemma \ref{lemma:construction-of-Y-1}, we see that these $a_i$'s are computed by the list of Lemma \ref{lemma:possilbe-a_i-1}. 
The only possible case when certain $a_i$ can take a nonpositive value is the case when the index of  $C_i$ in Lemma \ref{lemma:singularities-are-cyclic-1} is $6$ (see Part (4) of Lemma \ref{lemma:possilbe-a_i-1}). 
If that happens, then by the same lemma there are two unique curves $C_j$ and $C_k$ with indices $2$ and $3$ respectively such that 
\[ \pi^*L\cdot C_i=\pi^*L\cdot C_j=\pi^*L\cdot C_k>0. \] 
Hence $(a_iC_i+a_jC_j+a_kC_k)\cdot \pi^*L  >0$, as $a_i=-\frac 56$, $a_j=\frac 32$ and $a_k\in\left\{\frac 43, \frac 83\right\}$. 
Since the associations $C_i \mapsto C_j$ and $C_i\mapsto C_k$, from a curve of index 6 to a curve of index 2 and a curve of index 3, are one-to-one, 
after all, we always have  $\ell(\widetilde{X}, \rho) \cdot \gamma^*L \> 0$.  

We assume that the equality holds. 
Then we have  $\ell(\widetilde X, \rho)\cdot\gamma^*L=0$. 
From the previous paragraph, this implies that  $\pi^*L\cdot C_Y=0$ for every curve $C_Y\subset Y_{\sing}$. 
We can then deduce the second part of the lemma from  the property (3) of Lemma \ref{lemma:construction-of-Y-1}.
\end{proof}

Now we can deduce Proposition \ref{prop:Euler-characteristic-combined-1}.

\begin{proof}
[{Proof of Proposition \ref{prop:Euler-characteristic-combined-1}}]
When $X$ is projective it follows from \cite[Equation (14.4.1.2)]{Kol92}, so we may assume that $X$ is non-algebraic. Let $\pi:Y\to X$ be the proper bimeromorphic morphism constructed in Lemma \ref{lemma:construction-of-Y-1}. 
Let $\rho:\widetilde{X} \to Y$ be a desingularization and $\gamma\colon \widetilde{X} \to X$ be the composite morphism.  
Since $X$ is $\mathbb{Q}$-factorial, the conditions (1) and (2) of Setup \ref{set:setup-1} imply that  $X$ has klt singularities. 
Thus  $X$ has rational singularities, and it follows that 
\[\chi(X, \mathscr{O}_{{X}}(n L)) = \chi(\widetilde{X}, \mathscr{O}_{\widetilde{X}}(n\gamma^*L) )\]
for all $n \> 0$.  
We have  
\begin{eqnarray*}
    \chi(\widetilde{X}, \mathscr{O}_{\widetilde{X}}(n\gamma^*L) ) &= & \frac{n^3}{6}(\gamma^*L)^3-\frac{n^2}{4} (K_{\widetilde X}\cdot(\gamma^*L)^2)\\ 
     &&  +\frac{n}{12}(K_{\widetilde{X}}^2+c_2(\widetilde{X})) \cdot \gamma^*L  + \chi(\widetilde{X}, \mathscr{O}_{\widetilde{X}})\\
     &=& \frac{n}{12}(K_{\widetilde{X}}^2+c_2(\widetilde{X})) \cdot \gamma^*L + \chi(\widetilde{X}, \mathscr{O}_{\widetilde{X}}),
\end{eqnarray*}
where second equality follows from that $\nu(X, L)=2$ and that $K_X\cdot L^2=0$, see Lemma \ref{lem:nu-2-computation-1}.
By Lemma \ref{lemma:comparison-c2-1}, we have 
\[(K_{\widetilde{X}}^2+c_2(\widetilde{X}))\cdot \gamma^*L  \> (K_{Y}^2+\hat{c}_2(Y)) \cdot \pi^*L = \hat{c}_2(Y) \cdot \pi^*L,\]
where the last equality follows from the property (6)  in  Lemma \ref{lemma:construction-of-Y-1}.
Thus it is enough to show that $\hat{c}_2(Y) \cdot \pi^*L \> 0$.

By Lemma \ref{lemma:compare-log-c_2}, we have 
\[
\hat{c}_2(Y)  \cdot \pi^*L \> 
\hat{c}_2(\Omega_Y^{[1]}(\mathrm{log}\, B)) \cdot \pi^*L - (K_Y+B)\cdot B\cdot \pi^*L,
\]
where $B:=\pi^{-1}_*\Delta$ is defined in Lemma \ref{lemma:construction-of-Y-1}. 
From (5) of Lemma \ref{lemma:construction-of-Y-1},  we  know that if $P\subset Y$ is a $\pi$-exceptional divisor, then $\pi^*L|_P\equiv 0$, and thus 
\[
k(K_Y+B)\cdot B\cdot \pi^*L = \pi^*L \cdot B \cdot \pi^*L= \Delta \cdot L^2 = 0,
\]
where the last equality follows from the fact that $0=L^3=L^2\cdot kD$ and that $D_{\red}=\Delta$. 

If $X$ is not uniruled, then  $\hat{c}_2(\Omega_Y^{[1]}(\mathrm{log}\, B)) \cdot \pi^*L \> 0$ by Proposition \ref{prop:psef-c2}. 
If $X$ is uniruled, then we note that  $K_Y+(1-\epsilon) B$ is pseudoeffective for all $0<\epsilon\ll 1$ by Theorem \ref{thm-non-vanishing}.  
Hence by Proposition \ref{prop-log-orbifold-c2-semipositive}, we still have $\hat{c}_2(\Omega_Y^{[1]}(\mathrm{log}\, B)) \cdot \pi^*L \> 0$.  
In conclusion, we obtain that $\hat{c}_2(Y)  \cdot \pi^*L \> 0$. 

For the second part of the proposition, we assume that  there is a curve $C\subseteq X_{\sing}$ with $L\cdot C >0$. 
Then from Lemma \ref{lemma:comparison-c2-1}, we see that 
\[(K_{\widetilde{X}}^2+c_2(\widetilde{X}))\cdot \gamma^*L  > (K_{Y}^2+\hat{c}_2(Y)) \cdot \pi^*L,\]
and the RHS is nonnegative from the previous discussion. 
Hence $\chi(X, \mathscr{O}_{{X}}(n L))$ grows at least linearly in $n$. 
This completes our proof.
\end{proof}

\section{Construction of bimeromorphic models}

In this section, we will construct two bimeromorphic models, which will serve the proof of Theorem \ref{thm:lc-log-abundance}. 
The first one is an analogue of \cite[Lemma 14.2]{Kol92}.

\begin{lemma}
\label{lemma:reduction-2.1}
Let $(X,\Delta)$ be a dlt pair such that $X$ is a $\mathbb{Q}$-factorial compact K\"ahler threefold, and that $K_X+\Delta$ is nef of numerical dimension $2$. 
Assume that  $K_X+(1-\lambda)\Delta$ is pseudoeffective for some $0<\lambda< 1$. 
Then there is a bimeromorphic model $X'$ of $X$ and a reduced boundary $\Delta'$ on $X'$ such that the following properties hold.
\begin{enumerate}
 \item $(X',\D')$  satisfies the conditions of Setup \ref{set:setup-1}. 
%\item $X'$ is $\mathbb{Q}$-factorial.
%\item $(X',\Delta')$ is a reduced lc pair.
%\item $X'$ has terminal singularities outside $\Delta'$.
%\item $K_{X'}+\Delta'$ is nef.
%\item There is an effective $\mbQ-divisor$ $0\<D' \sim_\mathbb{Q} K_{X'}+\Delta'$ such that  $D'_{\red}= \Delta'$. 
%\item $\nu(K_X+\Delta) = \nu(K_{X'}+\Delta').$
\item $\kappa(X, K_X+\Delta) = \kappa(X', K_{X'}+\Delta').$
\item $K_{X'}+(1-\epsilon)\Delta'$ is nef for all $0<\epsilon\ll 1$.
%\item  If $C\subseteq X'$ is a curve with $(K_{X'}+\Delta')\cdot C'>0$, then $(X',\Delta')$ is klt around general points of $C$. 
\end{enumerate}
\end{lemma}

\begin{proof} 
By taking a dlt modification as in \cite[Corollary 1.30]{DasOu2022}, 
we may first assume that $X$ has terminal singularities. 
By running a $(K_X+(1-\lambda)\Delta)$-MMP, trivial with respect to $K_X+\Delta$, we obtain a lc pair $(X_1,\D_1)$ such that $\D_1$ is the strict transform of $\D$ and that  $K_{X_1}+(1-\epsilon)\Delta_1$ is nef for all $0<\epsilon \ll 1$ (see Lemma  \ref{lem:special-mmp}). 
Observe that every step of this MMP is in fact $K_X$-negative, and hence  $X_1$ has $\mbQ$-factorial terminal singularities. Moreover, $(X_1,\Delta_1)$ is lc, and thus $(X_1, t\Delta_1)$ is klt for any $0\<t<1$. 
We claim that, for all rational number $0<\epsilon\ll 1$, the following equalities hold:
\begin{eqnarray*}
   \nu(X_1, K_{X_1}+(1-\epsilon)\Delta_1) &=& \nu(X_1, K_{X_1}+\Delta_1)\ = \ \nu(X, K_X+\Delta) \ =\ 2, \\
%    & \mbox{and} & \\
  \kappa (X_1, K_{X_1}+(1-\epsilon)\Delta_1) &=& \kappa (X_1, K_{X_1}+\Delta_1) \  = \  \kappa(X, K_X+\Delta). 
\end{eqnarray*}
To see this, fix $0<\eps\ll 1$ so that $K_{X_1}+(1-2\eps)\Delta_1$ is nef. 
Then by Theorem \ref{thm-non-vanishing-general-setting}, 
there is an effective $\mbQ$-divisor $D_1\sim_{\mbQ} K_{X_1}+(1-2\eps)\Delta_1$. 
Then clearly $D_1 +\eps\Delta_1\sim_{\mbQ} K_{X_1}+(1-\eps)\Delta_1$ and $D_1+2\eps \Delta_1\sim_{\mbQ} K_{X_1}+\Delta_1$. 
Since $\Supp\, (D_1+\eps\Delta_1) = \Supp\, (D_1+2\eps \Delta_1)$, our claim follows from \cite[Lemma 11.3.3]{Kol92} (by using K\"ahler classes in place of ample divisors).\\

Now we fix a rational number $0<\mu\ll 1$ so that $K_{X_1}+(1-2\mu)\Delta_1$ is nef, and hence $K_{X_1}+(1-\mu)\Delta_1$ is nef. 
By Theorem \ref{thm-non-vanishing-general-setting}, there is an effective $\mbQ$-divisor $D_1\sim_{\mbQ} K_{X_1}+(1-2\mu)\Delta_1$. 
Then $(X_1, (1-\mu)\Delta_1+\eta D_1)$ is klt for some rational number $0<\eta\ll 1$. 
Moreover, 
\[(1+\eta)D_1+\mu\Delta_1\sim_{\mbQ} K_{X_1}+(1-\mu)\Delta_1+\eta D_1\] 
 and $D_1+2\mu\Delta_1\sim_{\mbQ} K_{X_1}+\Delta_1$. 
Thus by a similar argument as above, we see that 
\[\nu(X, K_X+\Delta)=\nu(X_1, K_{X_1}+(1-\mu)\Delta_1)=\nu(X_1, K_{X_1}+(1-\mu)\Delta_1+\eta D_1), \] 
\[\kappa(X, K_X+\Delta)=\kappa(X_1, K_{X_1}+(1-\mu)\Delta_1)=\kappa(X_1, K_{X_1}+(1-\mu)\Delta_1+\eta D_1).\]
Hence replacing $(X, \Delta)$ by $(X_1, (1-\mu)\Delta_1+\eta D_1)$,  we may assume that  $(X, \Delta)$ is a klt pair and there is an effective nef $\mbQ$-Cartier divisor $D\sim_{\mbQ} K_X+\Delta$ such that $D_{\red}=\Delta_{\red}$. 
\\

Let $r\colon Z\to X$ be a log resolution of $(X,\Delta)$. 
Let $\Gamma$ be the reduced sum of all $r$-exceptional divisors and the strict transform of the components of $\Delta$. 
Then $E = K_Z+\Gamma - r^*(K_X+\Delta) \> 0$, and its support contains the $r$-exceptional locus. 
Hence  $K_Z+\Gamma \sim_\mathbb{Q} D_Z := r^*D + E$, with $(D_Z)_{\red} = \Gamma$.

We run a $(K_Z+\Gamma)$-MMP and end with a minimal model $(Z, \Gamma)\to (Z', \Gamma')$ such that $(Z',\Gamma')$ is dlt, $K_{Z'}+\Gamma'\sim_{\mbQ} D_{Z'}$ is nef,    $D_{Z'}$ is the strict transform of $D_Z$,   
% Moreover, if $D_{Z'}$ is the strict transform of $D_Z$ in $Z'$, then  we see that  $K_{Z'}+\Gamma' \sim_\mathbb{Q} D_{Z'}$ 
and $(D_{Z'})_{\red} = \Gamma'$. 
Furthermore, we see that   $\kappa(X, K_X+\Delta) = \kappa(Z', K_{Z'}+\Gamma')$, and that $\nu(X, K_X+\Delta) = \nu(Z', K_{Z'}+\Gamma')$. 
Indeed,  pulling back the   divisors $D_{Z'}$ and $D$ to a common resolution of $X\bir Z'$, 
we see that the two effective nef divisors have the same support, 
and thus the arguments of \cite[Lemma 11.3.3]{Kol92} apply.  
Therefore, $(Z',\Gamma')$ satisfies the condition of Setup \ref{set:setup-1}. 
\\

Now we run a $(K_{Z'}+(1-\lambda)\Gamma')$-MMP, trivial with respect to $K_{Z'}+\Gamma'$ (see Lemma  \ref{lem:special-mmp}) and end with a bimeromorphic map $Z'\bir X'$ such that $K_{X'}+(1-\epsilon)\Delta'$ is nef for any $\epsilon>0$ small enough,  where $\Delta'$ is the pushforward of $\Gamma'$ onto $X'$. 
Moreover, the numerical dimension and the Kodaira dimension of $K_{X'}+\D'$ are the same as those of $K_X+\D$. 
We remark that this  MMP is also a $K_{Z'}$-MMP, trivial with respect to $K_{Z'}+\Gamma'$.
Thanks to Lemma \ref{lemma:contraction-dlt} and Lemma \ref{lem:flip-invariance}, we deduce that $(X',\D')$ also satisfies the conditions of  Setup \ref{set:setup-1}. 
This completes our proof of the lemma. 
\end{proof}

\begin{lemma}
\label{lemma:reduction-2.2}
Let $(X,\Delta)$ be a lc pair such that $X$ is a $\mathbb{Q}$-factorial non algebric compact K\"ahler threefold, and that $K_X+\Delta$ is nef. 
Assume that  $K_X+(1-\lambda)\Delta$ is not pseudoeffective for any $\lambda>0$.
Then there is a bimeromorphic model $g\colon X \dashrightarrow X'$     such that the following properties hold. 
\begin{enumerate}
    %\item $X'$ has $\mathbb{Q}$-factorial terminal singularities.
    \item $(X',  \Delta')$ is a lc  pair for some divisor $\Delta'$. 
    \item $K_{X'} + \Delta'$  is nef.     
    \item $\nu(X, K_X+\Delta) = \nu(X', K_{X'}+\Delta').$
    \item $\kappa(X, K_X+\Delta) = \kappa(X', K_{X'}+\Delta').$
    \item There is a Mori fiber space $f\colon X'\to Y$ over a compact K\"ahler surface $Y$ 
          such that $K_{X'}+\Delta'\sim_{f,\mbQ} 0$. 
    \item  $\lfloor  \Delta'_{\ver} \rfloor = 0$, where $\Delta'_{\ver}$ is the vertical part of $\Delta'$ over $Y$. 
    \item  If  $\lfloor  \Delta'_{\hor} \rfloor = 0$, then $(X', \D')$ is klt, where $\Delta'_{\hor}$ is the horizontal part of $\Delta'$ over $Y$. 
\end{enumerate} 
\end{lemma}

\begin{proof} 
By taking a dlt modification as in \cite[Theorem 1.29]{DasOu2022}, we may assume that $(X,\D)$ is dlt. 
The assumption implies that  $X$ is uniruled. 
Since $X$ is not algebraic, the base  $Z$ of the MRC fibration $X\bir Z$ has dimension two.
Let $\Delta_{\hor}$ be the horizontal part of $\Delta$ over $Z$. 
Then $K_X+\Delta_{\hor}$ is pseudoeffective by Theorem \ref{thm-non-vanishing}. 
Let $\Delta_{\ver}:=\Delta - \Delta_{\hor}$ be the vertical part of $\Delta$ over $Z$.
We may run a $(K_X+\Delta_{\hor})$-MMP, trivial with respect to $K_X+\Delta$, as in Lemma \ref{lem:special-mmp}. 
Then we obtain a  bimeromorphic model $(X_1,\Delta_1)$ such that $\D_1$ is the strict transform of $\D$ and that
$K_{X_1}+ (\Delta_1)_{\hor} + (1-\epsilon)(\Delta_1)_{\ver}$ is nef for all $\epsilon>0$ small enough. 
As in the first paragraph of the proof of Lemma \ref{lemma:reduction-2.1}, 
by the same argument of \cite[Lemma 11.3.3]{Kol92}, we obtain that  for all $\epsilon >0$ small enough, 
\[
\nu(X_1, K_{X_1}+ (\Delta_1)_{\hor} + (1-\epsilon)(\Delta_1)_{\ver}) = \nu(X_1, K_{X_1}+\Delta_1) = \nu(X, K_X+\Delta),
\]
%\begin{center}
 %   and 
%\end{center} 
\[\kappa(X_1, K_{X_1}+ (\Delta_1)_{\hor} + (1-\epsilon)(\Delta_1)_{\ver}) = \kappa(X_1, K_{X_1}+\Delta_1) = \kappa(X, K_X+\Delta).\]
We remark that $(X_1, (\Delta_1)_{\hor})$ is dlt and $(X_1,  \Delta_1)$  is lc; in particular, for some rational number $0<\epsilon\ll 1$,  $({X_1}, (\Delta_1)_{\hor} + (1-\epsilon)(\Delta_1)_{\ver})$ is dlt.
Replacing $(X,\Delta)$ by $(X_1, (\Delta_1)_{\hor} + (1-\epsilon)(\Delta_1)_{\ver})$, we may assume that 
$\lfloor \Delta_{\ver} \rfloor = 0.$   
This will imply the property (6). 
\\

We run a $K_X$-MMP, trivial with respect to $K_X+\D$, as in Lemma \ref{lem:special-mmp}. 
Since $K_X+(1-\lambda)\D$ is not pseudoeffective for any $\lambda >0$, this MMP  terminates with  a Mori fiber space $X'\to Y$, where $Y$ is a compact K\"ahler surface. 
Moreover, if  $\D'$ is the strict transform of $\D$, then    $K_{X'}+\Delta'\sim_{f,\mbQ} 0$. 
Since $ X \dashrightarrow X'$ is  $(K_{{X}}+{\D})$-trivial, we obtain that $(X', \D')$ is lc. 
Furthermore,   if $\lfloor \Delta'_{\hor} \rfloor =0$, then $ \lfloor \Delta  \rfloor =0$ as  $\lfloor \Delta_{\ver} \rfloor = 0.$  
Therefore $(X,\D)$ is klt in this case, and so is $(X',\D')$. 
This proves the property (7), and completes the proof of the lemma. 
\end{proof}

\section{Proof of the main theorem}
In this section we will prove the results stated in the introduction.  
We will divide the proof Theorem \ref{thm:lc-log-abundance} into two cases,  depending on whether $K_X+(1-\lambda)\Delta$ is pseudo-effective for some $\lambda>0$ or not.
In the first case we adapt the methods of \cite[Chapter 14]{Kol92} and \cite[Section 8.B]{CHP16}. 

\begin{lemma}
    \label{lemma:main-thm-1}
 Theorem \ref{thm:lc-log-abundance} holds if  $K_X+(1-\lambda)\Delta$ is pseudo-effective for some $0<\lambda \< 1$. 
\end{lemma}

\begin{proof} 
Passing to a dlt model as in \cite[Theorem 1.29]{DasOu2022} we may assume that $(X,\Delta)$ is a $\mbQ$-factorial dlt pair. 
By \cite[Theorem 1.35]{DasOu2022}, it is enough to show that $\kappa(K_X+\Delta)\> 1.$  
The proof is split into several steps. 
In the first two  we will replace $(X,\D)$ by some bimeromorphic model while keeping $\kappa(X,K_X+\Delta)$ and 
$\nu(X,K_X+\Delta)$ invariant.\\

\noindent
\textit{Step 1.} 
Replacing $(X,\Delta)$ with the bimeromorphic model obtained in Lemma \ref{lemma:reduction-2.1}, we may assume that the following hold:
\begin{enumerate}
\item $X$ is $\mathbb{Q}$-factorial, 
\item $(X,\Delta)$ is a reduced lc  pair and $X$ has terminal singularities outside  $\Supp\Delta$,  
\item $L = k(K_X+\Delta)$ is a nef Cartier divisor for some integer $k>0$,   
      and  $L \sim D \> 0$ with $D_{\red} = \Delta$.
\item $\nu(X, K_X+\Delta) = 2$,
\item $K_X+(1-\eps)\Delta$ is nef for all $0<\eps\ll 1$, 
%\item If $C\subset X$ is a curve such that $(K_X+\Delta)\cdot C>0$, then $(X, \Delta)$ is dlt near the general points of $C$, and
\end{enumerate} 
From Lemma \ref{lem:nu-2-computation-1} and Proposition \ref{prop:Euler-characteristic-combined-1}, we obtain the following properties:
\begin{enumerate}
    \item[(6)] $(K_X+\Delta)^2 \cdot \Delta= L^2 \cdot \Delta = 0$.
    \item[(7)] there is a constant $\eta\in\mbQ$ such that
\[
\chi(X,\mcO_X({nL})) \> \eta 
\] 
for all $n\> 0$.
Furthermore, if there is a curve $C\subseteq X_{\sing}$ such that $L\cdot C >0$, then $\chi(X, \mathscr{O}_{{X}}(n L) ) $ grows at least linearly in $n$. 
\end{enumerate}~\\

\noindent
\textit{Step 2.} 
In this step, we will eliminate the components of $\Delta$  on which the restrictions of $L$ are numerically trivial.
Suppose that $\Delta'$ is one such irreducible component of $\Delta$, i.e. $L|_{\Delta'}\equiv 0$. 
We set $\Delta'':=\Delta-\Delta'$. 
Fix some rational number $0<\lambda \ll 1$ such that $K_X+(1-\lambda)\D$ is nef. 
%and $D-k\lambda \D$ is effective with support equal to $\D$.  
Let $a$ be the coefficient of $\D'$ in $D$.
Then there is some rational number $0<\mu \ll \lambda$ such that  
\begin{equation}\label{eqn:theta-perturbation} 
\Theta := \lambda \Delta'' + \mu\D' - \frac{\mu}{a} D
\end{equation} 
is effective with support equal to $\D''$. 

Since $(X, \D - \lambda \D'' - \mu \D')$ is klt, we can run a $(K_X+\Delta - \lambda \D'' - \mu \D')$-MMP, trivial with respect to $K_X+\Delta$  (see Lemma  \ref{lem:special-mmp}).  
Then we  obtain a bimeromorphic model  $f\colon  X\dashrightarrow Y$. 
Let $\Gamma$, $\Gamma'$ and $\Gamma''$ be the pushforward of $\Delta$, $\Delta'$ and $\Delta''$ onto $Y$, respectively. 
Then $K_Y+\Gamma - \epsilon (\lambda\Gamma''+\mu\Gamma')$ is nef for all $0<\eps\ll 1$, $K_Y+\Gamma$ is nef of numerical dimension $2$  and  $(K_Y+\Gamma)|_{\Gamma'}\num 0$. We also note that  $k(K_Y+\Gamma) \sim f_*D$ and $(f_*D)_{\red} = \Gamma$. Thus we have
\[
(- \lambda\Gamma'' - \mu\Gamma')|_{\Gamma'}\num \frac{1}{\eps}(K_Y+\Gamma-\eps(\lambda\Gamma''+\mu\Gamma'))|_{\Gamma'}
\]
which is nef. 
However, we see from \eqref{eqn:theta-perturbation} that 
\[
- \lambda\Gamma'' - \mu\Gamma' = - f_*\Theta - \frac{\mu}{a} f_*D.
\]
Thus $(- \lambda\Gamma'' - \mu\Gamma')|_{\Gamma'} \equiv - (f_*\Theta)|_{\Gamma'}$, as $(f_*D)|_{\Gamma'} \equiv 0$ by our construction above.  The only way the effective divisor $f_*\Theta|_{\Gamma'}$ can be anti-nef is that $f_*\Theta|_{\Gamma'}=0$. 
In particular, $\Supp(f_*\Theta)\cap\Gamma'=\emptyset$, as $f_*\Theta$ is a $\mbQ$-Cartier divisor on $Y$. Since $\Supp(f_*\Theta)=\Gamma''$ by our construction, we have $\Gamma''\cap\Gamma'= \emptyset$. 
% Since the support of $f_*\Theta$ is exactly $\Gamma''$, and $(f_*\Theta)|_{\Gamma'}$ is an effective   we conclude that $\Gamma'' \cap \Gamma' = \emptyset.$

On the other hand, since $(Y,\Gamma - \lambda\Gamma'-\mu\Gamma'')$ is klt, $Y$ has rational singularities. 
Since $f_*D$ is nef of numerical dimension $2$, from \cite[Lemma 6.7]{CHP16} it follows that $\Supp(f_*D)$ is connected. 
Since $(f_*D)_{\red}=\Gamma$ by our construction, we deduce  that
  \[\Gamma' = 0  \mbox{ and }  \Gamma = \Gamma''.\]

We  remark that   $(Y,\Gamma)$ 
satisfies the conditions (1)-(6) of Step 1.  
It also satisfies the first part of the condition (7), as 
\[
\chi(X,\mcO_X({nk(K_X+\D)})) = \chi(Y,\mcO_Y({nk(K_Y+\Gamma)}))\mbox{ for all } n\>1.  
\]
We claim that the only divisor in $X$ contracted by $f$ is $\D'$. 
Indeed, assume that $E$ is a prime divisor contracted by $f$.  
Let $W$ be a resolution of singularities of the graph of $f:X\bir Y$, 
and let $p\colon W\to X$ and $q\colon W\to Y$ be the projection morphisms. 
Then we have   
\[
p^*(K_X+(1-\lambda)\D) = q^*(K_Y+(1-\lambda)\Gamma) - E_1
\]
and 
\[
p^*(K_X+\Delta - \lambda \D'' - \mu \D') = q^*(K_Y+\Gamma - \lambda \Gamma'') + E_2,
\]
where $E_1$ and $E_2$ are $q$-exceptional divisors. 
From \cite[Lemma 3.38]{KM98}, whose proof still works in the setting of complex  analytic varieties, 
we see that $E_2$ is effective and $\Supp(p_*E_2)$ contains all the divisors contracted by $f:X\bir Y$.    
In particular, $E\subset \Supp(p_*E_2)$. 
Moreover, since $K_X+(1-\lambda)\D$ is nef, from the negativity lemma (see \cite[Lemma 3.39]{KM98}), 
it follows that $E_1$ is also effective. 
% Moreover,  $E_2$ is effective 
% % and its  support containis $p^{-1}_*E$, 
% for $f$ is a $ (K_X+\Delta - \lambda \D'' - \mu \D')$-MMP.  
Subtracting the above two equations, we get 
\[
(\lambda-\mu)p^*\D'  = E_1 + E_2.  
\]
By applying $p_*$ to the above equation and by using the fact that $\Delta'$ is a prime divisor, we deduce that $E=\D'$. 
This proves our claim. 

Now, since $L|_{\D'} \equiv 0$, it follows  that $(K_Y+\Gamma)|_{q(p^{-1}_*{\D'})} \equiv 0$.  
Combining this with the argument of Lemma \ref{lem:flip-invariance}, we see that, 
if $C_Y \subseteq Y_{\sing}$ is a curve with $(K_Y+\Gamma)\cdot C_Y>0$, then $f$ is an isomorphism near the general points of $C_Y$. 
Hence, the second part of the condition (7) also holds for $(Y,\Gamma)$. 

Repeating this procedure for every component  of $\Delta$  on which the restriction  of $L$ is numerically trivial,  
we may assume that $X$ satisfies  all the conditions of Step 1, as well as  the following one:  
\begin{enumerate}
    \item[(8)] $L|_{\Delta_i} \not\equiv 0$ for any irreducible component $\Delta_i$ of $\Delta$.
\end{enumerate}
In the following steps,  we will argue as in Step 5 of the proof of \cite[Theorem 8.2]{CHP16}. \\

\noindent
\textit{Step 3.}  
In this step, we reduce the problem to the case when $\chi(X,\mcO_X(nL))$ is a constant, independent of $n$. 
We  remark that  the conditions (1) and (2) imply that $X$ has klt singularities. 
Let $S\subset X$ be a component of $\Delta$.  
Then $S$ is Cohen-Macaulay by \cite[Corollary 5.25]{KM98}. 
We note that in \cite[Corollary 5.25]{KM98}, $X$ is assumed to be quasi-projective. 
Nevertheless, the proof still works in our setting. 
Indeed, being Cohen-Macaulay is a local property. 
Moreover,   the only step in \cite[Corollary 5.25]{KM98} which uses the projectivity assumption is to construct a ramified cyclic cover with a branched locus in general position. 
Such a construction is still valid locally.   

Thus \cite[Lemma 8.1]{CHP16} is applicable here and we have  
\begin{equation}  \label{equation:main-thm-1-1}
    h^2(\Delta,\mcO_{\Delta}(nL|_{\Delta})) = 0
\end{equation} 
for all $n\gg 0$.
Note that, by abuse of notation, here we  also use $\D$ to mean the reduced analytic space structure of the support of $\D$.
By \cite[Corollary 8.1]{CHP16}, there are two constants $C_1$ and $C_2$ such that, for all $n \gg 0,$
we have 
\begin{equation} \label{equation:main-thm-1-2}
    h^2(X,\mcO_{X}(nL )) = C_1,
\end{equation}
\begin{equation}  \label{equation:main-thm-1-3}
    h^2(X,\mcO_{X}(nL - \Delta)) = C_2.
\end{equation} 

We  recall that 
\begin{eqnarray*} 
h^0(X, \mcO_{X}(nL)) 
&=& \chi(X,  \mcO_{X}(nL)) + h^1(X, \mcO_{X}(nL )) \\
& &    - h^2(X, \mcO_{X}(nL )) + h^3(X, \mcO_{X}(nL )).
\end{eqnarray*} 
Thus, from equation \eqref{equation:main-thm-1-2} above, we deduce that 
\begin{equation}\label{equation:Euler-estimate}
  h^0(X, \mcO_{X}(nL))  \ge    \chi(X,  \mcO_{X}(nL)) + h^1(X, \mcO_{X}(nL ))  - C_1
\end{equation}
for $n \gg 0$.
Therefore, if  $ \chi(X,\mcO_X({nL}))$ grows at least linearly in $n$, then  we get that $h^0(X,\mcO_{X}(nL ))$ also grows at least linearly in $n$, which implies that $\kappa(K_X+\Delta) \> 1$. 
Hence, by the condition (7), 
we may assume that $\chi(X,\mcO_X({nL})) $ is a constant independent of $n$ for the rest of the proof. 
This implies that 
\begin{enumerate}
    \item[(9)] For every curve $C\subseteq X_{\sing}$, we have $L\cdot C =0$.  
\end{enumerate}
Our  goal is to show that $h^1(X,\mcO_X(nL))$ grows at least linearly in $n$.\\

\textit{Step 4.} In this step, we show  that $h^1(\Delta,  \mcO_{\Delta}(nL|_{\Delta}))$ grows at least linearly in $n$. 
To see this, let $(K_\D + \Xi)\sim_{\mbQ} (K_X+\Delta)|_\Delta$ be defined by adjunction. Then from \cite[Remark 1.2(5)]{Fuj00} it follows that $(\Delta, \Xi)$ has slc singularities, in particular, $S$ is demi-normal (see \cite[Definition 5.1]{Kol13}). Thus 
by \cite[Theorem 3.1]{LiuRollenske2016}, we have 
\begin{equation}\label{eqn:rr-on-boundary}
\chi(\Delta,  \mcO_{\Delta}(nL|_{\Delta})) = \chi(\Delta, \mcO_{\Delta}) +  \frac{1}2 (nL|_{\Delta}) \cdot (nL|_{\Delta} - K_{\Delta}). 
\end{equation}
% By adjunction, there is a boundary $\Xi$ such that 
% \[
% K_\D + \Xi  = (K_X+\Delta)|_\Delta.
% \]
Moreover, from \cite[Proposition 16.6]{Kol92}, we see that $\Supp(\Xi)$ is contained in $X_{\sing}$. 
Hence the condition (9) above implies that 
\[
K_\D \cdot L|_\D  = (K_\D + \Xi)\cdot L|_\D  = (K_X+\Delta)|_\Delta \cdot L|_\D  =  (K_X+\Delta) \cdot L \cdot \D=0, 
\]
where the last equality follows from the condition (6). 
By the  condition (6) again, we deduce that   $(nL|_{\Delta}) \cdot (nL|_{\Delta} - K_{\Delta})=0$. 
Hence, from equation \eqref{eqn:rr-on-boundary} it follows that $\chi(\Delta, \mcO_{\Delta}(nL|_{\Delta}))$ is a constant independent of $n$.  Now we claim that $L|_{\Delta}=k(K_\Delta+\Xi)$ semi-ample. To see this observe that by  the  condition  (6) we have $\nu(\Delta, L|_{\Delta})\<1$. Thus from  the  condition  (8) and \cite[Proposition 5.2]{CHP16} it follows $L|_{\Delta}$ is semi-ample. 
Again, from  the  condition  (8) 
we have $\mcO_{\Delta}(nL|_{\Delta})\not\cong\mcO_{\Delta}$ for any $n\in\mbZ^+$; in particular,  $h^0(\Delta,  \mcO_{\Delta}(nL|_{\Delta}))$ grows  at least linearly in $n$. 
Thanks to   equation (\ref{equation:main-thm-1-1}),  
we deduce %from equation \eqref{eqn:rr-on-boundary}  
that $h^1(\Delta,  \mcO_{\Delta}(nL|_{\Delta}))$ grows at least linearly in $n$.\\

\textit{Step 5.} We complete the proof in this step.  
Consider the exact sequence 
\[
0 \to \mcO_X(nL-\Delta) \to \mcO_X(nL) \to \mcO_{\Delta}(nL|_\Delta) \to 0.
\]
From the long exact sequence of cohomology we deduce that 
\begin{equation}\label{eqn:linear-growth-for-L}
h^1(X,\mcO_{X}(nL)) \> h^1(\Delta, \mcO_{\Delta}(nL|_\Delta)) -  h^2(X,\mcO_{X}(nL-\Delta)).
\end{equation}
Now since $h^1(\Delta, \mcO_\Delta(nL|_\Delta))$ grows at least linearly in $n$ and $h^2(X, \mcO_X(nL-\Delta))$ is a constant (see equation \eqref{equation:main-thm-1-3}), it follows that $h^1(X, \mcO_X(nL))$ grows at least linearly in $n$. Combining this along with the condition (7) and equation \eqref{equation:Euler-estimate}, we see that $h^0(X, \mcO_X(nL))$ grows at least linearly in $n$, and hence $\kappa(X, K_X+\Delta)\>1$. This completes our proof. 
\end{proof}

We treat the second case in the following lemma.

\begin{lemma}
    \label{lemma:main-thm-2}
 Theorem \ref{thm:lc-log-abundance} holds if  $K_X+(1-\lambda)\Delta$ is not pseudoeffective for any  $0<\lambda \< 1$. 
\end{lemma}

\begin{proof}   
We note that $X$ is uniruled in this case.  
By log abundance for projective  threefolds, we may assume that $X$ is not algebraic. 
Replacing $(X,\Delta)$ with the bimeromorphic model of Lemma \ref{lemma:reduction-2.2}, we may assume that 
there is a Mori fiber space $f\colon X\to Y$ of relative dimension $1$ such that $K_X+\Delta$ is $f$-numerically trivial.  
Then by \cite[Theorem 2.5]{DasOu2022} there is an effective divisor $\Gamma$ on $Y$ such that $K_X+\Delta = f^*(K_Y+\Gamma)$. 
Moreover, $(Y,\Gamma)$ is klt if $(X, \Delta)$ is.

If $\lfloor \Delta \rfloor\neq 0$, then by property (6) of Lemma \ref{lemma:reduction-2.2}  there is a component $S$ of $\lrd \Delta\rrd$ which dominates $Y$. Let $S^n\to S$ be the normalization, then $(S^n, \Delta_{S^n})$ is lc, where $K_{S^n}+\Delta_{S^n}\sim_{\mbQ}(K_X+\Delta)|_{S^n}$ is defined by adjunction. 
Then by log abundance for K\"ahler surfaces (see \cite[Theorem 1.33]{DasOu2022}), 
we deduce that $K_{S^n}+\Delta_{S^n}$ is semi-ample. Let $g:S^n\to Y$ be the induced surjective morphism, then $K_{S^n}+\Delta_{S^n}\sim_{\mbQ} g^*(K_Y+\Gamma)$. Hence $K_Y+\Gamma$ is semi-ample by \cite[Lemma 1.8]{DasOu2022}. 
It  follows that $K_X+\Delta$ is semi-ample. 

Now assume that  $\lrd \Delta\rrd=0$. 
Then the pair   $(X,\Delta)$ is klt by property (7) of Lemma \ref{lemma:reduction-2.2}. 
Thus   $(Y,\Gamma)$ is  klt as well. 
Hence again by log abundance for K\"ahler surfaces (see \cite[Theorem 1.33]{DasOu2022}), we deduce that $K_Y+\Gamma$ is semi-ample, and consequently, $K_X+\Delta$ is semi-ample.  
This completes the proof of this lemma. 
\end{proof}

\begin{proof}
[{Proof of Theorem \ref{thm:lc-log-abundance}}]
It follows from Lemma \ref{lemma:main-thm-1} and Lemma \ref{lemma:main-thm-2}.
\end{proof}

\begin{proof}
[{Proof of Corollary  \ref{cor:log-abundance}}]
It follows from  Theorem \ref{thm:lc-log-abundance} and \cite[Theorem 1.1]{DasOu2022}.
\end{proof}

\begin{proof}
[{Proof of Corollary  \ref{cor:simple-3fold}}]
By applying Corollary \ref{cor:log-abundance}, we can conclude with the same argument  of \cite[Theorem 1.2]{CHP16}. 
\end{proof}

\bibliographystyle{hep}
\bibliography{references}

\end{document}